\documentclass[11pt]{amsart}

\usepackage{graphicx}
\usepackage{tikz}
\usetikzlibrary{matrix}
\usepackage{multicol}
\usepackage{amssymb,amsthm,amsaddr}
\usepackage{bbm}
\usepackage{graphbox}
\usepackage{upgreek}
\usepackage{tcolorbox}
\usepackage{tabto}
\usepackage{braket}
\usepackage[export]{adjustbox}
\usepackage{nicematrix}
\usepackage{appendix}
\usepackage{mathtools} % for vcent colon
\usepackage{mathrsfs} % for mathscr font
\usepackage{stmaryrd} % for llbracket, rrbracket
\usepackage{dsfont} % for mathds

\usepackage[margin=1in]{geometry}
\usepackage[pagebackref, colorlinks = true, linkcolor = blue, urlcolor  = blue, citecolor = red]{hyperref}
\usepackage{xcolor}
\definecolor{darkgreen}{rgb}{0.09, 0.45, 0.27}
\definecolor{darkred}{rgb}{0.55, 0.0, 0.0}

\newcommand{\ip}[1]{\left\langle #1 \right\rangle}

\renewcommand{\epsilon}{\varepsilon}
\renewcommand{\phi}{\varphi}
\newcommand{\C}{\mathbb{C}}

\newcommand{\lnorm}{\lvert \!\vert}
\newcommand{\rnorm}{\vert \! \rvert}
\DeclareMathOperator{\Tr}{Tr}

\DeclareMathOperator{\Id}{Id}

\newcommand{\E}{\mathbb{E}}
\renewcommand{\P}{\mathbb{P}}

\newtheorem{theorem}{Theorem}[section]
\newtheorem{definition}[theorem]{Definition}
\newtheorem{proposition}[theorem]{Proposition}
\newtheorem{corollary}[theorem]{Corollary}
\newtheorem{lemma}[theorem]{Lemma}
\newtheorem{remark}[theorem]{Remark}

\newcommand{\R}{\mathbb{R}}
\newcommand{\N}{\mathbb{N}}
\newcommand{\abs}[1]{\left\lvert #1 \right\rvert}
\newcommand{\diff}{\mathop{}\!\mathrm{d}}
\DeclareMathOperator{\Cov}{Cov}
\newcommand{\defeq}{\vcentcolon=}

\DeclareMathOperator{\Crt}{Crt}
\newcommand{\mc}[1]{\mathcal{#1}}
\newcommand{\ms}[1]{\mathscr{#1}}

\newcommand{\ii}{\mathrm{i}}

\DeclareMathOperator{\OO}{O}
\DeclareMathOperator{\oo}{o}

\DeclareMathOperator{\Corr}{Corr}

\newcommand{\injnorm}[1]{\lnorm #1 \rnorm_{\mathrm{inj}}}

\begin{document}

\title[Injective norm of real and complex random tensors]{Injective norm of real and complex random tensors I: \\
from spin glasses to geometric entanglement 
}

\author{Stephane Dartois}
\address{Université Paris-Saclay, CEA, List, Palaiseau, F-91120, France}
\email{stephane.dartois@cea.fr}
\author{Benjamin M\textsuperscript{c}Kenna}
\address{Harvard University, Center of Mathematical Sciences, Cambridge, MA, United States}
\email{bmckenna@fas.harvard.edu}

\begin{abstract}
    The injective norm is a natural generalization to tensors of the operator norm of a matrix. In quantum information, the injective norm is one important measure of \emph{genuine multipartite} entanglement of quantum states, where it is known as the geometric entanglement. In this paper, we give a high-probability upper bound on the injective norm of real and complex Gaussian random tensors, corresponding to a lower bound on the geometric entanglement of random quantum states, and to a bound on the ground-state energy of a particular multispecies spherical spin glass model. For some cases of our model, previous work used $\epsilon$-net techniques to identify the correct order of magnitude; in the present work, we use the Kac--Rice formula to give a one-sided bound on the constant which we believe to be tight.
\end{abstract}

\maketitle

\noindent \emph{Date:} April 4, 2024 

\vspace*{0.05in}

\noindent \hangindent=0.2in \emph{Keywords and phrases:} Random Tensors, Injective Norm, Spectral Norm, Spin glasses, Quantum Information, Quantum Resources, Geometric Entanglement, Kac--Rice formula, Random Matrices

\vspace*{0.05in}


\vspace*{5mm}

\tableofcontents

%%%%%%%%%%%%%%%%%%%%%%%%%%%%%%%%%%%%%%%%%%%%%%%%%%%%%%%%%%%%
%%%%%%%%%%%%%%%%%%%%%%%%%%%%%%%%%%%%%%%%%%%%%%%%%%%%%%%%%%%%
%%%%%%%%
%%%%%%%%             Section: Introduction
%%%%%%%%
%%%%%%%%%%%%%%%%%%%%%%%%%%%%%%%%%%%%%%%%%%%%%%%%%%%%%%%%%%%%
%%%%%%%%%%%%%%%%%%%%%%%%%%%%%%%%%%%%%%%%%%%%%%%%%%%%%%%%%%%%

\section{Introduction}

%%%%%%%%%%%%%%%%%%%%%%%%%%%%%%%%%%%%%%%%%%%%%%%%%%%%%%%%%%%%
%%%%%%%%             Subsection: Set-up and main results
%%%%%%%%%%%%%%%%%%%%%%%%%%%%%%%%%%%%%%%%%%%%%%%%%%%%%%%%%%%%

\subsection{Set-up and main results}
In this paper, we give a one-sided bound on the leading-order behavior of a quantity with several interpretations: First, as the injective norm of a random tensor, a natural generalization of the operator norm of a matrix; second, as the geometric entanglement of the random quantum state corresponding to this tensor; third, as the ground state (maximum/minimum energy) of a pure spherical spin glass whose couplings are stored in the tensor. When the tensor is real and symmetric, the problem is essentially completely solved (in one of the two scaling limits we consider), coming from the spin-glass side. But when the tensor is complex and/or nonsymmetric, previous results only identified the correct order of magnitude, whereas identifying the correct constant is of interest in all of these interpretations (see Remarks \ref{rem:subag} and \ref{rem:entanglement-no-sym-vs-sym}). In this work, we give a one-sided bound on the correct constant, for both real nonsymmetric and complex nonsymmetric tensors, which we believe to be tight (see Remarks \ref{rem:fln22} and \ref{rem:subag}).

In notation, we will consider a tensor of order $p$, written as $T \in (K^d)^{\otimes p}$, where the underlying field $K$ is either $\R$ or $\C$, and $K^d$ is endowed with the standard $2$-norm $\lnorm \cdot \rnorm_2 \defeq \sqrt{\langle \cdot, \cdot \rangle}$ induced by the standard Hermitian product. Writing the $(d-1)$-dimensional unit sphere in $K^{d}$ as $\mathbb{S}_K^{d-1} \defeq \{x \in K^d : \lnorm x \rnorm_2 = 1\}$, we define the \textbf{injective norm}\footnote{Injective norms encompass a broader family of norms defined on tensor products of vector spaces and Banach spaces, as detailed in \cite{Grothendieck1952RsumDR,lim2021tensors}. This family is really a tensorial generalization of the family of subordinated/operator norms of matrices. The specific injective norm addressed in this context is prevalent in models of spin glasses as the ``ground state energy", in tensor PCA as a MLE, and in quantum information theory as ``geometric entanglement."} of $T$ to be the quantity 
\begin{align}
    \injnorm{T}\defeq\max_{\substack{x^{(1)},\ldots,x^{(p)}\in K^d\\ \lnorm x^{(i)}\rnorm_2=1}}\lvert\langle T, x^{(1)}\otimes \ldots \otimes x^{(p)} \rangle\rvert=\max_{x^{(1)},\ldots,x^{(p)} \in \mathbb{S}_K^{d-1}} \abs{ \sum_{i_1,\ldots,i_p = 1}^d T_{i_1,\ldots,i_p} x^{(1)}_{i_1} \cdots x^{(p)}_{i_p} },
\end{align}
which is the magnitude of the overlap between $T$ and its best rank one approximation given that $\langle T, x^{(1)}\otimes \ldots \otimes x^{(p)} \rangle$ is induced by the Hermitian product on $K^d$. Notice that when $p = 2$, this reduces to the operator norm of the matrix $T$.

 Very closely related is the \emph{geometric entanglement} of quantum states. Consider a $p$-partite quantum state $\lvert \psi \rangle \in (\C^d)^{\otimes p}$ (that is, a normalized tensor); then its geometric entanglement (see, \cite{shimony1995degree,barnum2001monotones,wei2003geometric} and, e.g., \cite{aubrun2017alice} for a textbook treatment) is defined as
\begin{align}
    \mathrm{GME}(\lvert \psi \rangle)\defeq - \log(\injnorm{\lvert \psi\rangle}^2).
\end{align}
There are variations to this definition; in particular the natural $\log$ is often replaced by $\log_2$ or $\log_d$.

In this paper, we will be interested in various scaling limits: Both the case of $p$ fixed and $d \to \infty$, which is natural in spin glasses, as well as the case of $d$ fixed and $p \to \infty$, which is natural in quantum information (see Proposition \ref{prop:aubrun-gross} below).\\

To set up our main result, we let $T$ be a random standard Gaussian real or complex tensor, meaning the components of $T$ are distributed according to $T_{i_1,\ldots, i_p}\sim\mc{N}_K(0,1)$ for all $i_m\in [d]$ (in the complex case, we normalize the real and imaginary parts to be independent with variance $1/2$). With $\|T\|^2_{\textup{HS}} = \sum_{i_1,\ldots,i_p=1}^d \abs{T_{i_1,\ldots,i_p}}^2$ the Hilbert-Schmidt norm, we let $\lvert \psi \rangle=\frac{T}{\lnorm T \rnorm_{\textup{HS}}}$ be its normalized counterpart, so that $\lvert \psi\rangle \sim \mathrm{Unif}(\mathbb{S}^{pd-1}_K)$. We show
\begin{theorem}[Main theorem]\label{thm:main_upper_bound}
Let $T\in (K^d)^{\otimes p}$ and $\lvert \psi\rangle$ be random as above. For each integer $p\ge3$, let $\alpha(p)=\sqrt{p}E_0(p)$, where $E_0(p)$ is given in Lemma \ref{lem:sigma_p}. Let $\gamma_d^K(p)$ be as in Lemma \ref{lem:real_gamma_d_existence} (for $K = \R$) or Lemma \ref{lem:complex_gamma_d_existence} (for $K = \C$). Then we have the following:
\begin{itemize}
\item \textbf{$p$ fixed, $d \to \infty$:} For all $p \geq 2$ and all $\epsilon > 0$,
\begin{align}
\limsup_{d\to \infty}\frac1{p(d-1)}\log \P\left(\frac1{\sqrt{d}}\injnorm{T}>\alpha(p)+\epsilon\right)&<0, \label{eqn:main-thm-tensor-injective-norm-d-to-infty} \\
\limsup_{d\to \infty}\frac1{p(d-1)}\log \P\left(\injnorm{\lvert \psi \rangle}>\frac1{d^{\frac{p-1}{2}}}(\alpha(p)+\epsilon)\right)&<0, \label{eqn:main-thm-state-injective-norm-d-to-infty}
\end{align}
\item \textbf{$d$ fixed, $p \to \infty$:} For all $d \geq 2$ and all $\epsilon > 0$,
\begin{align}
   \limsup_{p\to \infty}\frac1{p}\log \P\left(\injnorm{T}>\sqrt{p(d-1)}\biggl(\gamma^K_d(p)+\frac{\epsilon}{\sqrt{\log p}}\biggr)\right)&<0, \label{eqn:main-thm-tensor-injective-norm-p-to-infty} \\
    \limsup_{p\to \infty}\frac1{p}\log \P\left(\injnorm{\lvert \psi \rangle}>\sqrt{\frac{p(d-1)}{d^p}}\left(\gamma^K_d(p)+\frac{\epsilon}{\sqrt{\log p}}\right)\right)<0. \label{eqn:main-thm-state-injective-norm-p-to-infty}
\end{align}
Furthermore, one can compute $\gamma_d^K(p)=\sqrt{\log p}+\frac{\log \log p}{2\sqrt{\log p}}+\frac{\beta_K(d)}{\sqrt{\log p}}+\oo\left( \frac1{\sqrt{\log p}}\right)$, where $\beta_K(d)$ is an explicit function of $d$ given in \eqref{eqn:beta_d_real} (for $K = \R$) or \eqref{eqn:beta_d_complex} (for $K = \C$).
\end{itemize}
\end{theorem}
Of course, the case $p = 2$ and $d \to \infty$ of this result, which is about operator norms of random matrices, was already known; see Remark \ref{rem:nonsym-vs-sym}.

Shortly before the posting of this paper, we learned of simultaneous independent physics work of Sasakura \cite{Sas2024}, which considers the $p=3$ complex version of this problem.

In order to prove this theorem, we study the number of critical points of two Gaussian processes defined from the real and complex random tensors using the Kac--Rice formula. While this approach is common in the spin-glass literature, we have not found it so much in the random-tensors literature, so we briefly describe it now: The injective norm is the maximum of some real-valued function on a manifold (see \eqref{eqn:intro-hamiltonian} below). In general, if $f$ is a nice function on some nice manifold $M$, we have $\max_{x \in M} f(x) \geq t$ if and only if there exists a critical point $x$ of $f$ at which $f(x) \geq t$. The Kac--Rice formula is a way to count these critical points when $f$ is random, and a first-moment computation gives a one-sided bound: If, for some $t_0$, the expected number of such critical points above level $t_0$ is exponentially small, then with high probability none exist (by Markov's inequality), which means that with high probability $\max_{x \in M} f(x) \leq t_0$. The smallest $t_0$ with this property is thus an upper bound for the maximum of $f$, here meaning the injective norm. (If the second moment of these counts of critical points approximately matches the first moment squared, then one can obtain a matching lower bound; we do not consider second moments in this paper, since the first moment already presented a number of technical difficulties.)

%%%%%%%%%%%%%%%%%%%%%%%%%%%%%%%%%%%%%%%%%%%%%%%%%%%%%%%%%%%%
%%%%%%%%             Subsection: Motivations
%%%%%%%%%%%%%%%%%%%%%%%%%%%%%%%%%%%%%%%%%%%%%%%%%%%%%%%%%%%%

\subsection{Motivations}
The motivations underlying our work stem from various sources. The injective norm manifests across diverse research domains under different guises and names. Among these domains are statistical physics, quantum information and computing, graph and hypergraph theory \cite{friedman1989second}, game theory \cite{aubrun2020universal}, theory of Banach spaces \cite{aubrun2017alice,Grothendieck1952RsumDR}, and data analysis.  However, as for other pivotal tensorial problems, computing the injective norm is a NP-hard problem\footnote{For symmetric tensors, the situation is more nuanced. In fact, computing the injective norm of symmetric tensors becomes hard for large $d$. However, it remains a polynomial problem for fixed $d$ and large $p$, albeit with a precision relative to the Hilbert-Schmidt norm of the tensor, as discussed in \cite{friedland2020spectral}. This last point allows for polynomial computation of the injective norm of bosonic quantum states.}  \cite{hillar2013most}. Even the rough evaluation of the injective norm is  NP-hard \cite{harrow2013testing}. This is a manifestation of the usual curse of dimensionality for non-convex optimisation. Despite this difficulty, there is an attempt at studying the injective norm of large size tensors numerically and comparing different optimization methods in the paper \cite{fitter2022estimating} (see Remark \ref{rem:fln22}).  \\

We now describe some of the motivations driving our work in greater detail.
\\

\noindent{\bf Quantum Information: } In the quantum information context, the injective norm has a very natural interpretation in term of the distance of a $p$-partite state $\lvert \psi \rangle$ to its closest separable state $\lvert x\rangle=\lvert x^{(1)}\rangle \ldots \lvert x^{(p)}\rangle$. In fact, denoting $\text{SEP}$ the set of (real or complex) product pure states, one has
\begin{equation}\label{eq:relation_to_distance}
\text{dist}_{\text{SEP}}(\lvert \psi \rangle)=\min_{\lvert x \rangle \in \text{SEP}}\lnorm \lvert x\rangle -\lvert \psi \rangle\rnorm_2=\sqrt{2-2\max_{\lvert x \rangle \in \text{SEP}} \mathfrak{R}\langle \psi\vert x\rangle}=\sqrt{2-2\injnorm{\lvert \psi\rangle}},
\end{equation}
where $\mathfrak{R}\langle \psi\vert x\rangle$ denotes the real part of the Hermitian product. For the bipartite case $p=2$, the injective norm corresponds to the largest Schmidt coefficient. When considering the components of the quantum state $\lvert \psi \rangle$ in a basis as the elements of a matrix, this coefficient is simply the largest singular value. Furthermore, in this bipartite scenario, geometric entanglement simplifies to the $\infty$-Renyi entropy.
This motivates the introduction of the injective norm as a measure of multipartite (i.e., $p \geq 3$) entanglement, generalizing the bipartite (i.e., $p = 2$) measure described above. The multipartite situation is much less understood than its bipartite counterpart.  \\

One seemingly natural question in quantum information and computing theory is the minimal value of the injective norm over the set of quantum states. In fact, such states are in a sense ``maximal quantum resources'' states \cite{steinberg2022maximizing}. Writing ``$K$-quantum state'' for a quantum state in a vector space over $K$, and recalling that a qudit is an element of $\mathbb{S}_K^{d-1}$, one can easily prove the following (refer to, \textit{e.g.}, the textbook treatment \cite[Lemma 8.26]{aubrun2017alice}).
\begin{proposition}
If $\lvert \psi\rangle \in (K^d)^{\otimes p}$ is a $p$-partite $K$-quantum state of qudits, then $$\injnorm{\lvert \psi\rangle}\ge \frac1{d^{\frac{p-1}{2}}}.$$
\end{proposition}
\begin{proof}
Note that 
\begin{align*}
\injnorm{\lvert \psi \rangle}^2&= \max_{x^{(i)}\in \mathbb{S}_K^{d-1}, \forall i\neq 1}\sum_{i_1=1}^d\left\lvert\sum_{i_2,\ldots, i_p=1}^d\psi_{i_1,i_2,\ldots, i_p}x^{(2)}_{i_2}\ldots x^{(p)}_{i_p} \right\rvert^2 \\
&\ge \sum_{i_1=1}^d\E_{x^{(i\ge 2)}\sim \text{Unif}(\mathbb{S}_K^{d-1})}\left(\left\lvert\sum_{i_2,\ldots, i_p=1}^d\psi_{i_1,i_2,\ldots, i_p}x^{(2)}_{i_2}\ldots x^{(p)}_{i_p} \right\rvert^2\right)=\frac1{d^{p-1}}\lnorm\lvert \psi \rangle\rnorm_2^2,
\end{align*}
proving the claim.
\end{proof}
This lower bound has been proven many times in the literature at different levels of generality. For the bipartite case $p=2$, it is straightforward to observe that the Bell state $\lvert \text{B}\rangle=\frac1{\sqrt{d}}\sum_{i=1}^d \lvert i\rangle \lvert i\rangle$ saturates the bound. Interestingly, when $K=\R$ and $d=2,4,8$ (the real dimensions of the real normed division algebras $\C, \mathbb{H}, \mathbb{O}$, respectively), the minimal value of the injective norm also saturates the aforementioned lower bound (see \textit{e.g.} \cite[Proposition 8.29 \& Exercise 8.10]{aubrun2017alice}). However, in the general $p$ case of qubits (so $d=2$, $K=\C$), one shows that the lower bound cannot be saturated \cite{jung2008reduced}. One can interpret this as a manifestation of monogamy of entanglement.\\
One important fact is that the square of the minimal value of the injective norm gives the worst approximation ratio of a QMA-hard problem that can be achieved by solving a NP-hard optimization problem \cite[Theorem 2]{gharibian2012approximation}, \cite{MontanaroTalk}. 
For the case $d = 2$ and $p \to \infty$, the following two-sided bound was known previously.
\begin{proposition}
\label{prop:aubrun-gross}
\cite{aubrun2017alice,gross2009most}
    If $\lvert \psi \rangle\sim \mathrm{Unif}(\mathbb{S}_\C^{2p-1})$ -- that is, $\lvert \psi \rangle \in (\C^2)^{\otimes p}$ is a uniform $p$ qubits quantum state -- then there exist absolute constants $c,C>0,$ such that
    $$\P\left(c\frac{\sqrt{p\log p}}{2^{\frac{p}{2}}}\le \injnorm{\lvert \psi\rangle}\le C\frac{\sqrt{p\log p}}{2^{\frac{p}{2}}}\right)\underset{p\rightarrow \infty}{\longrightarrow}1.$$
\end{proposition}
This result was obtained using $\epsilon$-nets techniques. The $d = 2$ case of our main result, \eqref{eqn:main-thm-state-injective-norm-p-to-infty}, improves upon the upper bound here by showing that one can take $C = 1+\epsilon$ for any $\epsilon$, analyzing lower-order corrections in $p$, and giving an extension to qudits (\textit{i.e.} $d \geq 2$); we also give a complementary result of $p$ fixed and $d \to \infty$ in eq. \eqref{eqn:main-thm-state-injective-norm-d-to-infty}.
As a result, we also provide a better upper bound on the minimal value of the injective norm and, therefore, on the worst approximation ratio, by product states, of QMA-hard problems.\\

In the opposite scaling limit of $p$ fixed and $d \to \infty$, high-probability bounds of the form $\injnorm{T} = \OO(\sqrt{dp\log p})$ were known previously for real tensors, also from $\epsilon$-net techniques \cite{TomSuz2014, NguDriTra2015}. Informally, our theorem shows that $\injnorm{T} \leq \alpha(p)\sqrt{d} = \sqrt{dp}E_0(p)$, where $E_0(p)$ is the solution of an explicit equation and in fact \eqref{eqn:e0_scaling} below shows that $\lim_{p \to \infty} \frac{E_0(p)}{\sqrt{\log p}} = 1$. \\ 

\noindent{\bf Statistical physics:} We now explain how the injective norm of a real symmetric tensor is related to the ground state of a pure spherical spin glass. The latter is often written as $\max_x \abs{\ip{J, x \otimes \cdots \otimes x}}$ for some nonsymmetric tensor $J$, but if $T$ is the symmetrization of such a tensor, 
\[
    T_{i_1,\ldots,i_p} = \frac{1}{p!} \sum_{\pi \in S_p} J_{i_{\pi(1)},\ldots,i_{\pi(p)}},
\]
then one can easily compute 
\[
    \sum_{i_1,\ldots,i_p = 1}^d T_{i_1,\ldots,i_p} x_{i_1} \cdots x_{i_p} = \frac{1}{p!} \sum_{\pi \in S_p} \sum_{i_1,\ldots,i_p=1}^d J_{i_{\pi(1)},\ldots,i_{\pi(p)}} x_{i_1} \cdots x_{i_p} = \sum_{i_1,\ldots,i_p=1}^d J_{i_1,\ldots,i_p} x_{i_1} \cdots x_{i_p}.
\]
On the other hand, if $T$ is a symmetric tensor, then one can show deterministically that
\[
    \injnorm{T} = \max_{x^{(1)},\ldots,x^{(p)}} \abs{\ip{T,x^{(1)} \otimes \cdots \otimes x^{(p)}}} = \max_{x} \abs{\ip{T,x \otimes \cdots \otimes x}}.
\]
(This fact seems to have been rediscovered many times in the literature; while we have not searched extensively for the earliest reference, a 1990 paper of Waterhouse \cite{Wat1990} suggests that the fairest way to attribute it is to split it between a 1928 paper of Kellogg \cite{Kel1928} and a 1935 paper of van der Corput and Schaake \cite{vanSch1935}.) Thus the injective norm of a real symmetric tensor is exactly the ground state of a pure $p$-spin spherical spin glass, and in the regime $p \geq 3$ fixed and $d \to \infty$ this is completely understood from the Kac--Rice side \cite{auffinger2013random, subag2017}.\\

When $T$ is a nonsymmetric Gaussian tensor, its injective norm corresponds to the negative of the ground state energy of a certain $p$-partite spherical spin glass model. This model's Hamiltonian reads: 
\begin{equation}
\label{eqn:intro-hamiltonian}
    H(x^{(1)},x^{(2)},\ldots,x^{(p)})=\sum_{i_1,\ldots,i_p=1}^dT_{i_1,\ldots,i_p}x^{(1)}_{i_1}\ldots x^{(p)}_{i_p}.
\end{equation}
While substantial progress has been made in understanding the ground state energy and landscape complexity of some spin glass models, general multipartite (also called \emph{multispecies}) spin glasses remain less understood. The bipartite spherical $p,q$-spin glass is studied in \cite{mckenna2024complexity, kivimae2023ground}. Multipartite spin glasses are studied in, for instance, \cite{bates2022crisanti,Sub2023}. However, the latter papers require some convexity property of a function $\zeta$ (see \textit{e.g.} \cite[(H3)]{bates2022crisanti}) which is used to construct the covariance of the relevant Gaussian process. Notably, our case involves a non-convex function $\zeta(q_1,\ldots, q_p)\defeq\prod_{i=1}^p q_i$ (see Remark \ref{rem:subag}). Moreover, our investigation also deals with a rarely seen regime in the statistical physics literature on spherical spin glasses: the regime of a large number of species.\\

Even though the cases of complex $T_{i_1,\ldots,i_p}$ are perhaps more interesting from the perspective of quantum information, the mathematical literature on spin glasses is mostly restricted to the case of $T_{i_1,\ldots,i_p}$ real; indeed, we are not aware of previous mathematical work on what one might call ``complex spin glasses.'' However, the statistical-physics work \cite{kent2021complex} explores the following complex $p$-spin Hamiltonian
\begin{equation}\label{eq:intro_C_p_spin}
    H_\C^{\text{p-spin}}(z)=\frac1{p!}\sum_{i_1,\ldots, i_p=1}^NJ_{i_1,\ldots,i_p}z_{i_1}\ldots z_{i_p},
\end{equation}
where $J_{i_1,\ldots,i_p}$ are complex Gaussian random variables and $z\in \C^N$ is a complex unit vector. This scenario of a complex random background becomes relevant where phases play an important role, such as for the physics of lasers in random background. They also argue that their Hamiltonian bears similarity with quiver Hamiltonian used in physics to model black hole horizons in the low temperature regime \cite{anninos2016disordered}. As \cite{anninos2016disordered} considers a multipartite-type Hamiltonian, our Hamiltonian \eqref{eqn:intro-hamiltonian} seems to provide additional connections to these quiver-type problems.\\

Therefore, our result on the injective norm of random tensors provides new results on the ground-state energy of multipartite real and complex spherical spin glasses, and suggests a natural motivation for complex spin glasses. \\

\noindent{\bf Landscape complexity:} As previously mentioned, our techniques come from studying critical points of random functions. This general program, sometimes called ``landscape complexity,'' is based fundamentally on the Kac--Rice formula, a kind of change-of-variables formula which translates these sorts of random-geometry questions into random-matrix questions. (For textbook treatments, we direct readers to \cite{adler2007random} and \cite{AzaWsc2009}.) Specifically, the Kac--Rice formula requires one to understand the determinant of a large random matrix (coming from the Hessian at a critical point). Historically, the bottleneck in applying Kac--Rice is whether or not this random matrix is tractable. The earliest results, like \cite{Fyo2004} and \cite{auffinger2013random}, considered random functions whose Kac--Rice-corresponding random matrix was closely related to the Gaussian Orthogonal Ensemble (GOE), which is the most classical random matrix. Previous work of the second author, joint with Ben Arous and Bourgade, gave a general recipe for computing such determinants when the random matrix is far from the GOE \cite{arous2022exponential} and explained how to use these more general results in Kac--Rice applications \cite{arous2024landscape}. While our proof is grounded in the techniques of these two papers, another motivation for our work is to make these techniques even more robust, and this work features several technical complications that we are not aware of in prior work. For example, the Hamiltonian \eqref{eqn:intro-hamiltonian} is linear in each $x^{(i)}$, and thus the corresponding Hessian has large diagonal zero blocks (coming from two derivatives in the same $x^{(i)}$). These zero blocks mean that the joint distribution of the field, gradient, and Hessian is degenerate, so textbook versions of Kac--Rice do not apply, and we need to derive a (one-sided) version of it; these zero blocks also require adaptation of the random-matrix arguments of \cite{arous2022exponential}, which had some assumptions like ``enough randomness on the diagonal'' which do not hold here. Additionally, for the complex case, the Kac--Rice random matrix has the block form $(\begin{smallmatrix} B & C \\ C & -B \end{smallmatrix})$ (informally coming from the Cauchy--Riemann equations), plus one additional row and column, i.e., it has quite substantial correlations among the entries. Escaping these difficulties relies partly on the recent work \cite{bandeira2023}; this provides bounds on the operator norm of very general Gaussian matrices, which are basically tight in our case, and which are obtained by interpolating between free semi-circular random variables and classical Gaussian random variables.  \\

\noindent{\bf Data analysis:}
Computing the injective norm of random tensors is also an important problem in data analysis, in the context of tensor PCA and independent components analysis (\cite{NguDriTra2015} and references therein). Here, random tensors, acting as noise, are perturbed by finite rank signals \cite{richard2014statistical, arous2019landscape,ros2019complex, piccolo2023topological}. One is interested in both detection and recovery of the signal. For $p=2$, detection and recovery depend on the signal-to-noise ratio, as seen in the BBP phase transition \cite{baik2005phase,benaych2011eigenvalues}. In one phase, detection and recovery are impossible, while in the other phase both are possible. The order parameter of this phase transition is the injective norm of the tensor (here, reducing to the operator norm of the matrix). Notably, this transition happens for a $\OO(1)$ signal-to-noise ratio. In the case $p\ge3$, of symmetric real tensors, a similar statement can be made. Weak recovery of the signal is information theoretically possible, when the injective norm departs from its value for purely random tensors, as has been shown in \cite{richard2014statistical}. However, due to the NP-hardness of computing the injective norm at $p\ge 3$, efficient recovery algorithms below a certain large threshold remain elusive, as explored in prior research \cite{richard2014statistical,ben2023long}.\\

Those results concern symmetric tensors. Indeed, symmetric tensors are easier to deal with technically. However, data tensors and the corresponding signals are not always symmetric, and so one might generalize results and problematics to the nonsymmetric case. Our work contributes in this direction by showing preliminary results on the injective norm of pure noise. Another work (\cite{draisma2016average}) considers the same real nonsymmetric tensor as we do. Their motivation was coming from both algebraic geometry and algebraic statistics.  They derive a version of what we call the Kac--Rice formula to count the number of critical points of the function $(x^{(1)},\ldots,x^{(p)}) \mapsto \sum_{i_1,\ldots,i_p=1}^d (T_{i_1,\ldots,i_p} - x^{(1)}_{i_1} \cdots x^{(p)}_{i_p})^2$. They analyze it exactly for $p = d = 2$, and give some numerical results for other values of $p$ and $d$. They also remark that the theory of large random matrices could be useful to compute asymptotics, which is what we do in this paper.

%%%%%%%%%%%%%%%%%%%%%%%%%%%%%%%%%%%%%%%%%%%%%%%%%%%%%%%%%%%%
%%%%%%%%             Subsection: Organization of the paper
%%%%%%%%%%%%%%%%%%%%%%%%%%%%%%%%%%%%%%%%%%%%%%%%%%%%%%%%%%%%

\subsection{Organization of the paper}
The paper is organized as follows. In Section \ref{sec:RMTandTools}, we introduce random-matrix-theory elements that are necessary to determine the asymptotics of the Kac--Rice formula. We also recall the definitions of concepts useful in this context. In Section \ref{sec:real_case}, we study the case of real Gaussian and uniform normalized tensors. To this aim, we use the Kac--Rice formula to bound the excursion probability of a Gaussian process whose supremum is the injective norm. In Section \ref{sec:complex_case}, we study the case of complex Gaussian and uniform normalized tensors (that is, uniform quantum states). The rest of the paper consists of an Appendix, whose goal is to prove a weak version of Kac--Rice formula applying to our specific situation which is slightly degenerate.

%%%%%%%%%%%%%%%%%%%%%%%%%%%%%%%%%%%%%%%%%%%%%%%%%%%%%%%%%%%%
%%%%%%%%             Subsection: Acknowledgements
%%%%%%%%%%%%%%%%%%%%%%%%%%%%%%%%%%%%%%%%%%%%%%%%%%%%%%%%%%%%

\subsection{Acknowledgements}
We wish to thank Paul Bourgade, Guillaume Dubach, Camille Male, Ion Nechita, Valentina Ros, Jens Siewert and Gilles Z\'emor for helpful discussions, references and their interest for this work. We would also like to thank Luca Lionni for informing us about the upcoming work \cite{Sas2024}. BM was partially supported by NSF grant DMS-1760471. SD was partially supported by the ARC grant DE210101323. Moreover, both BM and SD would like to thank CEA List, the Institut de Math\'ematiques de Bordeaux, and the Institut de Math\'ematiques de Toulouse for their hospitality during parts of this work.

%%%%%%%%%%%%%%%%%%%%%%%%%%%%%%%%%%%%%%%%%%%%%%%%%%%%%%%%%%%%
%%%%%%%%%%%%%%%%%%%%%%%%%%%%%%%%%%%%%%%%%%%%%%%%%%%%%%%%%%%%
%%%%%%%%
%%%%%%%%             Section: Random matrices and common tools
%%%%%%%%
%%%%%%%%%%%%%%%%%%%%%%%%%%%%%%%%%%%%%%%%%%%%%%%%%%%%%%%%%%%%
%%%%%%%%%%%%%%%%%%%%%%%%%%%%%%%%%%%%%%%%%%%%%%%%%%%%%%%%%%%%

\section{Random matrices and common tools}\label{sec:RMTandTools}

%%%%%%%%%%%%%%%%%%%%%%%%%%%%%%%%%%%%%%%%%%%%%%%%%%%%%%%%%%%%
%%%%%%%%             Subsection: Random-matrix definitions and first properties
%%%%%%%%%%%%%%%%%%%%%%%%%%%%%%%%%%%%%%%%%%%%%%%%%%%%%%%%%%%%

\subsection{Random-matrix definitions and first properties}

We will need two random-matrix models that we have not previously found in the literature, so although we are not sure whether they have interest beyond these problems, we will give them names so that we can refer to them succintly.

We call the first one the ``Block Hollow Gaussian Orthogonal Ensemble'' (BHGOE), because matrices are sometimes called ``hollow'' if their diagonal entries vanish (see, e.g. \cite{burkhardt2018random}), but in our definition what vanishes is actually large diagonal blocks. This ensemble depends on $d, p \in \mathbb{N}^*$ and $\sigma^2 > 0$, and we denote it as $\text{BHGOE}(d,p,\sigma^2)$.\\

Note that in the course of this section as well as the following ones, we use $N$ as a shorthand to denote the size of the random matrix under consideration. In particular, $N$ means different things in the real case versus the complex case. We recall the specific value of $N$ in terms of relevant parameters $d,p$ when needed, that is mostly at the beginning of Sections \ref{sec:KR-formula-real} and \ref{sec:KR-formula-complex}, to avoid confusion about its value.

\begin{definition}\label{def:BHGOE}
A $\textup{BHGOE}(d,p,\sigma^2)$ random matrix is a $pd \times pd$ real-symmetric random matrix, thought of as being split into $p^2$ blocks each of size $d \times d$, such that the diagonal blocks vanish, and the remaining entries are independent up to symmetry, each distributed according to a normal distribution $\mathcal{N}(0,\sigma^2)$. In other words, such a matrix is a GOE matrix with the diagonal blocks zeroed out.
\end{definition}

We will mostly use the special case $\sigma^2 = \frac{1}{pd}$, which we abbreviate as $\textup{BHGOE}(d,p)$. The BHGOE appears for the case of real tensors; for complex tensors, we will need the following second model, which we call ``twisted BHGOE'' and write as $\textup{tBHGOE}(d,p)$.

\begin{definition}\label{def:tBHGOE}
A $\textup{tBHGOE}(d,p)$ random matrix has the distribution of
 \begin{equation}
 \label{eqn:tBHGOE}
   W = \begin{pNiceMatrix}
      B &  C &   \Block{2-1}{\theta^T}\\
      C & -B \\
      \hline
      \Block{1-2}{\theta} & &   \Block{1-1}{0} \\
      %& & & &
      
      \CodeAfter
      \tikz \draw (1-|3) -- (4-|3) ;
    \end{pNiceMatrix},
\end{equation}
where $B$ and $C$ are independent copies of $\textup{BHGOE}(d,p,\frac{1}{2dp})$, and where $\theta \in R^{2dp}$ is a centered Gaussian (row) vector, independent of $B$ and of $C$, with the covariance
\[
    \E[\theta_a\theta_b] = \frac{1}{2dp}\delta_{ab} \mathbf{1}\{a \not\in \{1, \ldots, d\} \sqcup \{dp+1,\ldots,(p+1)d\}\},
\]
i.e. $\theta$ has i.i.d. $\mc{N}(0,\frac{1}{2dp})$ entries except that the first $d$ of each half are replaced with zeros. Notice that the $tBHGOE(d,p)$ is an $N \times N$ real-symmetric random matrix with $N = 2dp+1$. 
\end{definition}

In the following results, we give exponentially-high-probability bounds on the operator norm of these matrices, using heavily results of Bandeira, Boedihardjo, and van Handel \cite{bandeira2023}.

\begin{lemma}
\label{lem:hgoe_bandeira_combined}
There exists a universal constant $C$ such that, if $W \sim \textup{BHGOE}(d,p,\sigma^2)$, then for any finite $p$ and $d$ we have
\begin{align}
\label{eqn:hgoe_bandeira_combined}
    &\E[ \|W\|_{\textup{op}} ] \leq 2\sqrt{(p-1)d\sigma^2} + C \sigma (dp)^{\frac{1}{4}} (\log(dp))^{\frac{3}{4}}, \\
    &\P\left( \|W\|_{\textup{op}} > 2\sqrt{(p-1)d\sigma^2} +  C\sigma(dp)^{\frac{1}{4}}(\log(dp))^{\frac{3}{4}} + C\sigma t\right) \leq e^{-t^2}.
\end{align}
\end{lemma}
\begin{proof}
This is essentially an immediate consequence of the matrix concentration inequalities of Bandeira, Boedihardjo, and van Handel. Indeed, since $W$ is a $dp \times dp$ self-adjoint matrix with jointly Gaussian entries, combining Theorem 1.2 of \cite{bandeira2023} along with the bound from eq. (1.10) there (for the first estimate), or Corollary 2.2, eq. (1.10), and the estimate notated as $\sigma_\ast(X) \leq v(X)$ on \cite[p. 10]{bandeira2023} (for the second estimate), gives a universal $C$ such that
\begin{equation}
\label{eqn:hgoe_bandeira_general_combined}
\begin{split}
    \E[\|W_N\|_{\textup{op}}] \leq 2\sigma(W_N) + Cv(W_N)^{\frac{1}{2}} \sigma(W_N)^{\frac{1}{2}} (\log(dp))^{\frac{3}{4}}, \\
    \P(\|W\|_{\textup{op}} > 2\sigma(W) + Cv(W)^{\frac{1}{2}} \sigma(W)^{\frac{1}{2}} (\log(dp))^{\frac{3}{4}} + Cv(W)t) \leq e^{-t^2},
\end{split}
\end{equation}
where $\sigma(W)$ and $v(W)$ are certain quantities we now compute. (Note that $\sigma(W)$, a matrix statistic whose notation we borrow from \cite{bandeira2023}, is different from $\sigma^2$, the variance of the nonzero entries of $W$.) Since $W$ has independent entries up to symmetry, $\E[W^2]$ is a diagonal matrix, in fact equal to $(p-1)d\sigma^2 \Id$. Thus
\[
    \sigma(W) \defeq \| \E[W^2] \|_{\textup{op}}^{1/2} = \sqrt{(p-1)d\sigma^2}.
\]
Now we consider $\Cov(W)$, which is a real-symmetric $(dp)^2 \times (dp)^2$ matrix whose rows and columns are indexed by entries of $W$, and whose entries store the corresponding covariances:
\[
    \Cov(W)_{ij,k\ell} = \Cov(W_{ij},W_{k\ell}).
\]
This matrix is quite sparse: If $W_{ij} = 0$, then the corresponding row and column of $\Cov(W)$ vanish; otherwise, the corresponding row and column each have exactly two nonzero entries, in the $(ij)$th position and the $(ji)$th position, each of which have the value $\sigma^2$. Thus the maximum absolute row sum and the maximum absolute column sum of $\Cov(W)$ are both $2\sigma^2$, i.e. $\|\Cov(W)\|_1 = \|\Cov(W)\|_\infty = 2\sigma^2$, where $\|\cdot\|_p$ is the matrix operator norm induced by the Euclidean $p$-norm for vectors. Thus H\"{o}lder's inequality gives
\[
    v(W) \defeq \|\Cov(W)\|_{\textup{op}}^{1/2} \leq (\|\Cov(W)\|_1 \|\Cov(W)\|_\infty)^{1/4} = \sqrt{2\sigma^2}.
\]
Plugging these back into \eqref{eqn:hgoe_bandeira_general_combined}, and applying some trivial upper bounds, finishes the proof.
\end{proof}

The following corollary is immediate by taking $\sigma^2 = \frac{1}{dp}$ in Lemma \ref{lem:hgoe_bandeira_combined}.
\begin{corollary}\label{cor:hgoe_bandeira}
There exists a universal constant $C$ such that, if $W_N \sim \textup{BHGOE}(d,p)$, then for any finite $p$ and $d$ we have 
\[
    \E[ \|W_N\|_{\textup{op}} ] \leq 2\sqrt{\frac{p-1}{p}} + \frac{C(\log(dp))^{\frac{3}{4}}}{(dp)^{\frac{1}{4}}}. 
\]
If $p$ is fixed and $d \to \infty$, then for any fixed $\epsilon, \delta > 0$ there exists $d_0(\epsilon,\delta)$ such that for all $d \geq d_0(\epsilon,\delta)$ we have 
\[
    \P\left(\|W_N\|_{\textup{op}} > 2\sqrt{\frac{p-1}{p}} + \epsilon\right) \leq e^{-N^{1-\delta}}.
\]
If $p \to \infty$ (and either $d$ is fixed or $d \to \infty$), then for any fixed $\epsilon, \delta > 0$ there exists $p_0(\epsilon,\delta)$ such that for all $p \geq p_0(\epsilon,\delta)$ we have 
\[
    \P\left(\|W_N\|_{\textup{op}} > 2 + \epsilon\right) \leq e^{-N^{1-\delta}}.
\]
\end{corollary}

\begin{lemma}\label{lem:est_op_norm_tBHGOE}
There exists a universal constant $C$ such that, if $W_N \sim \textup{tBHGOE}(d,p)$, then for any finite $p$ and $d$ we have
\[
    \E[\|W_N\|_{\textup{op}}] \leq 2\sqrt{\frac{p-1}{p}} + \frac{C(\log N)^{\frac{3}{4}}}{N^{\frac{1}{4}}}.
\]
Now if we take the limit $N \to \infty$ (either from $p \to \infty$ with $d$ fixed, or $d \to \infty$ with $p$ fixed, or $d, p \to \infty$ simultaneously): For every $\epsilon, \delta > 0$, there exists $N_0(\epsilon,\delta)$ such that for all $N \geq N_0(\epsilon,\delta)$ we have 
\[
    \P\left( \|W_N\|_{\textup{op}} > 2\sqrt{\frac{p-1}{p}} + \epsilon\right) \leq e^{-N^{1-\delta}}.
\]
\end{lemma}
\begin{proof}
This is another application of \cite{bandeira2023}. We begin again with 
\begin{equation}
\label{eqn:hgoe_bandeira_probability_complex}
\begin{split}
    \E[\|W_N\|_{\textup{op}}] \leq 2\sigma(W_N) + Cv(W_N)^{\frac{1}{2}} \sigma(W_N)^{\frac{1}{2}} (\log N)^{\frac{3}{4}}, \\
    \P(\|W_N\|_{\textup{op}} > 2\sigma(W_N) + Cv(W_N)^{\frac{1}{2}} \sigma(W_N)^{\frac{1}{2}} (\log N)^{\frac{3}{4}} + Cv(W_N)t) \leq e^{-t^2}.
\end{split}
\end{equation}
The matrix $\E[(W_N)^2]$ is still diagonal, but it is no longer a constant times identity, because of the zero entries of $\theta$; its entries are either $\frac{d(p-1)}{2dp} + \frac{d(p-1)}{2dp} = \frac{p-1}{p}$ or $\frac{d(p-1)}{2dp} + \frac{d(p-1)}{2dp} + \frac{1}{2dp} = \frac{p-1}{p} + \frac{1}{2dp}$. Thus 
\[
    \sigma(W_N) = \sqrt{\frac{p-1}{p} + \frac{1}{2dp}} = \sqrt{\frac{p-1}{p} + \frac{1}{N-1}} \leq \sqrt{\frac{p-1}{p}} + \frac{1}{\sqrt{N-1}}.
\]
Each row (resp., each column) of the matrix $\Cov(W_N)$ has either zero, two, or four nonzero entries, depending on whether it comes from the zero matrix entries, the nonzero entries of $\theta$, or the nonzero entries of $B$ or $C$. Each of these entries is at most $\frac{1}{2dp}$ in absolute value, so
\[
    v(W) = \|\Cov(W)\|_{\textup{op}}^{\frac{1}{2}} \leq (\|\Cov(W)\|_1 \|\Cov(W)\|_\infty)^{\frac{1}{4}} \leq \left(\frac{2}{dp}\right)^{\frac{1}{2}} = \frac{2}{\sqrt{N-1}}.
\]
Plugging these into \eqref{eqn:hgoe_bandeira_probability_complex} completes the proof. 
\end{proof}

We now come to a few additional definitions. We define the log potential of a measure $\mu$.

\begin{definition}
    The log-potential of a measure $\mu$ is defined by 
    \begin{equation}
        \Omega_{\mu}(u)=\int_{\mathbb{R}}\log\lvert u-\lambda\rvert\diff \mu(\lambda). 
    \end{equation}
    Let $\mu(x)$ be the semicircle law, whose density with respect to the Lebesgue measure is given by 
    $$\mu(x)=\frac1{2\pi}\sqrt{(4-x^2)_+}.$$
    We denote its log-potential simply by $\Omega(u)$; this has the explicit expression (cf. \cite[eq. (2.4)]{subag2017})
\[
    \Omega(x) = \begin{cases} \frac{x^2}{4} - \frac{1}{2} & \text{if } \abs{x} \leq 2, \\ \frac{x^2}{4} - \frac{1}{2} - \left[ \frac{\abs{x}}{4} \sqrt{x^2-4} - \log\left( \sqrt{\frac{x^2}{4}-1} + \frac{\abs{x}}{2}\right) \right] & \text{if } \abs{x} \geq 2. \end{cases} 
\]
\end{definition} 
The function $\Sigma_p$ defined in the lemma below is used later to determine the asymptotics of the injective norm of random tensors. 
\begin{lemma}
\label{lem:sigma_p}
Set
\[
    \Sigma_p(u) = \frac{1+\log(p-1)}{2} + \Omega\left(u\sqrt{\frac{p}{p-1}}\right) - \frac{u^2}{2}.
\]
This function is even in $u$, concave, and for $p > 2$ there exists a unique positive solution $E_0(p)$ to 
\[
    \Sigma_p(E_0(p)) = 0.
\]
(For $p = 2$ this function vanishes identically on $[-\sqrt{2},\sqrt{2}]$; by convention we set $E_0(2) = \sqrt{2}$.) 
This number scales for large $p$ as
\begin{equation}
\label{eqn:e0_scaling}
    \lim_{p \to \infty} \frac{E_0(p)}{\sqrt{\log p}} = 1,
\end{equation}
and for every $p \geq 2$ we have
\begin{equation}
\label{eqn:e0_lower_bound}
    E_0(p) \geq 2\sqrt{\frac{p-1}{p}},
\end{equation}
with sharp inequality for $p \geq 3$.
\end{lemma}

\begin{proof}
It is easy to check that $\Sigma''_p(u) \leq \Sigma''_p(0) = \frac{p}{2(p-1)} - 1 \leq 0$. For the claims about $E_0(p)$, we note that $\Sigma_p$ and $E_0(p)$ have already appeared multiple times in the literature under different names: The function $\Theta_p$ defined in (2.15) of \cite{auffinger2013random} has $\Theta_p(u) = \Sigma_p(\min(u,0))$; the function called $\Sigma_{p+q}$ in \cite{mckenna2024complexity} is the same as $\Theta_{p+q}$ from \cite{auffinger2013random}, so the asymptotics \eqref{eqn:e0_scaling} appear as Corollary 2.8 in \cite{mckenna2024complexity}. Since
\[
    \Sigma_p\left(2\sqrt{\frac{p-1}{p}}\right) = \frac{1+\log(p-1)}{2} + \Omega(2) - 2\left(\frac{p-1}{p}\right) = -1 + \frac{\log(p-1)}{2} + \frac{2}{p}
\]
vanishes for $p = 2$ and is strictly increasing in $p$, and $\Sigma_p$ vanishes at $E_0(p)$, the claim \eqref{eqn:e0_lower_bound} follows.
\end{proof}
We also recall the definition of the Wasserstein$-1$ distance on measures. If $\mu, \nu $ are two measures on $\R$, then the Wasserstein$-1$ distance between them, denoted $W_1(\mu, \nu)$, is defined as
\begin{equation}
    W_1(\mu, \nu)\defeq\sup\left\{\int_\R f(x)\diff(\mu-\nu)(x): f:\R\rightarrow \R \text{ continuous and }\lvert f \rvert_{\mc{L}}\le 1 \right\},
\end{equation}
where $\lvert f \rvert_{\mc{L}}$ is the Lipschitz constant of $f$. If $M$ is a matrix, we write $\hat{\mu}_M = \sum_{i=1}^N \delta_{\lambda_i(M)}$ for its empirical measure of eigenvalues.

%%%%%%%%%%%%%%%%%%%%%%%%%%%%%%%%%%%%%%%%%%%%%%%%%%%%%%%%%%%%
%%%%%%%%%%%%%%%%%%%%%%%%%%%%%%%%%%%%%%%%%%%%%%%%%%%%%%%%%%%%
%%%%%%%%
%%%%%%%%             Section: Real case
%%%%%%%%
%%%%%%%%%%%%%%%%%%%%%%%%%%%%%%%%%%%%%%%%%%%%%%%%%%%%%%%%%%%%
%%%%%%%%%%%%%%%%%%%%%%%%%%%%%%%%%%%%%%%%%%%%%%%%%%%%%%%%%%%%

\section{Real case}\label{sec:real_case}

In this section we fix the underlying field $K=\R$. So we let $T\in (\R^d)^{\otimes p}$ be a real Gaussian random tensor of order $p$ and size $d$, that is $T$ is an array $(T_{i_1,i_2,\ldots,i_p})_{i_1,i_2,\ldots,i_p=1}^d$ where $T_{i_1,i_2,\ldots,i_p}$ are independent real standard normal random variables, as stated in the introduction. 
We are interested in the asymptotics ``$d\rightarrow \infty$, $p$ fixed'' and ``$d$ fixed, $p\rightarrow \infty$'' of the injective norm 
\[
    n_T(p,d)\defeq\injnorm{T}=\max_{x^{(i)}\in \mathbb{S}^{d-1}_\R}\lvert\langle T,x^{(1)}\otimes\ldots\otimes x^{(p)} \rangle \rvert = \max_{x^{(i)} \in \mathbb{S}^{d-1}_\R} \langle T,x^{(1)} \otimes \cdots \otimes x^{(p)} \rangle,
\]
where the last equality holds because, if $(x^{(1)},\ldots,x^{(p)})$ is a maximizer of $\lvert\langle T,x^{(1)}\otimes\ldots\otimes x^{(p)} \rangle \rvert$, then so is, e.g., $(\pm x^{(1)},\ldots,x^{(p)})$, and one of these sign choices makes $\langle T,x^{(1)} \otimes \cdots \otimes x^{(p)} \rangle$ positive. 
For the normalized real case, recalling the Hilbert--Schmidt norm $\|T\|_{\textup{HS}}^2 = \sum_{i_1, \ldots, i_p=1}^d \abs{T_{i_1,\ldots,i_p}}^2$, we introduce the normalised tensor $\widetilde{T}$ with components 
\[
    \widetilde{T}_{i_1,\ldots,i_p} = \frac{1}{\|T\|_{\textup{HS}}}T_{i_1,\ldots,i_p},
\]
with corresponding injective norm
\[
    n_{\widetilde{T}}(p,d) = \frac{1}{\|T\|_{\textup{HS}}}n_T(p,d).
\]
This turns $\widetilde{T}$ into a $\R$-quantum state. \\

To access the properties of the injective norm, we study the critical points of the random function $f_T:(\mathbb{S}^{d-1}_{\R})^{p}\rightarrow \mathbb{R}$ defined by 
\begin{equation}
   f_T(x^{(1)},x^{(2)},\ldots, x^{(p)})=(p(d-1))^{-1/2} \sum_{i_1,i_2,\ldots,i_p=1}^d T_{i_1,i_2,\ldots,i_p} x^{(1)}_{i_1}x^{(2)}_{i_2}\ldots x^{(p)}_{i_p},
\end{equation}
where the normalization $(p(d-1))^{-1/2}$ is used for later convenience.
%where here $\mathbb{S}^{d-1}$ is seen as the set of vectors of norm $1$ in $\mathbb{R}^d$. 
Equivalently, $f_T$ is the centered Gaussian process on this product of unit spheres with covariance
\[
    \E[f_T(x^{(1)},\ldots,x^{(p)}) f_T(y^{(1)},\ldots,y^{(p)})] = \frac{1}{p(d-1)} \prod_{i=1}^p \langle x^{(i)},y^{(i)} \rangle.
\]

%%%%%%%%%%%%%%%%%%%%%%%%%%%%%%%%%%%%%%%%%%%%%%%%%%%%%%%%%%%%
%%%%%%%%             Subsection: Kac--Rice formula
%%%%%%%%%%%%%%%%%%%%%%%%%%%%%%%%%%%%%%%%%%%%%%%%%%%%%%%%%%%%

\subsection{Kac--Rice formula}\label{sec:KR-formula-real}
We use the Kac--Rice formula to count the number of critical points of the process $f_T$. We define, for all Borel sets $D\subseteq \R$,
$$\Crt_{f_T}(D)\defeq\{x\in (\mathbb{S}_\R^{d-1})^{p}:\nabla f_T(x)=0, f_T(x)\in D\}.$$
An important note is that this process is isotropic, in the sense that the covariance is invariant under the action of $O(d)^{p}$, where $O(d)$ is the orthogonal group. This is important because the Kac--Rice formula is initially an integral over the high-dimensional product of spheres $(\mathbb{S}_\R^{d-1})^{p}$, but this isotropy property, along with the transitivity of the action of $O(d)$ on $\mathbb{S}_\R^{d-1}$, ensures that the integrand is constant on this product of spheres, so we can pull it out and pick up only a volume factor. Additionally, our process is degenerate: Since $\ip{T,x^{(1)} \otimes \cdots \otimes x^{(p)}}$ is linear in each $x^{(i)}$, the Hessian has zero blocks on the diagonal, and thus standard versions of Kac--Rice do not apply. In Appendix \ref{app:KR-with-degeneracies} we prove a version which is valid for our process, but only as an upper bound, since this is both much simpler than the corresponding lower bound and also sufficient for our purposes. The result is the following simplified formula.

\begin{lemma}\label{lem:exact_kac_rice}
(Proved in Appendix \ref{app:KR-with-degeneracies})
Fix $p$ and $d$, denote $N = p(d-1)$, and define $F(p,d) = \sqrt{\frac{N}{2\pi}}^{N+1} \left( \frac{2\pi^{d/2}}{\Gamma(d/2)} \right)^p$. Then, whenever $D \subset \R$ is a finite union of intervals, we have
\[
    \E[\Crt_{f_T}(D)] \leq F(p,d) \int_D e^{-N\frac{u^2}{2}} \E_{\textup{BHGOE}}[\abs{\det(W_N-u)}] \diff u
\]
where the expectation is taken over $W_N \sim \textup{BHGOE}(d-1,p)$.
\end{lemma}

\begin{remark}
As previously mentioned, a (two-sided) analogue of this result for $D=\R$ was shown in previous work of Draisma and Horobe\c{t} \cite{draisma2016average}. However, since we also need more general sets $D$, and the complex version which they do not consider, we have to re-derive both the real and complex versions of this result (which have similar proofs) in Appendix \ref{app:KR-with-degeneracies}.
\end{remark}

\begin{remark}\label{rem:p=2-relation-to-singular-values}
Interestingly, for $p=2$ we have 
$$W_N=\begin{pmatrix} 0 & A^t\\A & 0\end{pmatrix}$$
so that $$\lvert\det(W_N-u)\rvert=\lvert \det(A^tA-u^2)\rvert$$ where now $S=A^tA$ is a $(d-1)\times (d-1)$ Wishart random matrix, so that the density of eigenvalues for $S$ reads
\begin{equation}
    \rho(\lambda_1,\ldots,\lambda_{d-1})=C_{d-1}\prod_{i<j}\lvert s_i-s_j \rvert \prod_{i=1}^{d-1}s_i^{-1/2}e^{-\frac{N}{2}s_i},
\end{equation}
$C_{d-1}$ being a normalization constant that one can compute explicitly. As one is interested in the average
$$\sqrt{\frac{2\pi}{N}}\E[\lvert\det(W_N-u)\rvert]=\int_D \diff u e^{-\frac{N}{2}u^2}\E_{BHGOE}[\lvert\det(W_N-u)\rvert],$$
in light of what is said above and performing the change of variables $u^2\rightarrow t$, one has
\begin{multline*}
\int_D \diff u \, e^{-\frac{N}{2}u^2}\E_{BHGOE}[\lvert\det(W_N-u)\rvert]\\=\widetilde{C}_{d-1}\int_{\tilde D} \diff t \int_0^\infty \prod_{i=1}^{d-1}ds_i \, t^{-1/2} e^{-\frac{N}{2}t} \prod_{i=1}^{d-1}\lvert s_i-t \rvert \prod_{i<j}\lvert s_i-s_j \rvert \prod_{i=1}^{d-1}s_i^{-1/2}e^{-\frac{N}{2}s_i},
\end{multline*}
 where $\tilde{D}\defeq\{x \in \R_+: \pm \sqrt{x}\in D\}$. Setting $s_{d}\defeq t$, one recognizes the integral of the density of eigenvalues of a real Wishart $d\times d$ random matrix, such that the last eigenvalue is in $\tilde D$, so that
 $$\E[\Crt_{f_T}(D)]\propto\int_{\tilde D}ds_{d}\rho_1(s_d)$$
 where $\rho_1(s_d)$ is the one point eigenvalue density. This is not surprising, as for $p=2$, $x^{(1)},x^{(2)}\in \mathbb{S}_\R^{d-1}$, we have $\lvert \langle T,x^{(1)}\otimes x^{(2)}\rangle \rvert=\lvert \langle x^{(1)},T x^{(2)}\rangle\rvert$ when considering $T$ as a matrix. The critical points of $\lvert \langle x^{(1)},T x^{(2)}\rangle \rvert$ are the singular vectors of $T$ (up to a choice of sign), by the Courant-Fisher variational characterization of the singular values of a matrix, and the values of $\lvert \langle x^{(1)},T x^{(2)}\rangle\rvert$ at those critical points are the singular values of $T$. The distribution of those singular values is indeed given by the Wishart integral above.
  This remark is therefore a good check of our Kac--Rice computation. Eventhough Lemma \ref{lem:exact_kac_rice} is only an upper bound, we believe it to be true as an equality\footnote{In fact, we do not prove the equality in this paper because it is not useful for our result and avoids additional technical difficulties. However, we do not think there are any obstructions to proving the equality.}, and our calculation shows that this equality is verified at $p=2$. \\ 
  For higher values $p\ge 3$, such a simple identification is not possible anymore.
\end{remark}

%%%%%%%%%%%%%%%%%%%%%%%%%%%%%%%%%%%%%%%%%%%%%%%%%%%%%%%%%%%%
%%%%%%%%        Subsection: Real case: p fixed, d \to \infty
%%%%%%%%%%%%%%%%%%%%%%%%%%%%%%%%%%%%%%%%%%%%%%%%%%%%%%%%%%%%

\subsection{Real case: \texorpdfstring{$p$}{p} fixed, \texorpdfstring{$d \to \infty$}{d to infinity}}\label{subsec:R_case_d_to_infinity}
In this section we study the case of a real tensor in the regime of large local dimension $d$ while maintaining $p$ fixed. This regime is the typical regime studied in the spin glass and data analysis literature; readers familiar with this literature should think of this as a (nonsymmetric) pure $p$-spin multispecies spherical spin glass in high dimensions. We study both the case of un-normalized and normalized tensors. We show the following result.
\begin{theorem}
\label{thm:d_to_infinity_first_moment}
\textbf{(Real, nonsymmetric, $p$ fixed, $d \to +\infty$, upper bound.)}
Fix $p \geq 2$, and let
\[
    \alpha(p) = E_0(p)\sqrt{p}.
\]
Then for every $\epsilon > 0$ we have
\begin{align}
    \limsup_{d \to \infty} \frac{1}{d} \log \P\left(\frac{n_T(p,d)}{\sqrt{d}} > \alpha(p) + \epsilon\right) < 0, \label{eqn:real_nonsym_unnorm} \\
    \limsup_{d \to \infty} \frac{1}{d} \log \P\left( \frac{n_{\widetilde{T}}(p,d)}{\sqrt{d}} d^{p/2} > \alpha(p) + \epsilon\right) < 0. \label{eqn:real_nonsym_norm}
\end{align}
\end{theorem}
The normalized case \eqref{eqn:real_nonsym_norm} is a simple corollary of the unnormalized one \eqref{eqn:real_nonsym_unnorm}, relying essentially on the fact that $\|T\|_{\textup{HS}}^2$ concentrates very quickly about $d^p$. This is classical since, in our Gaussian context, $\|T\|_{\textup{HS}}^2 = \sum_{i_1,\ldots,i_p=1}^d \abs{T_{i_1,\ldots,i_p}}^2$ is actually a $\chi^2$ random variable with $d^p$ degrees of freedom (in the real case, or one-half of a $\chi^2$ variable with $2d^p$ degrees of freedom in the complex case, coming from the independent real and imaginary parts, each of variance $1/2$). The specific estimate we will use is given in the lemma below.

\begin{lemma}
\label{lem:LauMas2000}
\cite[Lemma 1]{LauMas2000}
For any $x > 0$, we have
\begin{align*}
    \text{Real case:} \qquad &\P(\|T\|_{\textup{HS}}^2 \leq d^p - 2d^{p/2}x) \leq \exp(-x^2), \\
    \text{Complex case:} \qquad &\P(\|T\|_{\textup{HS}}^2 \leq d^p - \sqrt{2}d^{p/2}x) \leq \exp(-x^2).
\end{align*}
\end{lemma}

Before proving this result, we give some coherency checks with respect to existing results in the following remarks.

\begin{remark}\label{rem:nonsym-vs-sym}
Informally speaking (and ignoring errors), our theorem says that, for $p$ fixed and $d \to \infty$, 
\[
    \max_{x^{(1)},\ldots,x^{(p)} \in \mathbb{S}^{d-1}} \abs{ \sum_{i_1,\ldots,i_p = 1}^d T_{i_1,\ldots,i_p} x^{(1)}_{i_1} \cdots x^{(p)}_{i_p} } \leq \sqrt{dp} E_0(p).
\]
For comparison, from spherical $p$-spin in $d$ dimensions \cite{auffinger2013random, subag2017} (recall our remark earlier that $T$ is projected onto its symmetric part by the Hamiltonian $H = \ip{T,x \otimes \cdots \otimes x}$), we know that
\[
    \max_{x \in \mathbb{S}^{d-1}} \abs{ \sum_{i_1,\ldots,i_p = 1}^d T_{i_1,\ldots,i_p} x_{i_1} \cdots x_{i_p} } \approx \sqrt{d} E_0(p),
\]
which of course is coherent because the former maximizes over a bigger set than the latter. Interestingly, if we assume that the upper bound is tight, the difference between the two is simply a factor of $\sqrt{p}$.
\end{remark}

\begin{remark}
\label{rem:p=2-rmt}
The $p = 2$, $d \to \infty$, unnormalized case of our real and complex results are both known from random matrices, where they describe the operator norms of $d \times d$ real/complex Ginibre matrices. These are known to scale like $2\sqrt{d}$ in our normalizations, matching our bound of $\alpha(2)\sqrt{d} = 2\sqrt{d}$.
\end{remark}

\begin{remark}
\label{rem:fln22}
The work \cite{fitter2022estimating} treats algorithms to compute the injective norm, and gives some numerical lower bounds for several different ensembles of random and deterministic tensors. These ensembles include, as a special case, the real and complex cases of our model, in the scaling limit $d \to \infty$ and $p = 2, 3$ fixed. The $p = 2$ matrix case was previously understood, but they use it to benchmark their algorithms before handling the new $p = 3$ case. Their numerical lower bounds closely match our analytic upper bounds (to match conventions, we note that their $n$ is our $p$, but we use $d$ in the same way):
\begin{itemize}
\item \textbf{Real case:} Their tensor entries are $\mc{N}(0,2/d)$-distributed in the real case, so their injective norm corresponds to our $\sqrt{2/d} \cdot n_T(p,d)$, so the upper bound we give translates into their normalization as $\sqrt{2}\alpha(p) = \sqrt{2p}E_0(p)$. When $p = 2$, this is $2\sqrt{2} \approx 2.828$, as indeed appears in Table 1 of \cite{fitter2022estimating} (and their numerical lower bound for this quantity is $\approx 2.795$). When $p = 3$, their numerical lower bound is $\approx 3.950$, which is quite close to our analytical upper bound of $\sqrt{6}E_0(3) \approx 4.054$.
\item \textbf{Complex case:} Their tensor entries have independent real and imaginary parts, each distributed as $\mc{N}(0,1/d)$, so their injective norm corresponds to our $\sqrt{2/d} \cdot n_T(p,d)$ and our analytic upper bound (from \eqref{eqn:complex_nonsym_unnorm} below) again translates to $\sqrt{2}\alpha(p) = \sqrt{2p}E_0(p)$. When $p = 3$, their numerical lower bound is $\approx 4.143$, whereas our analytical upper bound is $\sqrt{6}E_0(3) \approx 4.054$. We imagine that the mis-ordering of these quantities is a finite-size effect, and note that they are still quite close. (We also note that Fitter, Lancien, and Nechita considered the difference between $\approx 3.950$ and $\approx 4.143$ to be immaterial, and conjectured that the true analytic constants were the same in both cases; our main theorem gives an upper bound consistent with this conjecture.)
\end{itemize}
\end{remark}

\begin{remark}
\label{rem:subag}
Our model can be viewed as a pure multispecies spin glass, and recent work of Subag on such models \cite{Sub2023} takes the form ``Assuming the free energies of certain mixtures converge at all temperatures, the limiting ground-state energy of a pure multispecies spin glasses is given by the implicit solution of some equation.'' As Subag remarks, this convergence is expected in general, but currently only known \cite{BatSoh2022} when the original mixture function is convex, which is not the case for pure models like ours (our mixture function is $\xi(x_1,\ldots,x_p) = \prod_{i=1}^p x_i$). Therefore our result is not a consequence of \cite{Sub2023}. However, his conditional two-sided result seems to match our unconditional one-sided bound: Comparing our normalizations (his $N$ is our $dp$, his $\mathscr{S}$ is our $\{1,\ldots,p\}$, and so on), we see that the results match and our bound is tight if
\begin{equation}
\label{eqn:subag_east}
    E_0(p) = E_\ast(p),
\end{equation}
where $E_\ast(p)$ is found by letting $q_c = q_c(p) \in (0,1)$ be the unique solution to the implicit equation
\[
    \frac{q_c^2}{p(1-q_c)} = \frac{-\log(1-q_c)}{1+p(1-q_c)\frac{1}{q_c}},
\]
and then setting
\[
    E_\ast(p) = \sqrt{\left( -\log (1-q_c) \right)\left( 1+p(1-q_c)\frac{1}{q_c}\right)}.
\]
We do not currently know how to show \eqref{eqn:subag_east} analytically, but one can compute both sides numerically for various $p$ values, and they seem to agree (for example, one can numerically estimate $q_c(p=3) \approx 0.645$ and $q_c(p=4) \approx 0.805$, which give $E_\ast(3) \approx 1.657 \approx E_0(3)$ and $E_\ast(4)\approx 1.794 \approx E_0(4)$). 
\end{remark}

\begin{remark}\label{rem:entanglement-no-sym-vs-sym}
   We note that our result is consistent with the folklore notion (see, e.g., the comments above Theorem 1.1 in \cite{friedland2018most}) that nonsymmetric states are generally more entangled than their symmetric counterparts. Indeed, consider symmetric real standard Gaussian tensors $S=P_S(T)$, where $P_S$ is the projector onto the symmetric subspace $\text{Sym}_p(\R^d)$, that is $P_S(T)_{i_1,\ldots, i_p}=\frac{1}{p!} \sum_{\pi \in S_p} T_{i_{\pi(1)},\ldots,i_{\pi(p)}}$.  According to \cite{Kel1928,vanSch1935,friedland2013best}, the injective norm of a symmetric tensor is realized by a symmetric rank one tensor. Therefore, the injective norm of a real symmetric tensor can be obtained from the ground state of spherical spin glasses deduced in \cite{auffinger2013random}. Indeed, from \cite{auffinger2013random, subag2017}, we obtain, as $d\rightarrow \infty$, $\injnorm{S}\overset{\P}{\rightarrow}E_0(p)$ (where $\overset{\P}{\rightarrow}$ denotes convergence in probability). Additionally, $\sqrt{\frac{p!}{d^{p-1}}}\lnorm S\rnorm_2\overset{\P}{\rightarrow} 1$. By Slutsky's theorem, we have that $\sqrt{\frac{d^{p-1}}{p!}}\frac{\injnorm{S}}{\lnorm S \rnorm_2}\rightarrow E_0(p)$ in distribution (and so in probability since $E_0(p)$ is a constant), and so $\frac{\injnorm{S}}{\lnorm S \rnorm_2}\sim_{d\rightarrow \infty}\frac{\sqrt{p!}E_0(p)}{d^{\frac{p-1}{2}}}.$ 
    Denoting $\lvert \psi^{\text{sym}}\rangle=\frac{S}{\lnorm S\rnorm_2}$ the normalized version of $S$ to get a $\R$-quantum state, and using our main theorem \ref{thm:main_upper_bound} for $K=\R$, shows roughly that $\frac{\injnorm{\ket{\psi}}}{\injnorm{\lvert \psi^{\text{sym}}\rangle}}\le \sqrt{\frac{1}{(p-1)!}}$, and thus $\textup{GME}(\ket{\psi}) - \textup{GME}(\ket{\psi^{\textup{sym}}}) \geq \log((p-1)!) > 0$, in the limit of large $d$.
\end{remark}

Theorem \ref{thm:d_to_infinity_first_moment} will follow rapidly from the following lemma, which we prove afterwards.

\begin{lemma}
\label{lem:d_to_infinity_first_moment}
For every $D \subset \R$ with positive Lebesgue measure that is the closure of its interior, and that is contained in
\begin{equation}
\label{eqn:def_a}
    \mc{A} = (-\infty,-E_0(p)) \cup (E_0(p),+\infty),
\end{equation}
we have
\[
    \lim_{d \to \infty} \frac{1}{p(d-1)} \log \E[\Crt_{f_T}(D)] \leq \sup_{u \in D} \Sigma_p(u).
\]
\end{lemma}

\begin{remark}
Since $D$ is closed and $\mc{A}$ is open, notice that $D$ must be gapped away from $\pm E_0(p)$. The requirement $D \subset \mc{A}$ is only used for the Wegner estimate, Lemma \ref{lem:wegner}.
\end{remark}

\begin{proof}[Proof of Theorem \ref{thm:d_to_infinity_first_moment}]
First we prove \eqref{eqn:real_nonsym_unnorm}. With $\epsilon' = \frac{\epsilon}{\sqrt{p}}$, we find
\begin{align*}
    \P\left(\frac{n_T(p,d)}{\sqrt{d}} > \alpha(p) + \epsilon\right) &= \P\left(\frac{n_T(p,d)}{\sqrt{p(d-1)}} > \frac{\alpha(p) + \epsilon}{\sqrt{p}} \cdot \sqrt{\frac{d}{d-1}} \right) \\
    &\leq \P\left(\frac{n_T(p,d)}{\sqrt{p(d-1)}} \geq E_0(p) + \epsilon'\right) \\
    &\leq \P(\Crt_{f_T}([E_0(p) + \epsilon',+\infty)) \geq 1) + \P(\Crt_{f_T}((-\infty,-E_0(p) - \epsilon']) \geq 1) \\
    &= 2\P(\Crt_{f_T}((-\infty,-E_0(p) - \epsilon']) \geq 1),
\end{align*}
since $f_T$ and $-f_T$ have the same distribution. Markov's inequality then gives 
\[
    \P\left(\frac{n_T(p,d)}{\sqrt{d}} > \alpha(p) + \epsilon\right) \leq 2\E[\Crt_{f_T}((-\infty,-E_0(p)-\epsilon'))].
\]
Let $D = (-\infty,-E_0(p) - \epsilon']$; then Lemma \ref{lem:sigma_p} gives $\sup_{u \in D} \Sigma_p(u) = \Sigma_p(-E_0(p) - \epsilon') < 0$, and thus Lemma \ref{lem:d_to_infinity_first_moment} finishes the proof of \eqref{eqn:real_nonsym_unnorm}.

The normalized version \eqref{eqn:real_nonsym_norm} is a short corollary of the unnormalized version \eqref{eqn:real_nonsym_unnorm}: For every $\epsilon' > 0$, we have
\begin{align*}
    &\P\left( \frac{n_{\widetilde{T}}(p,d)}{\sqrt{d}} \sqrt{d^p} > \alpha(p) + \epsilon\right) \\
    &= \P\left( \frac{n_T(p,d)}{\sqrt{d}} \sqrt{\frac{d^p}{\|T\|_{\textup{HS}}^2}} > \alpha(p) + \epsilon\right) \\
    &\leq \P\left( \frac{n_T(p,d)}{\sqrt{d}} \sqrt{\frac{d^p}{\|T\|_{\textup{HS}}^2}} > \alpha(p) + \epsilon, \frac{\|T\|_{\textup{HS}}^2}{d^p} > 1-\epsilon'\right) + \P\left(\frac{\|T\|_{\textup{HS}}^2}{d^p} \leq 1-\epsilon'\right) \\
    &\leq \P\left( \frac{n_T(p,d)}{\sqrt{d}} > (\alpha(p) + \epsilon)\sqrt{1-\epsilon'}\right) + \P\left(\frac{\|T\|_{\textup{HS}}^2}{d^p} \leq 1-\epsilon'\right).
\end{align*}
For $\epsilon'$ small enough depending on $p$ and $\epsilon$, we have $(\alpha(p)+\epsilon)\sqrt{1-\epsilon'} > \alpha(p) + \frac{\epsilon}{2}$, say; then Theorem \ref{thm:d_to_infinity_first_moment} takes care of the first probability, and Lemma \ref{lem:LauMas2000} gives the final estimate % it remains only to show that
\[
    \limsup_{d \to \infty} \frac{1}{d} \log \P\left(\frac{\|T\|_{\textup{HS}}^2}{d^p} \leq 1-\epsilon'\right) < 0.
\]
\end{proof}

The proof of Lemma \ref{lem:d_to_infinity_first_moment} is a consequence of Lemma \ref{lem:exact_kac_rice} and Lemma \ref{lem:d_to_infinity_first_moment_laplace}.

\begin{lemma}
\label{lem:d_to_infinity_first_moment_laplace}
We have
\begin{equation}
\label{eqn:kdp_large_p}
    \lim_{d \to \infty} \frac{1}{p(d-1)} \log F(p,d) = \frac{1+\log(p)}{2},
\end{equation}
and, for every $D \subset \mc{A}$ with positive Lebesgue measure that is the closure of its interior, 
\begin{equation}
\label{eqn:p_to_infinity_laplace}
\begin{split}
    &\lim_{d \to \infty} \frac{1}{p(d-1)} \log \left( \int_D \ e^{-\frac{N}{2}u^2}\E_{\textrm{BHGOE}}\left[\lvert \det(W_N-u)\vert \right] \diff u \right) \\
    &= \frac{\log(p-1)-\log(p)}{2} + \sup_{u \in D} \left\{ \Omega\left(u \sqrt{\frac{p}{p-1}}\right) - \frac{u^2}{2}\right\}.
\end{split}
\end{equation}
\end{lemma}

\begin{remark}
Because Lemma \ref{lem:exact_kac_rice} is only an upper bound, we have only stated Lemma \ref{lem:d_to_infinity_first_moment} as an upper bound. However, in Lemma \ref{lem:d_to_infinity_first_moment_laplace} we prove both upper and lower bounds, in case they are useful for future work. An analogous remark holds for, e.g., Lemma \ref{lem:laplace_d_fixed} below.
\end{remark}

To prove Lemma \ref{lem:d_to_infinity_first_moment_laplace}, we need to prove several good properties of the matrices $W_N-u$, which we now call 
\[
    H_N(u) \defeq W_N - u.
\]
The reason is that our proof of Lemma \ref{lem:d_to_infinity_first_moment_laplace} is essentially of the form ``Apply Theorem 4.1 of \cite{arous2022exponential},'' and the next several lemmas serve to check the assumptions of this theorem. In the proofs we will need the measures $\mu_p(u)$, parameterized by $u \in \R$, which are certain rescalings and shifts of the semicircle law, precisely those with density $\mu_p(u,\cdot)$ with respect to Lebesgue measure given by
\begin{equation}
\label{eqn:def_mu_infty}
    \mu_p(u, x) = \frac{1}{2\pi} \frac{p}{p-1} \sqrt{\left(4\left(\frac{p-1}{p}\right) - (x+u)^2\right)_+}.
\end{equation}

\begin{lemma}
\label{lem:real-lsi}
There exist $C, c > 0$ such that, for every Lipschitz $f : \R \to \R$,
\begin{equation}
\label{eqn:lsi}
    \sup_{u \in \R} \P\left(\abs{\frac{1}{N} \Tr(f(H_N(u))) - \frac{1}{N} \E[\Tr(f(H_N(u)))]} > \delta \right) \leq C\exp\left(-\frac{cN^2\delta^2}{\|f\|_{\textup{Lip}}^2}\right).
\end{equation}
\end{lemma}
\begin{proof}
Notice that $\hat{\mu}_{H_N(u)}$ and $\mu_p(u)$ are the pushforwards of $\hat{\mu}_{H_N(0)}$ and $\mu_p(0)$, respectively, by the map ``translate by $u$.'' We can shift this translation into the function $f$, which does not effect the Lipschitz norm; thus it suffices to check \eqref{eqn:lsi} without the supremum over $u$. Since $W_N$ has independent Gaussian entries (some degenerate), all of variance order at most $\frac{1}{N}$, its upper triangle, considered as a vector, satisfies the log-Sobolev inequality with a constant of order $\frac{1}{N}$. Then the $u=0$ version of \eqref{eqn:lsi} is a corollary of the Herbst argument (see \cite[Lemma 2.3.3 \& Theorem 2.3.5]{anderson2010introduction} for a very similar argument).
\end{proof}

\begin{lemma}
\label{lem:thm_1.2}
There exists $\kappa > 0$ such that
\begin{equation}
\label{eqn:w1}
    \sup_{u \in \R} {\rm W}_1 (\E[\hat{\mu}_{H_N(u)}],\mu_p(u)) \leq N^{-\kappa}.
\end{equation}
\end{lemma}
\begin{proof}
 Since ${\rm W}_1$ is invariant under translations, ${\rm W}_1(\E[\hat{\mu}_{H_N(u)}], \mu_p(u))$ is actually independent of $u$. Thus for \eqref{eqn:w1} it suffices to consider $u = 0$, and we recall that $H_N(0) = W_N$. 

When $u = 0$, the proof of \eqref{eqn:w1} is similar to that found in the proof of Corollary 1.10 of \cite{arous2022exponential}. However, the latter result had some assumption essentially saying that the diagonal entries had enough randomness; that assumption is violated here, since the diagonal entries of $W_N$ are deterministically zero. Nevertheless, a close reading of that proof shows that most parts go through. Changes are only necessary to check the assumption called (W) there, via Proposition 3.1 there, which has three conditions to check, with equation numbers (3.1), (3.2), and (3.3). These changes are essentially of the form ``at a few precise points in the proof of \cite{arous2022exponential}, one needed to plug in a local law; for our model, we use a different local law and/or the operator norm estimates coming from \cite{bandeira2023}''; we sketch these changes in the following list, then give them in detail afterwards.
\begin{itemize}
\item Equation (3.1) of \cite{arous2022exponential} is some condition saying the averaged Stieltjes transform of $\hat{\mu}_{W_N}$ is close to that of $\mu_p(0)$ when evaluated at $z = E+\ii N^{-\gamma}$, for some small $\gamma > 0$ and $E$ in some big compact set. This follows from a local law; we give the relevant local law for our model in \eqref{eqn:kronecker_local_law} below, then for completeness copy the very brief argument leading from there to (3.1).
\item As already noticed in \cite{arous2022exponential}, matrices with independent Gaussian entries like ours satisfy significantly stronger estimates than needed to check (3.2) and (3.3) of \cite{arous2022exponential}. Here, both of those estimates follow from the following claim: There exist constants $c_1, c_2, c_3$ such that, for all $x > 0$,
\begin{equation}
\label{eqn:kronecker_mde_tail}
    \P(\|W_N\| \geq x) \leq c_1 e^{-c_2N(x-c_3)}.
\end{equation}
Indeed, applying the log-Sobolev inequality and the Herbst argument as above, we find
\begin{equation}
    \P(\|W_N\| \geq x) \leq e^{-Nx} \E[e^{N\|W_N\|}] \leq e^{N(\sup_N\E\|W_N\| + 4c - x)}.
\end{equation}
Thus we only need the very weak estimate $\sup_N \E\|W_N\| < \infty$. In \cite[Lemma 3.9]{arous2022exponential} this followed from some local law; here we instead use Corollary \ref{cor:hgoe_bandeira}.

\end{itemize}

Thus it remains only to import the appropriate local law, which in our case is \cite{alt2019location}. In the notation (2.1) there, we take $L = p$, $N = d-1$, $\ell = \frac{p(p-1)}{2}$, all the $\widetilde{\alpha}_i$'s and $\widetilde{a}_i'$ to be zero, all the $Y_\nu$'s to have i.i.d. $\mc{N}(0,\frac{1}{p(d-1)})$ entries, all the $\widetilde{\gamma}_\nu = (\widetilde{\beta}_\nu)^T$, and actually index the $(\widetilde{\beta}_\nu)_{\nu=1}^{\frac{p(p-1)}{2}}$ by the positions of the strict-upper-triangular entries in a $p \times p$ matrix: $(\widetilde{\beta}_{(i,j)})_{1 \leq i < j \leq p}$. We take $\widetilde{\beta}_{(i,j)} = E_{ij}$, the $p \times p$ matrix with a one in position $(i,j)$ and zeros everywhere else. Then $t_{k\ell}^{(i,j)}$ defined in (2.3) there is $t_{k\ell}^{(i,j)} = \frac{1}{p(d-1)}$ for all $i, j, k, \ell$. The operators $\ms{S}_i$ defined in (2.4)\footnote{
Remark 2.5(v) in \cite{alt2019location} explains that the Hermitization is superfluous for Hermitian matrices like ours; thus we treat the un-Hermitized versions of these operators. 
}
actually do not depend on $i$ for us: for $i = 1, \ldots, d-1$, we have $\ms{S}_i : (\C^{p \times p})^{(d-1)} \to \C^{p \times p}$ defined on $(d-1)$-tuples of $p \times p$ matrices $\mathbf{r} = (r_1, \ldots, r_{d-1})$ as
\[
    \ms{S}_i[\mathbf{r}] = \frac{1}{p(d-1)} \sum_{k=1}^{d-1} \sum_{1 \leq m < n \leq p} (E_{mn} r_k E_{nm} + E_{nm} r_k E_{mn}).
\]
Notice that if each $r_i$ is the same constant $c$ times the identity matrix, then since $E_{mn}E_{nm} = E_{mm}$ we have
\begin{equation}
\label{eqn:kronecker_mde_operators}
    \ms{S}_i[(c\Id, \ldots, c\Id)] = \frac{c}{p} \sum_{1 \leq m < n \leq p} (E_{mm} + E_{nn}) = c \left(\frac{p-1}{p}\right)\Id.
\end{equation}
Then we seek, for each $z \in \mathbb{H}$, the unique solution $\mathbf{m}(z) = (m_1(z),\ldots,m_{d-1}(z)) \in (\C^{p \times p})^{(d-1)}$ to the constrained problem
\begin{equation}
\label{eqn:kronecker_mde}
\begin{split}
    &\Id_{p \times p} + (z\Id_{p \times p} + \ms{S}_i[\mathbf{m}(z)])m_i(z) = 0, \quad \text{for each } i = 1, \ldots, d-1, \\
    &\text{subject to } \Im m_i(z) > 0 \text{ as a quadratic form.}
\end{split}
\end{equation}
Then we consider the probability measure $\mu_d = \mu_{p,d}$ on $\R$ whose Stieltjes transform at the point $z$ is $\frac{1}{p(d-1)} \sum_{i=1}^{d-1} \Tr(m_i(z))$. The results of \cite{alt2019location} deal with this measure, which we now study.

We claim that, for some scalar $\widetilde{m}_p(z)$ which is crucially independent of $d$, the solution has the particularly simple form
\begin{equation}
\label{eqn:kronecker_mde_factors}
    m_i(z) = \widetilde{m}_p(z) \Id_{p \times p} \quad \text{for each } i = 1, \ldots, d-1.
\end{equation}
Indeed, consider the constrained problem
\begin{equation}
\label{eqn:kronecker_mde_scalar}
\begin{split}
    &1 + \left[z+\left(\frac{p-1}{p}\right)\widetilde{m}_p(z)\right]\widetilde{m}_p(z) = 0, \\
    &\text{subject to } \Im \widetilde{m}_p(z) > 0.
\end{split}
\end{equation}
On the one hand, if $\widetilde{m}_p(z)$ solves \eqref{eqn:kronecker_mde_scalar}, then clearly the choice \eqref{eqn:kronecker_mde_factors} exhibits a solution to \eqref{eqn:kronecker_mde}, by the computation \eqref{eqn:kronecker_mde_operators}. On the other hand, \eqref{eqn:kronecker_mde_scalar} is explicitly solvable, since it is quadratic in $\widetilde{m}_p(z)$, and clearly has the unique solution
\begin{equation}
\label{eqn:widetilde_m_d}
    \widetilde{m}_p(z) = \frac{-z+\sqrt{z^2-4\left(\frac{p-1}{p}\right)}}{2\left(\frac{p-1}{p}\right)}.
\end{equation}
Thus the measure $\mu_d$ actually has $\widetilde{m}_p(z)$ as its Stieltjes transform; in particular, it is (a) actually independent of $d$, and (b) equal to the shifted, rescaled semicircle law $\mu_p(0)$.  One consequence is that Theorem 2.4 of \cite{alt2019location} says that, with high probability, the spectrum of $W_N$ lies in any $\epsilon$-neighborhood of its limiting support $[-2\sqrt{\frac{p-1}{p}},2\sqrt{\frac{p-1}{p}}]$ (we already saw another proof of this, with a better probability bound, in Corollary \ref{cor:hgoe_bandeira}). However, we use the local law \cite[(B.5)]{alt2019location} from the MDE for Kronecker random matrices to explain how \cite[(3.1)]{arous2022exponential} is satisfied.
 Write $s_N$ for the Stieltjes transform of $\hat{\mu}_{W_N}$. The local law \cite[(B.5)]{alt2019location} thus implies that there exist universal constants $\delta > 0$ and $P \in \N$ such that, for every $0 < \gamma < \delta$ and $\epsilon > 0$, there exists $C_{\epsilon,\gamma}$ with
\begin{equation}
\label{eqn:kronecker_local_law}
    \sup_{E \in \R} \P\left(\abs{s_N(E+\ii N^{-\gamma}) - \widetilde{m}_p(E+\ii N^{-\gamma})} \geq \frac{N^{\epsilon+\gamma P}}{N} \right) \leq C_{\epsilon,\gamma} N^{-100}. 
\end{equation}
From this, we check
\begin{align*}
    \lvert \E[s_N(z)]-\tilde m_{p}(z)\rvert&\le \E\left[\lvert s_N(z)-\tilde m_p(z)\rvert\right]\\
    &=\E\left[\lvert s_N(z)-\tilde m_p(z)\rvert\mathbbm{1}_{\lvert s_N(z)-\tilde m_p(z)\rvert< \frac{N^{\epsilon + \gamma P}}{N}}\right]\\&\hspace{40mm}+\E\left[\lvert s_N(z)-\tilde m_p(z)\rvert\mathbbm{1}_{\lvert s_N(z)-\tilde m_p(z)\rvert\ge \frac{N^{\epsilon + \gamma P}}{N}}\right]\\
    &\le N^{\epsilon+\gamma P-1}+\E\left[\lvert s_N(z)-\tilde m_p(z)\rvert^2\right]\P\left(\lvert s_N(z)-\tilde m_p(z)\rvert\ge \frac{N^{\epsilon + \gamma P}}{N}\right)\\
    &\le N^{\epsilon+\gamma P -1}+2C_{\epsilon, \gamma}N^{-100+2\gamma}\lesssim N^{\epsilon+\gamma P -1}.
\end{align*}
which suffices for \cite[(3.1)]{arous2022exponential}.
\end{proof}

\begin{lemma}
\label{lem:wegner}
For every $\epsilon > 0$ and every closed $K \subset \mc{A}$, 
\[
    \lim_{N \to \infty} \inf_{u \in K} \P(H_N(u) \text{ has no eigenvalues in } [-e^{-N^\epsilon}, e^{-N^\epsilon}]) = 1.
\]
\end{lemma}

\begin{proof}
Decompose $\mc{A}$ as 
\[
    \mc{A} = \mc{A}_- \sqcup \mc{A}_+ \defeq (-\infty,-E_0(p)) \sqcup (E_0(p),+\infty).
\]
Since $K \subset \mc{A}$ is closed, it is gapped away from $\pm E_0(p)$, i.e., we actually have
\[
    K \subset (-\infty,-E_0(p) - \delta_K] \cup [E_0(p) + \delta_K, +\infty)
\]
for some $\delta_K > 0$. Consider the good event
\[
    \mc{E} = \mc{E}_K = \left\{\|W_N\| \leq 2\sqrt{\frac{p-1}{p}} + \frac{\delta_K}{2}\right\}.
\]
On this event, for $u \in K \cap \mc{A}_+$, we use \eqref{eqn:e0_lower_bound} to obtain
\[
    \lambda_{\textup{max}}(H_N(u)) = \lambda_{\textup{max}}(W_N) - u \leq 2\sqrt{\frac{p-1}{p}} + \frac{\delta_K}{2} - E_0(p) - \delta_K \leq -\frac{\delta_K}{2},
\]
and similarly for $u \in K \cap \mc{A}_-$ we have
\[
    \lambda_{\textup{min}}(H_N(u)) = \lambda_{\textup{min}}(W_N) - u \geq -2\sqrt{\frac{p-1}{p}} - \frac{\delta_K}{2} + E_0(p) + \delta_K \geq \frac{\delta_K}{2}.
\]
That is, on the event $\mathcal{E}$, $\mathrm{Spec}\, H_N(u)$ is gapped away from zero. Since $\delta$ depends only on $K$, which is fixed, this gives
\[
    \inf_{u \in K} \P(H_N(u) \text{ has no eigenvalues in } [-e^{-N^\epsilon}, e^{-N^\epsilon}]) \geq \P(\mc{E})
\]
for $N \geq N_0(p,\epsilon,K)$. As a consequence of Corollary \ref{cor:hgoe_bandeira}, the probability of the right hand side tends to one as $N\to \infty$.% As we showed in \eqref{eqn:eps_away_from_spectrum}, the right-hand side probability tends to one as $N \to \infty$.
\end{proof}

\begin{lemma}
\label{lem:cont_decay_in_u}
For each $N$, the map $u \mapsto H_N(u)$ is entrywise continuous, and there exists $C > 0$ with
\[
    \E[\abs{\det(H_N(u))}] \leq (C \max(\|u\|,1))^N.
\]
\end{lemma}

\begin{proof}
The proof mimics that of \cite[Lemma 4.4]{arous2024landscape}, but since it is quite short we reproduce it: First one estimates
\[
    \abs{\det(H_N(u))} \leq \|H_N(u)\|^N \leq (\|W_N\| + \|u\|)^N \leq (2\|W_N\|)^N + (2\|u\|)^N,
\]
and then \eqref{eqn:kronecker_mde_tail} gives an estimate of the form $\E[\|W_N\|^N] \leq C^N$.
\end{proof}

\begin{proof}[Proof of Lemma \ref{lem:d_to_infinity_first_moment_laplace}]
The asymptotics \eqref{eqn:kdp_large_p} are routine. To verify \eqref{eqn:p_to_infinity_laplace}, we use Theorem 4.1 of \cite{arous2022exponential}, with the choices $m = 1$, $\mathfrak{D} = D$, $\alpha = \frac{1}{2}$, $p = 0$, and what \cite{arous2022exponential} calls $\mu_\infty(u)$ being the measure $\mu_p(u)$ given in \eqref{eqn:def_mu_infty} (later, we will use $\mu_\infty(u)$ for the $p \to \infty$ limit of $\mu_p(u)$). We checked the condition ``assumptions locally uniform in $u$'' (in the sense of Theorem 1.2 there) in Lemmas \ref{lem:thm_1.2} and \ref{lem:wegner}, using the constant sequence $(\mu_N(u))_{N=1}^\infty = \mu_p(u)$; this makes the condition ``limit measures'' there vacuous. Finally, we checked the condition ``continuity and decay in $u$'' in Lemma \ref{lem:cont_decay_in_u}. Then Theorem 4.1 gives us
\begin{align*}
    &\lim_{d \to \infty} \frac{1}{p(d-1)} \log \left( \int_D \ e^{-\frac{N}{2}u^2}\E_{\textrm{BHGOE}}\left[\lvert \det(W_N-u)\vert \right] \diff u \right) = \sup_{u \in D} \left\{ \int_\R \log\abs{\lambda} \mu_p(u, \lambda) \diff \lambda - \frac{u^2}{2}\right\}.
\end{align*}
A routine change of variables gives
\[
    \int_\R \log\abs{\lambda} \mu_p(u,\lambda) = \log\sqrt{\frac{p-1}{p}} + \Omega\left(u \sqrt{\frac{p}{p-1}}\right), 
\]
which finishes the proof. 
\end{proof}

%%%%%%%%%%%%%%%%%%%%%%%%%%%%%%%%%%%%%%%%%%%%%%%%%%%%%%%%%%%%
%%%%%%%%        Subsection: Real case: p \to \infty, d fixed
%%%%%%%%%%%%%%%%%%%%%%%%%%%%%%%%%%%%%%%%%%%%%%%%%%%%%%%%%%%%

\subsection{Real case: \texorpdfstring{$p \to \infty$}{p to infinity}, \texorpdfstring{$d$}{d} fixed}\label{subsec:Rcase-p-infinity}
In this section we study the case $K=\R$, $d$ fixed, and $p\to\infty$. This asymptotic regime introduces technical difficulties that are not present in the preceding section. In particular, some arguments of \cite{arous2022exponential} could be used off-the-shelf in the previous part, but require adaptation in this section. This is reflected in the form of the proof of Lemmas \ref{lem:laplace_d_fixed_determinants} and \ref{lem:laplace_d_fixed_laplace}. 

\begin{definition}
\label{definition:beta_d}
For $d \geq 2$, we compute
\begin{equation}
\label{eqn:beta_d_real}
\begin{split}
    \beta_\R(d) &\defeq \lim_{p \to \infty} \left[ \frac{1}{p(d-1)} \log F(p,d) - \frac{\log p}{2} \right] 
    = \frac{1}{2} \log\left(\frac{d-1}{2\pi}\right) + \frac{1}{d-1} \log\left(\frac{2\pi^{d/2}}{\Gamma(d/2)}\right).
\end{split}
\end{equation}
\end{definition}
Although we will not need it, we remark that for large $d$ one can compute
\[
    \beta_\R(d) = \frac{1}{2} + \left(\frac{\log 2}{2}\right) \frac{1}{d} + \OO_{d \to \infty}\left(\frac{1}{d^2}\right),
\]
so that this definition is consistent with \eqref{eqn:kdp_large_p}.

\begin{lemma}\label{lem:real_gamma_d_existence}
Fix $d \geq 2$. For every $p$, there exists a unique $\gamma_d^\R(p) > 0$ satisfying
\[
    \frac{\log p}{2} + \beta_\R(d) + \Omega\left( \gamma_d^\R(p) \right) - \frac{\gamma_d^\R(p)^2}{2} = 0.
\]
For fixed $d$, this quantity scales for large $p$ like
\begin{equation}
\label{eqn:gamma-scaling}
     \gamma_d^\R(p) = \sqrt{\log p}+\frac{\log \log p}{2\sqrt{\log p}}+\frac{\beta_\R(d)}{\sqrt{\log p}}+\oo_{p \to \infty} \left(\frac{1}{\sqrt{\log p}}\right).
\end{equation}

\end{lemma}
\begin{proof} Existence and uniqueness follow from arguments of the proofs of \cite[Corollary 2.8]{mckenna2024complexity} which we repeat here: The function $x \mapsto \frac{x^2}{2}-\Omega(x)$ is even, strictly convex, and globally minimized at $x = 0$, where it takes the value $1/2$, so $\frac{x^2}{2}-\Omega(x) = t$ has a unique positive solution $x = x_t$ for each $t > 1/2$; plus we have $\frac{\log p}{2} + \beta_\R(d) \geq \beta_\R(d) > 1/2$ for each $p$ and $d$.

For the aymptotics, one introduces additional constants $\gamma_d^\R(p)_-$, $\gamma_d^\R(p)^\dagger$, and $\gamma_d^\R(p)^\ast$ as the largest positive solutions of
\begin{align*}
    (\gamma_d^\R(p)_-)^2&=\log p + 2 \beta_\R(d),\\
    (\gamma_d^\R(p)^\dagger)^2&=\log p +2\beta_\R(d)+\log((\gamma_d^\R(p)^\dagger)^2)-\frac1{\gamma_d^\R(p)^\dagger},\\
    (\gamma_d^\R(p)^*)^2&=\log p +2\beta_\R(d)+\log((\gamma_d^\R(p)^*)^2),
\end{align*}
respectively. In fact, considering the functions $f_1(x) = \frac{x^2}{2}$, $f_2(x) = \frac{1}{2}(x^2 - \log(x^2) + \frac{1}{x})$, $f_3(x) = \frac{x^2}{2} - \Omega(x)$, and $f_4(x) = \frac{1}{2}(x^2 - \log(x^2))$, one has $f_1(x) \geq f_2(x) \geq f_3(x) \geq f_4(x)$ for $x$ large enough ($x\ge 2$, say), which implies the ordering
\begin{equation}
\label{eqn:gamma-ordering}
    \gamma_d^\R(p)_{-} \leq \gamma_d^\R(p)^\dagger \leq \gamma_d^\R(p) \leq \gamma_d^\R(p)^\ast.
\end{equation}
We first claim
\begin{equation}
\label{eqn:gamma_d_R_*_asymptotics}
    \gamma_d^\R(p)^*=\sqrt{\log p}+\frac{\log \log p}{2\sqrt{\log p}}+\frac{\beta_\R(d)}{\sqrt{\log p}}+\oo_{p \to \infty} \left(\frac1{\sqrt{\log p}}\right).
\end{equation}
These asymptotics of $\gamma_d^\R(p)^*$ are obtained by noticing that
$$-(\gamma_d^\R(p)^*)^2e^{-(\gamma_d^\R(p)^*)^2}=-\frac1{p}e^{-2\beta_\R(d)}.$$
Since $-e^{-1}\le -\frac1{p}e^{-2\beta_\R(d)}<0$, the solution is given in terms of the second branch of the Lambert $\mc{W}$ function, which is negative and is traditionally denoted $\mc{W}_{-1}(x)$, so that 
\begin{equation}
\label{eqn:lambert-w}
    \gamma_{d}^\R(p)^*=\sqrt{-\mc{W}_{-1}\left(-\frac1{p}e^{-2\beta_\R(d)}\right)}.
\end{equation}
The asymptotics of $\mc{W}_{-1}$ for small negative arguments are known from, e.g., \cite[pp. 349-350]{corless1996lambert}:
$$\mc{W}_{-1}(x)=\log(-x)-\log(-\log(-x))+\frac{\log(-\log(-x))}{\log(-x)}+\OO_{x\to 0^-}\left(\left(\frac{\log(-\log(-x))}{\log(-x)} \right)^2 \right).$$
Plugging in $x = -\frac{1}{p} e^{-2\beta_\R(d)}$ and rearranging, we find 
\begin{align*}
    -\mc{W}_{-1}\left(-\frac{1}{p}e^{-2\beta_\R(d)}\right) &= \log(p) + 2\beta_\R(d) + \log(\log(p) + 2\beta_\R(d)) + \frac{\log(\log(p) + 2\beta_\R(d))}{\log(p) + 2\beta_\R(d)} \\
    &+ \OO_{p \to \infty} \left( \left( \frac{\log(\log(p) + 2\beta_\R(d))}{\log(p) + 2\beta_\R(d)} \right)^2 \right).
\end{align*}
For fixed $c > 0$, one can check $\log(\log(p)+c) = \log\log p + \frac{c}{\log p} + \oo(\frac{1}{\log(p)})$, and $\frac{\log p}{\log p + c} = 1 - \frac{c}{\log p} + \oo(\frac{1}{\log(p)})$, so that $\frac{\log(\log p + c)}{\log p + c} = \frac{\log \log p}{\log p} + \oo(\frac{1}{\log p})$. This gives
\[
    -\mc{W}_{-1}\left(-\frac{1}{p}e^{-2\beta_\R(d)}\right) = \log p + \log\log p + 2\beta_\R(d) + \frac{\log \log p}{\log p} + \frac{2\beta_\R(d)}{\log p} + \oo\left( \frac{1}{\log p}\right).
\]
We now factor out $\sqrt{\log p}$ and use the small-$x$ asymptotics $\sqrt{1+x} = 1+\frac{x}{2}-\frac{x^2}{8}(1+\oo_{x \to 0}(1))$ to find
\begin{align*}
\sqrt{-\mc{W}_{-1}\left(-\frac{1}{p}e^{-2\beta_\R(d)}\right)}&=\sqrt{\log p + \log\log p + 2\beta_\R(d) + \frac{\log \log p}{\log p} + \frac{2\beta_\R(d)}{\log p} + \oo\left( \frac{1}{\log p}\right)}\\
&=\sqrt{\log p}\sqrt{1+\frac{\log \log p}{\log p}+\frac{2\beta_\R(d)}{\log p}+\frac{\log \log p}{(\log p)^2}+\frac{2\beta_\R(d)}{(\log p)^2}+\oo\left(\frac1{(\log p)^2}\right)}\\
&=\sqrt{\log p}\Biggl[1+ \frac{\log \log p}{2\log p}+\frac{\beta_\R(d)}{\log p}+\frac{\log \log p}{2(\log p)^2}+\frac{\beta_\R(d)}{(\log p)^2}+\oo\left(\frac1{(\log p)^2}\right)\\
&-\frac1{8}\left( \frac{\log \log p}{\log p}\right)^2\left(1+\oo\left(1\right)\right)\Biggr]\\
&=\sqrt{\log p}+\frac{\log \log p}{2\sqrt{\log p}}+\frac{\beta_\R(d)}{\sqrt{\log p}}+\frac{\log \log p}{2(\log p)^{\frac32}}+\frac{\beta_\R(d)}{(\log p)^{\frac32}}+\oo\left(\frac1{(\log p)^{\frac32}}\right)\\
&-\frac18\frac{(\log \log p)^2}{(\log p)^{\frac32}}(1+\oo(1)).
\end{align*}
Combining this with \eqref{eqn:lambert-w} proves \eqref{eqn:gamma_d_R_*_asymptotics}. Combined with the simple estimate $\gamma_d^\R(p)_- = \sqrt{\log p}(1+\oo_{p \to \infty}(1))$ and \eqref{eqn:gamma-ordering}, this implies $\gamma_d^\R(p)^\dagger = \sqrt{\log p}(1+\oo_{p \to \infty}(1))$. Consequently, 
\begin{align*}
    \lvert \gamma_d^\R(p)^2-(\gamma_d^\R(p)^*)^2\rvert&\le \lvert(\gamma_d^\R(p)^\dagger)^2 - (\gamma_d^\R(p)^*)^2\rvert =\left\lvert\log\left(\frac{(\gamma_d^\R(p)^\dagger)^2}{(\gamma_d^\R(p)^*)^2}\right)-\frac1{\gamma_d^\R(p)^\dagger} \right\rvert\\
\end{align*}
and the $p\to \infty$ limit of the right hand side is zero. We notice $\lvert \gamma_d^\R(p)^2-(\gamma_d^\R(p)^*)^2\rvert=\lvert\gamma_d^\R(p)+\gamma_d^\R(p)^* \rvert \lvert\gamma_d^\R(p)-\gamma_d^\R(p)^* \rvert$ and remark $\lvert\gamma_d^\R(p)+\gamma_d^\R(p)^* \rvert=2\sqrt{\log p}(1+o(1))$; therefore $\lvert\gamma_d^\R(p)-\gamma_d^\R(p)^* \rvert=\oo(\frac1{\sqrt{\log p}})$ and we conclude \eqref{eqn:gamma-scaling}.
\end{proof}

\begin{theorem}\label{thm:real_p_infty_d_fixed}\textbf{(Real, nonsymmetric, $p \to \infty$, d fixed, upper bound.)}
For every $\epsilon > 0$ and every $d \geq 2$, we have
\[
    \limsup_{p \to \infty} \frac{1}{p} \log \P\left( \frac{n_T(p,d)}{\sqrt{p(d-1)}} - \gamma_d^\R(p) > \frac{\epsilon}{\sqrt{\log p}}\right) < 0.
\]
while the normalized version reads
\[
    \limsup_{p\to \infty}\frac1{p}\log \P\left(n_{\widetilde T}(p,d)>\sqrt{\frac{p(d-1)}{d^{p}}}\gamma_d^\R(p)+\epsilon\sqrt{\frac{p(d-1)}{d^{p}\log p}}\right)<0.
\]
\end{theorem}
Informally, this theorem says that for $d$ fixed and $p \to \infty$, we have
\[
    \max_{x^{(1)},\ldots,x^{(p)} \in \mathbb{S}^{d-1}} \abs{ \sum_{i_1,\ldots,i_p = 1}^d T_{i_1,\ldots,i_p} x^{(1)}_{i_1} \cdots x^{(p)}_{i_p} } \leq \sqrt{p(d-1)\log p}
\]
or more thoroughly
\[
    \max_{x^{(1)},\ldots,x^{(p)} \in \mathbb{S}^{d-1}} \abs{ \sum_{i_1,\ldots,i_p = 1}^d T_{i_1,\ldots,i_p} x^{(1)}_{i_1} \cdots x^{(p)}_{i_p} }\leq \sqrt{p(d-1)}\left( \sqrt{\log p} + \frac{\log \log p}{2\sqrt{\log p}} + \frac{\beta_\R(d)}{\sqrt{\log p}} \right).
\]
As in the opposite scaling limit, the proof starts with Lemma \ref{lem:exact_kac_rice}. We then continue with the following lemma.

\begin{lemma}
\label{lem:laplace_d_fixed}
Fix $d$, and let $(\gamma(p))_{p=1}^\infty$ be some sequence tending to infinity with $p$, such that
\[
    \gamma(p) = \oo(\exp(p^{1-\delta}))
\]
for some $\delta > 0$. Then
\[
    \lim_{p \to \infty} \left[ \frac{1}{p(d-1)} \log \int_{\gamma_d^\R(p)}^\infty e^{-\frac{N}{2}u^2} \E_{\textup{BHGOE}}[\abs{\det(W_N-u)}] \diff u - \left( \Omega\left(\gamma(p)\right) - \frac{(\gamma(p))^2}{2} \right) \right] = 0.
\]
\end{lemma}
We split this lemma into two parts, namely Lemmas \ref{lem:laplace_d_fixed_determinants} and \ref{lem:laplace_d_fixed_laplace} below, but first assume it to prove Theorem \ref{thm:real_p_infty_d_fixed}. 
\begin{proof}[Proof of Theorem \ref{thm:real_p_infty_d_fixed}]
For every $\epsilon > 0$ and every $d \geq 2$, we claim 
\[
    \limsup_{p \to \infty} \frac{1}{p} \log \P\left( \frac{n_T(p,d)}{\sqrt{p(d-1)}} - \gamma_d^\R(p) > \frac{\epsilon}{\sqrt{\log p}}\right) < 0.
\]
Indeed, we have
\begin{align*}
    &\P\left( \frac{n_T(p,d)}{\sqrt{p(d-1)}} > \gamma_d^\R(p) + \frac{\epsilon}{\sqrt{\log p}} \right) \\
    &\leq \P(\Crt_{f_T}([\gamma_d^\R(p)+\epsilon/\sqrt{\log p},+\infty)) \geq 1) + \P(\Crt_{f_T}((-\infty,-\gamma_d^\R(p) - \epsilon/\sqrt{\log p}]) \geq 1) \\
    &= 2\P(\Crt_{f_T}([\gamma_d^\R(p)+\epsilon/\sqrt{\log p},+\infty)) \geq 1) \leq 2\E[\Crt_{f_T}([\gamma_d^\R(p)+\epsilon/\sqrt{\log p},+\infty))].
\end{align*}
Then applying Lemmas \ref{lem:exact_kac_rice} and \ref{lem:laplace_d_fixed}, not to the sequence $\gamma(p) = \gamma_d^\R(p)$ but instead to the sequence $\gamma(p) = \gamma_d^\R(p) + \epsilon/\sqrt{\log p}$, and using Lemma \ref{lem:real_gamma_d_existence} and some of the technology behind its proof to deal with the limit, we obtain
\begin{align*}
    &\limsup_{p \to \infty} \frac{1}{p(d-1)} \log \E[\Crt_{f_T}([\gamma_d^\R(p)+\epsilon/\sqrt{\log p},+\infty))] \\
    &\leq \beta_\R(d) + \limsup_{p \to \infty} \left[ \frac{\log p}{2} + \Omega(\gamma_d^\R(p) + \epsilon/\sqrt{\log p}) - \frac{(\gamma_d^\R(p)+\epsilon/\sqrt{\log p})^2}{2} \right] \leq -\epsilon,
\end{align*}
where the $-\epsilon$ comes from the large-$p$ limit of the cross term $-\frac{2\gamma_d^\R(p)\epsilon/\sqrt{\log p}}{2}$. 

The normalized case is again a consequence of the unnormalized version, but there are a few subtleties not present in the real case because of our finer asymptotics here: Previously it sufficed to use that $\frac{\|T\|_{\textup{HS}}}{d^{p/2}} > 1-\delta$ for every fixed $\delta$ with exponentially high probability, whereas now we need to show that $\frac{\|T\|_{\textup{HS}}}{d^{p/2}} > 1-\delta_p$ with some $\delta_p \to 0$. Indeed, following the proof of Theorem \ref{thm:d_to_infinity_first_moment}, if we insert this event, we find that we need to show 
\[
    \beta_\R(d) + \limsup_{p \to \infty} \left[ \frac{\log p}{2} + \Omega\left((1-\delta_p)\left(\gamma_d^\R(p) + \frac{\epsilon}{\sqrt{\log p}}\right)\right) - \frac{(1-\delta_p)^2(\gamma_d^\R(p) + \epsilon/\sqrt{\log p})^2}{2} \right] < 0.
\]
Since $\Omega$ is increasing, we can discard its $(1-\delta_p)$ factor as an upper bound, and show as above that $\limsup_{p \to \infty} [\Omega(\gamma_d^\R(p) + \epsilon/\sqrt{\log p}) - \Omega(\gamma_d^\R(p))] \leq 0$, so it remains only to show that
\[
    \limsup_{p \to \infty} \left[ -\frac{(1-\delta_p)^2(\gamma_d^\R(p)+\epsilon/\sqrt{\log p})^2}{2} + \frac{(\gamma_d^\R(p)+\epsilon/\sqrt{\log p})^2}{2} \right] \leq 0,
\]
and for this it suffices to show that $\limsup_{p \to \infty} \delta_p (\gamma_d^\R(p) + \epsilon/\sqrt{\log p})^2 \leq 0$, i.e., from \eqref{eqn:gamma-scaling} we have to be able to take some $\delta_p = \oo(1/\log p)$. But the estimate $\limsup_{p \to \infty} \frac{1}{p} \log \P(\frac{\|T\|_{\textup{HS}}}{d^{p/2}} \leq 1-\delta_p) < 0$ for such $\delta_p$ is immediate from Lemma \ref{lem:LauMas2000} (which actually allows us to take $\delta_p \to 0$ much more quickly).

\end{proof}
Lemma \ref{lem:laplace_d_fixed} will follow from the following two lemmas. 
\begin{lemma}
\label{lem:laplace_d_fixed_determinants}
Fix $d$, and let $(\gamma(p))_{p=1}^\infty$ be some sequence tending to infinity with $p$, such that $\gamma(p) = \oo(\exp(p^{1-\delta}))$ for some $\delta > 0$. Then
\[
    \lim_{p \to \infty} \left[ \frac{1}{p(d-1)} \log \int_{\gamma(p)}^\infty e^{-\frac{N}{2}u^2} \E_{\textup{BHGOE}}[\abs{\det(W_N-u)}] \diff u - \frac{1}{p(d-1)} \log \int_{\gamma(p)}^\infty e^{-\frac{N}{2}u^2 + N\Omega(u)} \diff u \right] = 0.
\]
\end{lemma}

\begin{lemma}
\label{lem:laplace_d_fixed_laplace}
Fix $d$, and let $(\gamma(p))_{p=1}^\infty$ be some sequence tending to infinity with $p$, such that $\gamma(p) = \oo(\exp(p^{1-\delta}))$ for some $\delta > 0$. Then
\[
    \lim_{p \to \infty} \left[ \frac{1}{p(d-1)} \log \int_{\gamma(p)}^\infty e^{-\frac{N}{2}u^2 + N\Omega(u)} \diff u - \left( \Omega (\gamma(p)) - \frac{(\gamma(p))^2}{2} \right) \right] = 0.
\]
\end{lemma}

We comment briefly on these lemmas. The first is really about determinants of random matrices, akin to \cite{arous2022exponential}, whereas the second is really about the Laplace method. However, since $\gamma(p) \to \infty$, both applications are nonstandard; roughly speaking, typical applications show $\abs{a_N-b_N} \to 0$ for some sequences $(a_N)_{N=1}^\infty$ and $(b_N)_{N=1}^\infty$ having the same finite limit, whereas in our applications both sequences tend to infinity.

We also need one new piece of notation: Recall that we write $\mu_p(u)$ for the semicircle law rescaled to lie in $[-2\sqrt{\frac{p-1}{p}},2\sqrt{\frac{p-1}{p}}]$, translated by $u$. Now we write $\mu_\infty(u)$ for the $p \to +\infty$ limit of these measures, i.e. the usual semicircle on $[-2,2]$ translated by $u$, with density
\[
    \mu_\infty(u,x) = \frac{1}{2\pi}\sqrt{(4-(x+u)^2)_+}.
\]

\begin{proof}[Proof of Lemma \ref{lem:laplace_d_fixed_determinants}]
It suffices to show
\begin{align}
    \limsup_{p \to \infty} \frac{1}{p(d-1)} \log \sup_{u \geq \gamma(p)} \left\{ \E_{\textup{BHGOE}}[\abs{\det(W_N - u)}]e^{-N \Omega(u)} \right\} &\leq 0, \label{eqn:gammadp_upper_bound} \\
    \liminf_{p \to \infty} \frac{1}{p(d-1)} \log \inf_{u \geq \gamma(p)} \left\{ \E_{\textup{BHGOE}}[\abs{\det(W_N - u)}]e^{-N \Omega(u)} \right\} &\geq 0. \label{eqn:gammadp_lower_bound}
\end{align}
The proof adapts and extends that of \cite[Theorem 1.2, Corollary 1.9.A]{arous2022exponential}. First, we make the following random-matrix claims, recalling the notation $H_N(u) = W_N-u$:
\begin{itemize}
\item For every $\epsilon, \delta > 0$, 
\begin{equation}
\label{eqn:dfixed_wegner}
    \lim_{N \to \infty} \inf_{\abs{u} > 2+\delta} \P(H_N(u) \text{ has no eigenvalues in } [-e^{-N^\epsilon}, e^{-N^\epsilon}]) = 1.
\end{equation}
\item There exist universal $C, c > 0$ such that, whenever $f : \R \to \R$ is Lipschitz,
\begin{equation}
\label{eqn:dfixed_lipschitz_concentration}
    \sup_{u \in \R} \P\left( \abs{ \int_\R f(\lambda) (\hat{\mu}_{H_N(u)} - \E[\hat{\mu}_{H_N(u)}])(\diff \lambda) } > \delta\right) \leq C\exp\left(-\frac{cN^2\delta^2}{\|f\|_{\textup{Lip}}^2} \right).
\end{equation}
\item There exists (small) $\kappa > 0$ such that
\begin{equation}
\label{eqn:w1_nkappa}
    \sup_{u \in \R} {\rm W}_1(\E[\hat{\mu}_{H_N(u)}],\mu_\infty(u)) \leq N^{-\kappa}.
\end{equation}
\end{itemize}
Before proving these claims, we show how they imply \eqref{eqn:gammadp_upper_bound} and \eqref{eqn:gammadp_lower_bound}. 

For the upper bound \eqref{eqn:gammadp_upper_bound}, we apply \eqref{eqn:dfixed_lipschitz_concentration} with $f$ being the mollified logarithm $\log_\eta(x) = \log\abs{x+\ii \eta}$, where 
\begin{equation}
\label{eqn:choice_of_eta}
    \eta = N^{-\kappa/2}
\end{equation}
(here $\kappa$ is as in \eqref{eqn:w1_nkappa}). This function is $\frac{1}{2\eta}$-Lipschitz on $\R$. From this estimate; the equality $\E[X] = \int_0^\infty \P(X \geq t) \diff t$, valid for all nonnegative random variables $X$; and the explicit integral $\int_0^\infty \exp(-k(\log t)^2) \diff t = \sqrt{\frac{\pi}{k}} \exp(\frac{1}{4k})$, after redefining $C \leftarrow C\sqrt{\frac{\pi}{4c}}$, then $c \leftarrow \frac{1}{4c}$, we find
\[
    \sup_{u \in \R} \E[e^{N\int_\R \log_\eta(\lambda)(\hat{\mu}_{H_N(u)} - \E[\hat{\mu}_{H_N(u)}])(\diff \lambda)}] \leq \frac{C}{\eta} \exp\left[ c \left(\frac{1}{2\eta}\right)^2\right].
\]
From this, we obtain the upper bound
\begin{align*}
    &\sup_{u \geq \gamma(p)} \left\{ \frac{1}{N} \log \E[\abs{\det(H_N(u))}] - \Omega(u) \right\} \leq \sup_{u \geq \gamma(p)} \left\{ \frac{1}{N} \log \E[e^{N\int_\R \log_\eta(\lambda) \hat{\mu}_{H_N(u)}(\diff \lambda)}] - \Omega(u) \right\} \\
    &\leq \frac{\log(C/\eta)}{N} + \frac{c}{4N\eta^2} + \sup_{u \geq \gamma(p)} \left\{ \int_\R \log_\eta(\lambda) \E[\hat{\mu}_{H_N(u)}](\diff \lambda) - \Omega(u) \right\} \\
    &\leq \frac{\log(C/\eta)}{N} + \frac{c}{4N\eta^2} + \frac{1}{2\eta} \sup_{u \geq \gamma(p)} {\rm W}_1(\E[\hat{\mu}_{H_N(u)}],\mu_\infty(u)) + \sup_{u \geq \gamma_d^\R(p)} \left\{\int_\R (\log_\eta(\lambda) - \log\abs{\lambda}) \mu_\infty(u,\lambda)\right\}.
\end{align*}
From \eqref{eqn:w1_nkappa} and \eqref{eqn:choice_of_eta}, the first three terms are collectively $\oo(1)$. To estimate the last term, we do not need the full strength of $\gamma(p) \to \infty$, but only that $\gamma(p)$ is gapped away from the semicircle bulk, $\gamma(p) > 2+\epsilon$ for some fixed $\epsilon > 0$; then
\begin{multline*} 
    \sup_{u \geq \gamma(p)} \left\{ \int (\log_\eta(\lambda) - \log\abs{\lambda}) \frac{\sqrt{(4-(\lambda+u)^2)_+}}{2\pi} \diff \lambda\right\} \\
    = \sup_{u \geq \gamma(p)} \left\{ \int (\log_\eta(t-u) - \log\abs{t-u}) \frac{\sqrt{(4-t^2)_+}}{2\pi} \diff t \right\} \\
    = \sup_{u \geq \gamma(p)} \left\{ \int \log\abs{1+\ii \frac{\eta}{t-u}} \frac{\sqrt{(4-t^2)_+}}{2\pi} \diff t \right\} \\
    \leq \int \log\abs{1+\ii \frac{\eta}{t-(2+\epsilon)}} \frac{\sqrt{(4-t^2)_+}}{2\pi} \diff t,
\end{multline*}
and since $\eta = \eta_N \to 0$, this tends to zero by dominated convergence. This completes the proof of the upper bound \eqref{eqn:gammadp_upper_bound}, modulo the random-matrix claims \eqref{eqn:dfixed_wegner}, \eqref{eqn:dfixed_lipschitz_concentration}, and \eqref{eqn:w1_nkappa}.

For the lower bound \eqref{eqn:gammadp_lower_bound}, following \cite{arous2022exponential}, we introduce the parameters
\[
    t = w_b = N^{-\kappa/4}, \qquad p_b = N^{-\kappa^2/8},
\]
as well as, for some fixed small $\epsilon > 0$, the events
\begin{align*}
    \mc{E}_{\textrm{gap}}(u) &= \left\{H_N(u) \text{ has no eigenvalues in } [e^{-N^\epsilon}, e^{-N^\epsilon}]\right\}, \\ %\textrm{Spec}(H_N(u))\cap [u-e^{-N^\epsilon}, u+e^{-N^\epsilon}]=\emptyset\right\},\,\\
    \mc{E}_{\textrm{Lip}}(u) &= \left\{\left\lvert \int \log_\eta(\lambda)(\hat\mu_{H_N(u)}-\E\hat\mu_{H_N(u)})(\diff \lambda)\right\rvert\le t\right\}, \\
    \mc{E}_b(u) &= \left\{ \int_\R b(\lambda) \hat{\mu}_{H_N(u)}(\diff \lambda) \leq p_b\right\},
\end{align*}
where we fix $b = b_N : \R \to \R$ to be some smooth, even, nonnegative function that is identically one on $[-w_b,w_b]$, vanishes outside of $[-2w_b,2w_b]$, and is $\frac{1}{w_b}$-Lipschitz.

With this notation, we have
\begin{equation}
\label{eqn:dfixed_main_lower_bound}
\begin{split}
    &\inf_{u \geq \gamma(p)} \left\{ \frac{1}{N} \log \E[\abs{\det(H_N(u))}] - \Omega(u) \right\} \\
    &\geq \inf_{u \geq \gamma(p)} \left\{ \frac{1}{N} \log \E\left[e^{N \left( \int (\log\abs{\lambda} - \log_\eta(\lambda))\hat{\mu}_{H_N(u)}(\diff \lambda) + \int \log_\eta(\lambda)(\hat{\mu}_{H_N(u)} - \E[\hat{\mu}_{H_N(u)}])(\diff \lambda)\right)} \mathds{1}_{\mc{E}_{\textup{Lip}}(u)} \mathds{1}_{\mc{E}_{\textup{gap}}(u)} \right]\right\}  \\
    &\qquad \qquad + \inf_{u \geq \gamma(p)} \left\{ \int \log_\eta(\lambda) \E[\hat{\mu}_{H_N(u)}](\diff \lambda) - \Omega(u) \right\} \\
    &\geq \inf_{u \geq \gamma(p)} \left\{ \frac{1}{N} \log \E\left[e^{N \int (\log\abs{\lambda} - \log_\eta(\lambda))\hat{\mu}_{H_N(u)}(\diff \lambda)} \mathds{1}_{\mc{E}_{\textup{Lip}}(u)} \mathds{1}_{\mc{E}_{\textup{gap}}(u)} \right]\right\} \\
    & \qquad \qquad - t - \frac{1}{2\eta} \sup_{u \geq \gamma(p)} {\rm W}_1(\E[\hat{\mu}_{H_N(u)}],\mu_\infty(u)) 
\end{split}
\end{equation}
since dominated convergence actually shows that, at fixed $N$, we have
\[
    \inf_{u \geq \gamma(p)} \left\{ \int \log_\eta(\lambda) \mu_\infty(u,\diff \lambda) - \Omega(u) \right\} = \inf_{u \geq \gamma(p)} \int \log \abs{1 + \ii \frac{\eta}{t-u}} \frac{\sqrt{(4-t^2)_+}}{2\pi} \diff t = 0
\]
(the infimum happens as $u \to +\infty$). The last two terms on the right-hand side of \eqref{eqn:dfixed_main_lower_bound} are $\oo(1)$ by our choice of $\eta$ and $t$, so it remains only to consider the first term. Following the proof of \cite[Lemma 2.5]{arous2022exponential}, we lower-bound this term by
\begin{align*}
    &\inf_{u \geq \gamma(p)} \left\{ \frac{1}{N} \log \E\left[e^{N \int (\log\abs{\lambda} - \log_\eta(\lambda))\hat{\mu}_{H_N(u)}(\diff \lambda)} \mathds{1}_{\mc{E}_{\textup{Lip}}(u)} \mathds{1}_{\mc{E}_{\textup{gap}}(u)} \mathds{1}_{\mc{E}_{b}(u)} \right]\right\} \\
    &\geq -\frac{p_b}{2}\log(1+e^{2N^\epsilon}\eta^2) - \frac{\eta^2}{2w_b^2} + \inf_{u \geq \gamma(p)} \frac{1}{N} \log \P(\mc{E}_{\textup{Lip}}(u), \mc{E}_{\textup{gap}}(u), \mc{E}_{b}(u)) \\
    &= \inf_{u \geq \gamma(p)} \frac{1}{N} \log \P(\mc{E}_{\textup{Lip}}(u), \mc{E}_{\textup{gap}}(u), \mc{E}_{b}(u)) + \oo(1).
\end{align*}
Among the random-matrix claims above, \eqref{eqn:dfixed_wegner} exactly reads $\lim_{N \to \infty} \inf_{\abs{u} \geq 2+\delta} \P(\mc{E}_{\textup{gap}}(u)) = 1$, and applying \eqref{eqn:dfixed_lipschitz_concentration} to $f = \log_\eta$ gives $\inf_{u \in \R} \P(\mc{E}_{\textup{Lip}}(u)) \geq 1 - C \exp(-4c \eta^2 t^2 N^2) \to 1$. So it remains only to show $\inf_{u \geq \gamma(p)} \P(\mc{E}_b(u)) \to 1$. We have 
\begin{align*}
    \sup_{u \geq \gamma(p)} \P(\mc{E}_b(u)^c) &\leq \frac{1}{p_b} \sup_{u \geq \gamma(p)} \int b(\lambda) \E[\hat{\mu}_{H_N(u)}](\diff \lambda) \\
    &\leq \frac{1}{p_bw_b} \sup_{u \geq \gamma(p)} {\rm W}_1(\E[\hat{\mu}_{H_N(u)}],\mu_\infty(u)) + \frac{1}{p_b} \sup_{u \geq \gamma(p)} \int b(\lambda) \mu_\infty(u,\lambda) \diff \lambda. 
\end{align*}
The first term is at most $\frac{1}{p_bw_b} N^{-\kappa} = N^{\frac{\kappa}{4} + \frac{\kappa^2}{8} - \kappa}$, which is $\oo(1)$ for $\kappa$ sufficiently small; the second term is actually zero as long as $\gamma(p) > 2$, since the two functions in the integrand have disjoint support ($b$ has support on a shrinking interval around zero, while $\mu_\infty$ has support on $[-u-2,-u+2]$). This finishes the proof of the lower bound \eqref{eqn:gammadp_lower_bound}, modulo the random-matrix claims \eqref{eqn:dfixed_wegner}, \eqref{eqn:dfixed_lipschitz_concentration}, and \eqref{eqn:w1_nkappa}.

Finally we prove the random-matrix claims. The no-small-eigenvalues estimate \eqref{eqn:dfixed_wegner} is proved exactly as in the case of $p$ fixed (compare to Lemma \ref{lem:wegner}), as is the Lipschitz concentration estimate \eqref{eqn:dfixed_lipschitz_concentration} (compare to \eqref{eqn:lsi}). However, \eqref{eqn:w1_nkappa} requires some modifications of the fixed-$p$ proofs.

We will use the local law of \cite{alt2019location}, in a slightly different sense than we did for $p$ fixed and $d \to \infty$. In the notation (2.1) there, we take $L = 1$, $N = p(d-1)$, $\ell = 1$, all the $a_i$'s to be $u$, and $\widetilde{\alpha}_1 = 1$, $\widetilde{\beta}_1 = \widetilde{\gamma_1} = 0$. We identify $\C^{1 \times 1} \otimes \C^{N \times N} \cong \C^{N \times N}$. We drop the $\mu$ and $\nu$ from their notation $s^\mu_{ij}$ and $t^\nu_{ij}$, since there are only one $\mu$ and $\nu$. We compute $t_{ij} \equiv 0$ and, partitioning $\llbracket 1, N \rrbracket = \sqcup_{\ell=1}^p I_\ell$ where $I_\ell = \llbracket(\ell-1) (d-1)+1,\ell(d-1)\rrbracket$, we have
\[
    s_{ij} = \frac{1}{N} \mathbf{1}\{i \text{ and } j \text{ are not in the same } I_\ell\}.
\]
The operators $\ms{S}_i$ defined\footnote{
Remark 2.5(v) in \cite{alt2019location} explains that the Hermitization is superfluous for Hermitian matrices like ours; thus we treat the un-Hermitized versions of these operators.
}
in eq. (2.4) of \cite{alt2019location} now depend on $i$, but only in a coarse-grained way: for $i \in \llbracket 1, N \rrbracket$, we have $\ms{S}_i : \C^N \to \C$ defined on $\mathbf{r} = (r_1,\ldots,r_N)$ as 
\[
    \ms{S}_i[\mathbf{r}] = \sum_{k=1}^N s_{ik}r_k = \frac{1}{N} \sum_{k : \lfloor \frac{k}{d-1} \rfloor \neq \lfloor \frac{i}{d-1} \rfloor} r_k.
\]
Then we seek, for each $z \in \mathbb{H}$, the unique solution $\mathbf{m}(z) = (m_1(z),\ldots,m_N(z)) \in \C^N$ to the constrained problem
\begin{equation}
\label{eqn:kronecker_mde_d_fixed}
\begin{split}
    &1 + (z+\ms{S}_i[\mathbf{m}(z)])m_i(z) = 0 \text{ for each } i, \\
    & \text{subject to } \Im m_i(z) > 0 \text{ for each } i.
\end{split}
\end{equation}
The same argument as in the fixed-$p$ case shows that the solution has the particularly simple form 
\[
    m_i(z) \equiv \widetilde{m}_p(z),
\]
independently of $d$, where $\widetilde{m}_p(z)$, given explicitly by \eqref{eqn:widetilde_m_d}, is the Stieltjes transform of the measures $\mu_p$. Then following the proof of \eqref{eqn:w1}, except with this local law, gives
\[
    \sup_{u \in \R} {\rm W}_1(\E[\hat{\mu}_{H_N(u)}],\mu_p(u)) \leq N^{-\kappa}
\]
(as before, ${\rm W}_1(\E[\hat{\mu}_{H_N(u)}],\mu_p(u))$ is independent of $u$). Plus, one can easily compute
\[
    {\rm W}_1(\mu_p(0),\mu_\infty(0)) = {\rm W}_1(\mu_p(u),\mu_\infty(u)) = \OO(1/p) = \OO(1/N)
\]
since $N = p(d-1)$ and $d$ is fixed, and we can add $\sup_{u \in \R}$ for free in the same way. This finishes the proof of \eqref{eqn:w1_nkappa}, and thus of Lemma \ref{lem:laplace_d_fixed_determinants}.
\end{proof}

\begin{proof}[Proof of Lemma \ref{lem:laplace_d_fixed_laplace}]
The function $\mc{S}[u] = \Omega(u) - \frac{u^2}{2}$ is continuous, and since $\lim_{u \to \infty} \frac{\Omega(u)}{\log(u)} = 1$, we note that, for every $\epsilon > 0$, there exists $K_\epsilon$ such that for all $u \geq K_\epsilon$ we have 
\[
    -\frac{1}{2} u^2 \leq \mc{S}[u] \leq \left(-\frac{1}{2}+\epsilon\right)u^2,
\]
which we will apply with fixed small $\epsilon$ to be chosen. Also, $\mc{S}$ is decreasing for $u > 0$. Thus, for some threshold $b_N \to \infty$ to be chosen, $b_N \geq \gamma(p)$, we can split $\int_{\gamma(p)}^\infty = \int_{\gamma(p)}^{b_N} + \int_{b_N}^\infty$, bounding the integrand by its pointwise maximum on the former interval and by the upper bound for $\mc{S}$ on the latter interval, using $\log(a+b) \leq \log(2\max(a,b))$ to find
\[
    \frac{1}{N} \log \int_{\gamma(p)}^\infty e^{N\mc{S}(u)} \diff u \leq \frac{\log(2)}{N} + \max\left\{ \mc{S}(\gamma_d^\R(p)) +  \frac{\log(b_N - \gamma(p))}{N}, \frac{1}{N} \log \int_{b_N}^\infty e^{N\left(-\frac{1}{2}+\epsilon\right)u^2} \diff u \right\}.
\]
The standard Gaussian tail bound $\int_t^\infty \frac{e^{-s^2/2}}{\sqrt{2\pi}} \diff s \leq e^{-t^2/2}$ for $t \geq 1$ gives
\[
    \int_{b_N}^\infty e^{N\left(-\frac{1}{2}+\epsilon\right)u^2} \diff u = \frac{\sqrt{2\pi}}{\sqrt{N(1-2\epsilon)}} \int_{\sqrt{N(1-2\epsilon)}b_N}^\infty \frac{e^{-u^2/2}}{\sqrt{2\pi}} \diff u \leq \frac{\sqrt{2\pi}}{\sqrt{N(1-2\epsilon)}} e^{-N(\frac{1}{2}-\epsilon)b_N^2}, %$y = \sqrt{N(1-2\epsilon)}u$
\]
so 
\[
    \frac{1}{N}\log \int_{b_N}^\infty e^{N\left(-\frac{1}{2}+\epsilon\right)u^2} \diff u \leq -\left(\frac{1}{2}-\epsilon\right)b_N^2 + \oo_{N \to \infty}(1).
\]
Now we choose $b_N = 2\gamma(p)$. On the one hand, our assumption $\gamma(p) = \oo(\exp(p^{1-\delta}))$ gives $$\lim_{N \to \infty} \frac{\log(b_N - \gamma(p))}{N} = 0.$$ On the other hand, with this choice we have $-(\frac{1}{2}-\epsilon)b_N^2 = -4(\frac{1}{2}-\epsilon)\gamma(p)^2 \leq -\frac{1}{2}\gamma(p)^2 \leq \mc{S}(\gamma(p))$, as long as we choose $\epsilon$ with $\frac{1}{2}-\epsilon > \frac{1}{8}$. Thus 
\[
    \lim_{p \to \infty} \left[ \frac{1}{p(d-1)} \log\int_{\gamma(p)}^\infty e^{N\mc{S}(u)} \diff u - \mc{S}(\gamma(p)) \right] \leq 0,
\]
which finishes the upper bound of Lemma \ref{lem:laplace_d_fixed_laplace}. For the lower bound, we use the sharper bounds
\[
    -\frac{u^2}{2}+\log(u)-\frac{1}{u} \leq \mc{S}[u] \leq -\frac{u^2}{2}+\log(u),
\] 
valid for $u \geq 1$, say. By assumption we have $\gamma(p) = \oo(\exp(N^{1-\delta}))$ for some $\delta$; with this same $\delta$, we set $\epsilon_N = \exp(-N^{1-\delta})$, and compute
\[
    \frac{1}{N} \log \int_{\gamma(p)}^\infty e^{N\mc{S}[u]} \diff u \geq \frac{1}{N} \log \int_{\gamma(p)}^{\gamma(p)+\epsilon_N} e^{N\mc{S}[u]} \diff u \geq \mc{S}[\gamma(p)] + (\mc{S}[\gamma(p)+\epsilon_N]-\mc{S}[\gamma(p)]) + \frac{\log \epsilon_N}{N}.
\]
The first term on the right-hand side is the one we want. By our choice of $\epsilon_N$, we have $(\log \epsilon_N)/N \to 0$. For the difference-of-$\mc{S}$ term, we compute
\begin{align*}
    \mc{S}[\gamma(p) + \epsilon_N] -\mc{S}[\gamma(p)] &\geq -\frac{(\gamma(p)+\epsilon_N)^2}{2} + \log(\gamma(p)+\epsilon_N) - \frac{1}{\gamma(p)+\epsilon_N} + \frac{\gamma(p)^2}{2} - \log(\gamma(p))\\
    &= -\gamma(p)\epsilon_N + \oo(1),
\end{align*}
and by assumption this tends to zero. Thus
\[
    \liminf_{p \to \infty} \left[ \frac{1}{p(d-1)} \log \int_{\gamma(p)} e^{-\frac{N}{2}u^2 + N\Omega(u)} \diff u - \mc{S}(\gamma(p)) \right] \geq 0,
\]
which finishes the lower bound of Lemma \ref{lem:laplace_d_fixed_laplace}.
\end{proof}

%%%%%%%%%%%%%%%%%%%%%%%%%%%%%%%%%%%%%%%%%%%%%%%%%%%%%%%%%%%%
%%%%%%%%%%%%%%%%%%%%%%%%%%%%%%%%%%%%%%%%%%%%%%%%%%%%%%%%%%%%
%%%%%%%%
%%%%%%%%             Section: Complex case
%%%%%%%%
%%%%%%%%%%%%%%%%%%%%%%%%%%%%%%%%%%%%%%%%%%%%%%%%%%%%%%%%%%%%
%%%%%%%%%%%%%%%%%%%%%%%%%%%%%%%%%%%%%%%%%%%%%%%%%%%%%%%%%%%%

\section{Complex case}\label{sec:complex_case}
The basic outline of the complex case is similar to the real case explored in the previous section, but in order to access the injective norm with the Kac--Rice formula, we have to work, not on the original product of complex spheres, but on a kind of quotient space thereof. This leads to additional structure in the matrices appearing in Kac--Rice.

We now let $K=\C$ so that $T$ is a complex Gaussian random tensor of order $p$ and size $d$, that is an array $(T_{i_1,i_2,\ldots,i_p})_{i_1,\ldots,i_p=1}^d$ where $T_{i_1,i_2,\ldots,i_p}$ are independent complex standard normal random variables (\textit{i.e.} their real and imaginary part are independent and identically distributed $\mathcal{N}(0,\frac12)$ random variables). We are interested in
\[
    n_T(p,d)\defeq\max_{x^{(i)}\in \mathbb{C}^d, \lnorm x^{(i)}\rnorm=1 \forall i \in [p]}\lvert \langle T, x^{(1)}\otimes \ldots \otimes x^{(p)}\rangle\rvert,
\]
that is we want to maximize the modulus of $\langle T, x^{(1)}\otimes \ldots \otimes x^{(p)}\rangle$ over the product $(\mathbb{S}_{\mathbb{C}}^{d-1})^{p}$ of $p$ complex spheres. As was implicit in equation \eqref{eq:relation_to_distance} we have 
\begin{equation}
    n_T(p,d)=\max_{x^{(i)}\in \mathbb{C}^d, \lnorm x^{(i)}\rnorm=1 \forall i \in [p]} \mathfrak{R}\langle T, x^{(1)}\otimes \ldots \otimes x^{(p)}\rangle,
\end{equation}
so that we might be tempted to study the Gaussian process $g_T:(\mathbb{S}_\C^{d-1})^{p}\rightarrow \R$ defined as 
$$g_T(x^{(1)},\ldots, x^{(p)})=\frac1{\sqrt{p(d-1)}}\mathfrak{R}\sum_{i_1,\ldots, i_d}T_{i_1,\ldots, i_p}x^{(1)}_{i_1}\ldots x^{(p)}_{i_p}$$
to access $n_T(p,d)$. %This idea reveals a bit naive as
However, this process is not Morse on $(\mathbb{S}_\C^{d-1})^{p}$, essentially because one can rotate the phases of the $x^{(i)}$'s in a continuous way (as opposed to the real case, where one could only flip the signs discretely). Indeed, let $(x^{(1)}_M,\ldots, x^{(p)}_M)$ be a maximizer so that $g_T(x^{(1)}_M,\ldots, x^{(p)}_M)=n_T(p,d)$; then, for example, $g_T(e^{i\theta}x^{(1)}_M,e^{-i\theta}x_M^{(2)},\ldots, x^{(p)}_M)=n_T(p,d)$ for all $\theta\in \R$. Thus critical points of $g_T$ on $(\mathbb{S}_\C^{d-1})^{p}$ are not isolated, preventing the application of the Kac--Rice formula. To get around this difficulty, we reduce the domain of definition of $g_T$ so that arbitrary global choice of phases of the $x^{(i)}$ is not possible anymore. To this aim we introduce
\begin{equation*}
    \mathbb{U}^{d-1}\defeq\{x=(x_1,\ldots, x_d)\in \mathbb{S}_{\mathbb{C}}^{d-1}: x_1\in \mathbb{R}\}.
\end{equation*}
Here, $\mathbb{U}^{d-1}\simeq \mathbb{S}^{2(d-1)}_\R$. We have the following proposition.
\begin{proposition}
$n_T(d,p)$ is attained by $g_T$ on $\mathbb{S}_\C^{d-1}\times (\mathbb{U}^{d-1})^{p-1}$.
\end{proposition}
\begin{proof}
Let $(x^{(1)}_M,\ldots, x^{(p)}_M)\in (\mathbb{S}_{\mathbb{C}}^{d-1})^{p}$ be a maximizer of $\lvert \langle T, x^{(1)}\otimes\ldots \otimes x^{(p)}\rangle\rvert$, and write $\theta\defeq\arg\left( \langle T, x^{(1)}_M\otimes\ldots \otimes x^{(p)}_M\rangle\right)$. For all $i\in \{2,\ldots,p\}$, set $\theta_i=\arg((x^{(i)}_M)_1)$ and $\theta_1=\theta-\sum_{i=2}^p\theta_i$. Define $\tilde{x}_M^{(j)} \defeq e^{-i\theta_j}x_M^{(j)}$ for $j = 1, \ldots, p$. 
One can check that $\langle T, \tilde{x}_M^{(1)}\otimes \ldots \otimes \tilde{x}_M^{(p)}\rangle$ is positive, and obviously $n_T(p,d)=\lvert\langle T, \tilde{x}_M^{(1)}\otimes \ldots \otimes \tilde{x}_M^{(p)}\rangle \rvert$, while by construction $(\tilde{x}_M^{(1)}, \ldots, \tilde{x}_M^{(p)})\in \mathbb{S}_\C^{d-1}\times (\mathbb{U}^{d-1})^{p-1}$, so we may conclude.
\end{proof}

%%%%%%%%%%%%%%%%%%%%%%%%%%%%%%%%%%%%%%%%%%%%%%%%%%%%%%%%%%%%
%%%%%%%%             Subsection: Kac--Rice formula
%%%%%%%%%%%%%%%%%%%%%%%%%%%%%%%%%%%%%%%%%%%%%%%%%%%%%%%%%%%%

\subsection{Kac--Rice formula} \label{sec:KR-formula-complex}
 The above proposition tells us that we can determine the injective norm of a complex random tensor $T$ by focusing on the minima and maxima of the real Gaussian process $g_T$ using the Kac--Rice formula, on the condition that we evaluate it on the manifold $\mathbb{S}_{\mathbb{C}}^{d-1}\times (\mathbb{U}^{d-1})^{\times (p-1)}$. This Gaussian process is centered and has covariance
\begin{equation}
\label{eqn:complex-covariance}
    \E(g_T(x^{(1)},\ldots,x^{(p)})g_T(y^{(1)},\ldots,y^{(p)}))=\frac1{2p(d-1)}\mathfrak{R}\left(\prod_{i=1}^p\langle x^{(i)},y^{(i)}\rangle \right),
\end{equation}
 where on the right-hand side we take, for concreteness, the inner product $\ip{x,y} = \sum_{i=1}^d x_i \overline{y_i}$.
 An important note is that this process is isotropic in the sense that its covariance is invariant under the action of $U(d)\times(\mathbb{Z}_2\times PU(d))^{p-1}$, where $U(d)$ denotes the unitary group and $PU(d)$ denotes the projective unitary group $U(d)/U(1)$ of unitary matrices up to a choice of phase.  
\begin{lemma}\label{lem:K-R_Complex}%\label{thm:K-R_Complex}
(Proved in Appendix \ref{app:KR-with-degeneracies})
Fix $p$ and $d$, denote $N=2p(d-1)+1$, and define 
$$L(p,d)\defeq\left(\sqrt{\frac{p(d-1)}{\pi}}\right)^{2p(d-1)}\frac{2^p\pi^{dp+\frac{1-p}{2}}}{\Gamma(d)\Gamma(\frac{2d-1}{2})^{p-1}}.$$
Then, whenever $D \subset \R$ is a finite union of intervals, we have
%we have for every $D\subset \mathbb{R}$ with positive Lebesgue measure
\[
    \E[\Crt_{g_T}(D)] \leq L(p,d)\int_D e^{-p(d-1) u^2}\E_{\textup{tBHGOE}}[|\det(W^\C_N-u)|] \diff u%\left( \lvert \det H_N^\C(u)\rvert\right) \diff u 
\]
where the expectation is taken over $W^\C_N \sim \textup{tBHGOE}(d-1,p)$. 
\end{lemma}

%%%%%%%%%%%%%%%%%%%%%%%%%%%%%%%%%%%%%%%%%%%%%%%%%%%%%%%%%%%%
%%%%%%%%     Subsection: Complex case: p fixed, d \to \infty
%%%%%%%%%%%%%%%%%%%%%%%%%%%%%%%%%%%%%%%%%%%%%%%%%%%%%%%%%%%%

\subsection{Complex case: \texorpdfstring{$p$}{p} fixed, \texorpdfstring{$d \to \infty$}{d to infinity}}

In this section we denote 
\[
    H^{\C}_N(u) = W^\C_N - u,
\]
where $W^\C_N \sim \textup{tBHGOE}(d-1,p)$.

\begin{theorem}\textbf{(Complex case, $p$ fixed, $d\to \infty$, upper bound)}
Fix $p \geq 2$, and let
\[
    \alpha(p) = E_0(p)\sqrt{p}.
\]
Then for every $\epsilon > 0$ we have
\begin{align}
    \limsup_{d \to \infty} \frac{1}{d} \log \P\left(\frac{n_T(p,d)}{\sqrt{d}} > \alpha(p) + \epsilon\right) < 0, \label{eqn:complex_nonsym_unnorm} \\
    \limsup_{d \to \infty} \frac{1}{d} \log \P\left( \frac{n_{\widetilde{T}}(p,d)}{\sqrt{d}} d^{p/2} > \alpha(p) + \epsilon\right) < 0. \label{eqn:complex_nonsym_norm}
\end{align}
\end{theorem}
Again, $\tilde T$ denotes the normalized version of the tensor. As earlier, this theorem is a simple consequence of the two key Lemmas \ref{lem:K-R_Complex} and \ref{lem:Complex_d_to_infinity_first_moment_laplace} (as well as Lemma \ref{lem:LauMas2000} to handle the normalization). We do not write the proof here since it is the same as in the real case.
\begin{lemma}
\label{lem:Complex_d_to_infinity_first_moment_laplace}
We have
\begin{equation}
\label{eqn:lpd_large_d}
    \lim_{d \to \infty} \frac{1}{2p(d-1)} \log L(p,d) = \frac{1+\log(p)}{2},
\end{equation}
and, for every $D \subset \mc{A}$ with positive Lebesgue measure that is the closure of its interior, 
\begin{equation}
\label{eqn:complex_d_to_infinity_laplace}
\begin{split}
    &\lim_{d \to \infty} \frac{1}{2p(d-1)} \log \left( \int_D \ e^{-p(d-1)u^2}\E_{\textup{tBHGOE}}\left[\lvert \det(W_N^\C-u)\vert \right] \diff u \right) \\
    &= \frac{\log(p-1)-\log(p)}{2} + \sup_{u \in D} \left\{ \Omega\left(u \sqrt{\frac{p}{p-1}}\right) - \frac{u^2}{2}\right\}.
\end{split}
\end{equation}
\end{lemma}
Again as earlier, this lemma follows from \cite[Theorem 4.1]{arous2022exponential} once the good properties of $H_N^\C(u) = W^\C_N - u$ are checked, which we do in Lemmas \ref{lem:complex_Wasserstein_and_Lipschitz_traces}, \ref{lem:complex_Wegner} and \ref{lem:complex_det_bound}. We prove those lemmas below by essentially following the proof structure of Section \ref{subsec:R_case_d_to_infinity}.
We start with the analogues of Lemmas \ref{lem:real-lsi} and \ref{lem:thm_1.2} above, whose extensions to the complex case we write below.

\begin{lemma}
\label{lem:complex-lsi}
There exists $C,c>0$ such that, for every Lipschitz $f:\mathbb{R}\rightarrow \mathbb{R}$, we have 
\begin{equation}
 \sup_{u\in \mathbb{R}}\P\left( \left\vert \frac1N\Tr(f(H_N^\C(u)))-\frac1N\E(\Tr(f(H_N^\C(u))))\right\vert >\delta\right)\le C\exp\left( -c\frac{N^2\delta^2}{\rnorm f \lnorm^2_{\textrm{Lip}}}\right).
\end{equation}
\end{lemma}
\begin{proof}
The same proof as in the real case shows that it suffices to consider $u = 0$, but it is no longer the case that the vector of upper-triangular entries of $W^\C_N$ satisfies the log-Sobolev inequality, because the entries are correlated. However, considering the representation of the $\textup{tBHGOE}$ from eq. \eqref{eqn:tBHGOE} (and the notation there), the vector $V$ that concatenates the upper-triangular entries of $B$, the upper-triangular entries of $C$, and the entries of $\theta$ does satisfy the log-Sobolev inequality with constant of order $\frac{1}{N}$, so the Herbst argument described in the real case goes through if we can show that, whenever $f : \R \to \R$ is Lipschitz with Lipschitz constant $\|f\|_{\textup{Lip}}$, the map from the entries of $V$ to $\Tr(f(W_N^\C))$ is Lipschitz with Lipschitz constant order $\sqrt{N}\|f\|_{\textup{Lip}}$. This is a short exercise using Hoffman--Wielandt, since if $V$ and $V'$ are such vectors inducing $W$ and $W'$, then from Cauchy--Schwarz we have
\[
    \abs{\Tr(f(W)) - \Tr(f(W'))} \leq \|f\|_{\textup{Lip}} \sum_{i=1}^N \abs{\lambda_i(W) - \lambda_i(W')} \leq \sqrt{N}\|f\|_{\textup{Lip}} \Tr((W-W')^2)
\]
and the latter is bounded above by four times the squared Euclidean norm of $V$ (since, because of the structure of $W$, each entry of $V$ appears at most four times in $W$). This finishes the proof.
\end{proof}

\begin{lemma}\label{lem:complex_Wasserstein_and_Lipschitz_traces}
There exists $\kappa>0$ such that
\begin{equation}
    \sup_{u\in \R}W_1(\E[\hat\mu_{H_N^\C}(u)],\mu_p(u))\le N^{-\kappa}.
\end{equation}
\end{lemma}
\begin{proof}

The proof follows similar guidelines as the proof of Lemma \ref{lem:thm_1.2}. However, there are a few technical differences due to the specific form of $W_N^\C$ in the complex case. As before, it suffices to consider the case $u = 0$, and to prove the analogues of equations (3.1), (3.2), and (3.3) of \cite{arous2022exponential}.

In the notation of the proof of Lemma \ref{lem:complex-lsi}, the map from $V$ to $\|W\|$ is has an order-one Lipschitz constant, since $\|W-W'\| \leq \Tr((W-W')^2)$; thus the Herbst argument gives
\[
    \P(\|W^\C_N\| \geq x) \leq e^{-Nx} \E[e^{N\|W^\C_N\|}] \leq e^{N(\sup_N \E\|W^\C_N\| + C - x)}
\]
for some $C$, and Lemma \ref{lem:est_op_norm_tBHGOE} shows that $\sup_N \E\|W_N^\C\| < \infty$. As explained in the real case, this suffices for (3.2) and (3.3). 

It remains only to show that $(3.1)$ of \cite{arous2022exponential} is satisfied. There are two ingredients. First, we define $\tilde W_N^\C$ obtained as the minor of $W_N^\C$ by discarding the last row and column. We do this because $\tilde W_N^\C$, unlike $W_N^\C$, fits into the framework of ``Kronecker random matrices'' as considered in \cite{alt2019location} (from which we will import the local law, as in the proof of the real case), because one can write $\tilde W_N^\C=\eta_0\otimes \tilde B+\eta_1\otimes \tilde C$ with both $\tilde B,\tilde C\sim\textrm{BHGOE}(d-1,p,\frac1{2p(d-1)})$ independent and 
\[
    \eta_0 =\begin{pmatrix}1&0\\0&-1\end{pmatrix}, \qquad \eta_1=\begin{pmatrix} 0&1\\1&0\end{pmatrix}.
\]
We can further decompose the matrix as
\begin{equation}\label{eq:Complex_tilde_Kronecker_Matrix}
\tilde W_N^\C=\sum_{s=0}^1\sum_{1\le i < j \le p}(\eta_s\otimes \beta_{(i,j)})\otimes Y_{(s,i,j)}+(\eta_s\otimes \beta_{(i,j)})^T\otimes Y_{(s,i,j)}^T,
\end{equation}
where the $Y_{(s,i,j)}$ are i.i.d. $(d-1)\times (d-1)$ random matrices with i.i.d. $\mc{N}(0,\frac{1}{2p(d-1)})$ entries, and the $\beta_{(i,j)}$ are the strict-upper-triangular $\beta_{(i<j)} = E_{ij}$ (i.e., the $p \times p$ matrix with a one in position $(i,j)$ and zeros everywhere else). In this way, one has, for $s=0$,  $$\sum_{1\le i < j \le p}(\eta_0\otimes \beta_{(i,j)})\otimes Y_{(0,i,j)}+(\eta_0\otimes \beta_{(i,j)})^T\otimes Y_{(0,i,j)}^T =\begin{pmatrix} \tilde{B} & 0\\ 0 & - \tilde{B}\end{pmatrix},$$
and for $s=1$,
$$\sum_{1\le i < j \le p}(\eta_1\otimes \beta_{(i,j)})\otimes Y_{(1,i,j)}+(\eta_1\otimes \beta_{(i,j)})^T\otimes Y_{(1,i,j)}^T=\begin{pmatrix} 0 & \tilde{C} \\ \tilde{C} & 0\end{pmatrix}$$
justifying equation \eqref{eq:Complex_tilde_Kronecker_Matrix}. We show that the corresponding MDE has deterministic solutions $\tilde m_p(z)\mathrm{Id}$, where, as in the real case, $\widetilde{m}_p(z)$ (written explicitly in \eqref{eqn:widetilde_m_d}) is the Stieltjes transform of the rescaled semicircle law $\mu_p$. The local law coming from this MDE allows us to bound $\E[\lvert \tilde s_p(z)-\tilde m_p(z)\rvert]$, where $\tilde s_p(z)$ is the Stieltjes transform of $\tilde W_N$. The second ingredient is a short general argument using resolvent identities to show that, deterministically, the Stieltjes transform of a matrix and its minor are close. We use this to compare $\tilde s_p(z)$ to $s_p(z)$, the Stieltjes transform of $W_N^\C$. In more details:
\begin{itemize}
    \item \emph{MDE for $\tilde W_N^\C$:} As defined in \cite{alt2019location}, the operators $\ms{S}_i^{\C}:(\C^{2\times 2}\otimes \C^{p\times p})^{d-1}\rightarrow \C^{2\times 2}\otimes \C^{p\times p}$ have the form
$$\ms{S}_i^{\C}[\mathbf{r}]=\frac1{2p(d-1)}\sum_{k=1}^{d-1}\sum_{s=0}^1\sum_{1\le m < n \le p}(\eta_s\otimes E_{mn} )r_k(\eta_s\otimes E_{nm})+(\eta_s\otimes E_{nm})r_k(\eta_s\otimes E_{mn}).$$
As before, $\ms{S}_i^{\C}$ do not depend on $i$. Similarly, if each $r_i$ is a multiple of the identity $r_i=c \textrm{Id}_{2\times 2}\otimes \textrm{Id}_{p\times p}=c \textrm{Id}_{2p\times 2p}$ we have 
$$\ms{S}_i^{\C}[(c\textrm{Id},\ldots,c\textrm{Id})]=\frac{d-1}{2p(d-1)}\sum_{s=0}^1\sum_{1\le m<n\le p}c\textrm{Id}_{2\times 2}\otimes(E_{mm}+E_{nn})=c\frac{(p-1)}{p} \textrm{Id}_{2p\times 2p}.$$
As in the proof of Lemma \ref{lem:thm_1.2}, this allows us to show that $\mathbf{m}(z) = (m_1(z),\ldots,m_{d-1}(z))$, the solution to the Matrix Dyson Equation
\begin{equation}
\label{eqn:kronecker_mde_Complex}
\begin{split}
    &\Id_{2p \times 2p} + (z\Id_{2p \times 2p} + \ms{S}_i[\mathbf{m}(z)])m_i(z) = 0, \quad \text{for each } i = 1, \ldots, d-1, \\
    &\text{subject to } \Im m_i(z) > 0 \text{ as a quadratic form,}
\end{split}
\end{equation}
is given by the constant solution $m_i(z) = \widetilde{m}_p(z)\Id_{2p \times 2p}$ for $i = 1, \ldots, d-1$, where $\widetilde{m}_p(z)$ is as in \eqref{eqn:widetilde_m_d}. From there, we obtain that the measure $\mu^\C_d$ defined as having $\tilde{m}_p$ as its Stieltjes transform, is independent of $d$, and is a shifted (by $-u$) and rescaled semicircle law, namely the measure $\mu_p(u)$ as in \eqref{eqn:def_mu_infty}. \\

\item \emph{Resolvent Identities:} Suppose that $M$ is any $N \times N$ matrix with resolvent $G = G(z)$ and Stieltjes transform $s(z) = \frac{1}{N}\Tr(G(z))$, and $\tilde M$ is its $(N-1) \times (N-1)$ principal minor, with resolvent $G^{(N)}$ and Stieltjes transform $\tilde s(z) = \frac{1}{N-1} \Tr G^{(N)}$. Standard resolvent identities (see, e.g., \cite[Lemma 3.5]{BenKno2016}) give
\[
    G_{ij} = G^{(N)}_{ij} + \frac{G_{iN}G_{Nj}}{G_{NN}}
\]
for any $i, j < N$. Applying this with the Ward identity $\sum_{i=1}^N \abs{G_{iN}(z)}^2 = \frac{\Im G_{NN}(z)}{\Im z}$, we obtain
\begin{align*}
    \abs{s(z) - \tilde{s}(z)} &= \frac{1}{N}\abs{\sum_{i=1}^N G_{ii}(z) - \left(1+\frac{1}{N-1}\right) \sum_{i=1}^{N-1} G^{(N)}_{ii}(z)} \\
    &= \frac{1}{N}\abs{G_{NN}(z) + \sum_{i=1}^{N-1} (G_{ii}(z) - G^{(N)}_{ii}(z)) - \tilde{s}(z)} \\
    &= \frac{1}{N}\abs{G_{NN}(z) + \sum_{i=1}^{N-1} \frac{G_{iN}(z)G_{Ni}(z)}{G_{NN}(z)} - \tilde{s}(z)} \\
    &\leq \frac{1}{N}\abs{G_{NN}(z)} + \frac{1}{N\abs{G_{NN}(z)}} \left( \frac{\Im G_{NN}(z)}{\Im z} \right) - \frac{1}{N}\abs{G_{NN}(z)} + \frac{1}{N}\abs{\tilde s(z)} \leq \frac{2}{N\Im z},
\end{align*}
where in the last inequality we used the deterministic bound $\abs{\tilde s(z)} \leq \frac{1}{\Im z}$.
\end{itemize}
Now we bound
\[
    \abs{\E[s_p(z)] - \widetilde{m}_p(z)} \leq \E[\abs{s_p(z) - \widetilde{s}_p(z)}] + \E[\abs{\widetilde{s}_p(z) - \widetilde{m}_p(z)}],
\]
estimate the first term on the right-hand side with the resolvent argument above, and estimate the second-term on the right-hand side with the local law of \cite[(B.5)]{alt2019location}, as in the proof of Lemma \ref{lem:thm_1.2}, to show that there exist $A, \epsilon_1, \epsilon_2 > 0$ such that
\[
    \int_{-3A}^{3A} \diff E \lvert \E[s_p(E+i N^{-\epsilon_1})] -\tilde m_p(E+i N^{-\epsilon_1})\rvert\le N^{-\epsilon_2},
\]
which is (3.1) of \cite{arous2022exponential}.
\end{proof}

\begin{lemma}[Complex Wegner estimate]\label{lem:complex_Wegner}
Let $p>2$, then for every $\epsilon>0$ and every closed $K\subset \mathcal A =\mc{A}_- \sqcup \mc{A}_+ \defeq (-\infty,-E_0(p)) \sqcup (E_0(p),+\infty).$
\begin{equation}
    \lim_{N\rightarrow \infty} \inf_{u\in K}\P\left(H_N^\C(u) \text{ has no eigenvalues in } [-e^{-N^\epsilon}, e^{-N^\epsilon}]\right)=1.
\end{equation}
\end{lemma}
\begin{proof}
The proof is the same as for Lemma \ref{lem:wegner}; one just uses Lemma \ref{lem:est_op_norm_tBHGOE} instead of Corollary \ref{cor:hgoe_bandeira}. 
\end{proof}
Finally, the determinant bound is obtained in the same way as in subsection \ref{subsec:R_case_d_to_infinity}, leading to the following.
\begin{lemma}\label{lem:complex_det_bound}
For each $N=2p(d-1)+1$, the map $u\mapsto H_N^{\C}(u)$ is entrywise continuous, and there exists $C>0$ such that
$\E[\lvert\det(H_N^\C(u)) \rvert]\le (C\max(\lvert u\rvert,1))^N.$
\end{lemma}

As explained before, the proof of Lemma \ref{lem:Complex_d_to_infinity_first_moment_laplace} follows as in the real case from Lemmas \ref{lem:complex_Wasserstein_and_Lipschitz_traces}, \ref{lem:complex_Wegner} and \ref{lem:complex_det_bound}.

%%%%%%%%%%%%%%%%%%%%%%%%%%%%%%%%%%%%%%%%%%%%%%%%%%%%%%%%%%%%
%%%%%%%%     Subsection: Complex case: p \to \infty, d fixed
%%%%%%%%%%%%%%%%%%%%%%%%%%%%%%%%%%%%%%%%%%%%%%%%%%%%%%%%%%%%

\subsection{Complex case: \texorpdfstring{$p\to \infty$}{p to infinity}, \texorpdfstring{$d$}{d} fixed}
The proofs are the same as in Section \ref{subsec:Rcase-p-infinity}, up to slight changes which lead to the following small adaptation of the upper bound from Theorem \ref{thm:real_p_infty_d_fixed} for the complex case.
\begin{theorem}\label{thm:complex_p_infty_d_fixed}\textbf{(Real, nonsymmetric, $p \to \infty$, d fixed, upper bound.)}
For every $\epsilon > 0$ and every $d \geq 2$, we have
\[
    \limsup_{p \to \infty} \frac{1}{p} \log \P\left( \frac{n_T(p,d)}{\sqrt{p(d-1)}} - \gamma_d^\C(p) > \frac{\epsilon}{\sqrt{\log p}}\right) < 0,
\]
while the normalized version reads 
\[
    \limsup_{p\to \infty}\frac1{p}\log \P\left(n_{\widetilde T}(p,d)>\sqrt{\frac{p(d-1)}{d^{p}}}\gamma_d^\C(p)+\epsilon\sqrt{\frac{p(d-1)}{d^{p}\log p}}\right)<0.
\] 
\end{theorem}
As in the proof of Theorem \ref{thm:d_to_infinity_first_moment}, these results follow from Markov's inequality and the complexity estimates as explained in the following lemmas (and as in the real case, one needs slightly tighter estimates on the normalization here than in the case $p$ fixed, but Lemma \ref{lem:LauMas2000} is much stronger than necessary), so we do not write the proof. The following lemmas are themselves adapted from the real case.

Lemma \ref{lem:K-R_Complex} goes through as written, but now we compute
\begin{equation}
\label{eqn:beta_d_complex}
\begin{split}
    \beta_\C(d) &\defeq \lim_{p \to \infty} \left[ \frac{1}{2p(d-1)} \log L(p,d) - \frac{\log p}{2} \right] = \frac{1}{2} \log\left(d-1\right)+\frac1{2(d-1)}\log\left(\frac{2\sqrt{\pi}}{\Gamma\left(\frac{2d-1}{2} \right)^2}\right).
\end{split}
\end{equation}
We note $\lim_{d\to \infty }\beta_{\C}(d)=\frac12$. We also obtain the following lemma, analogous to Lemma \ref{lem:real_gamma_d_existence}.
\begin{lemma}\label{lem:complex_gamma_d_existence}
Fix $d \geq 2$. For every $p$, there exists a unique $\gamma^{\C}_d(p) > 0$ satisfying
\[
    \frac{\log p}{2} + \beta_{\C}(d) + \Omega\left( \gamma^\C_d(p) \right) - \frac{\gamma_d^\C(p)^2}{2} = 0.
\]
For each $d$, 
\[
\gamma_d^\C(p)=\sqrt{\log p}+\frac{\log \log p}{2\sqrt{\log p}}+\frac{\beta_\C(d)}{\log p}+o\left(\frac1{\sqrt{\log p}}\right).
\]
\end{lemma}
\begin{proof}
The same arguments from Lemma \ref{lem:real_gamma_d_existence} apply here.
\end{proof} 

\begin{lemma}\label{lem:complex_laplace_d_fixed}
Fix $d$, and let $(\gamma(p))_{p=1}^\infty$ be some sequence tending to infinity with $p$, such that 
\[
    \gamma(p) = \oo(\exp(p^{1-\delta}))
\]
for some $\delta > 0$. Then
\[
    \lim_{p \to \infty} \left[ \frac{1}{2p(d-1)} \log \int_{\gamma(p)}^\infty e^{-p(d-1)u^2} \E_{\textup{~}}[\abs{\det(W_N-u)}] \diff u - \left( \Omega\left(\gamma(p)\right) - \frac{(\gamma(p))^2}{2} \right) \right] = 0.
\]

\end{lemma}
This is just the complex version of Lemma \ref{lem:laplace_d_fixed}, whose proof was split into Lemmas \ref{lem:laplace_d_fixed_determinants} and \ref{lem:laplace_d_fixed_laplace}. The proof of Lemma \ref{lem:laplace_d_fixed_determinants} really relied on the three random matrix claims, while the rest, although quite intricate, did not use the specifics of the random matrix under consideration. Similarly, the proof of Lemma \ref{lem:laplace_d_fixed_laplace} is really about Laplace's method, and can be repeated almost exactly in this complex case. (Notice that the factors of two work out, since $N = p(d-1)$ in the real case but $N = 2p(d-1)+1$ in the complex case.) For these reasons, we hope to convince the reader that the following is a sufficient proof of Lemma \ref{lem:complex_laplace_d_fixed}.
\begin{proof}[Proof of Lemma \ref{lem:complex_laplace_d_fixed}]
We recall the three random matrix claims:
\begin{itemize}
    \item For every $\epsilon, \delta > 0$, 
\begin{equation}
\label{eqn:dfixed_wegner_complex}
    \lim_{N \to \infty} \inf_{\abs{u} > 2+\delta} \P(H_N^\C(u) \text{ has no eigenvalues in } [-e^{-N^\epsilon}, e^{-N^\epsilon}]) = 1.
\end{equation}
\item There exist universal $C, c > 0$ such that, whenever $f : \R \to \R$ is Lipschitz,
\begin{equation}
\label{eqn:dfixed_lipschitz_concentration_complex}
    \sup_{u \in \R} \P\left( \abs{ \int_\R f(\lambda) (\hat{\mu}_{H_N^\C(u)} - \E[\hat{\mu}_{H_N^\C(u)}])(\diff \lambda) } > \delta\right) \leq C\exp\left(-\frac{cN^2\delta^2}{\|f\|_{\textup{Lip}}^2} \right).
\end{equation}
\item There exists (small) $\kappa > 0$ such that
\begin{equation}
\label{eqn:w1_nkappa_complex}
    \sup_{u \in \R} {\rm W}_1(\E[\hat{\mu}_{H_N^\C(u)}],\mu_\infty(u)) \leq N^{-\kappa}.
\end{equation}
\end{itemize}
The first claim \eqref{eqn:dfixed_wegner_complex} is % just a consequence of
proved as in Lemma \ref{lem:complex_Wegner}. The concentration of Lipschitz traces \eqref{eqn:dfixed_lipschitz_concentration_complex} actually already appeared as Lemma \ref{lem:complex-lsi}. Finally, the bound \eqref{eqn:w1_nkappa_complex} on the Wasserstein distance comes from:
\begin{itemize}
    \item \emph{Checking (3.2) and (3.3) of Proposition 3.1 of \cite{arous2022exponential}}: As in the opposite scaling limit, this is a consequence of Lemma \ref{lem:est_op_norm_tBHGOE} and Herbst argument.
    \item \emph{MDE for large $p$:} We use the MDE for the matrix $\tilde{W}_N^{\C}$, where we define the MDE as in the proof of Lemma \ref{lem:laplace_d_fixed_determinants}. We can therefore again use the local law \cite[(B.5)]{alt2019location} which controls $\int_{-3A}^{3A}\E(\lvert \tilde{s}_N(z)-\tilde{m}_p(z)\rvert)\diff E$.
    \item \emph{Resolvent identities:} We want to bound $\lvert s_N(z) - \tilde{s}_N(z)\rvert$. To this aim we can use the exact same resolvent identities we used in the proof of Lemma \ref{lem:complex_Wasserstein_and_Lipschitz_traces}, from which we obtain $\lvert s_N(z) - \tilde{s}_N(z)\rvert\le \frac{2}{N\Im z}$.
\end{itemize}
As in the opposite scaling limit, we show $\sup_{u \in \R} {\rm W}_1(\E[\hat{\mu}_{H_N^\C(u)}],\mu_p(u)) \leq N^{-\kappa}$ by combining the second two points. Together with the remark that $\sup_{u \in \R} {\rm W}_1(\mu_p, \mu_\infty)=\OO(1/p)$, this shows \eqref{eqn:w1_nkappa_complex}. As previously explained, the rest of the proof is just a copy of the arguments leading to Lemma \ref{lem:laplace_d_fixed_determinants} and \ref{lem:laplace_d_fixed_laplace}.
\end{proof}

\appendix

%%%%%%%%%%%%%%%%%%%%%%%%%%%%%%%%%%%%%%%%%%%%%%%%%%%%%%%%%%%%
%%%%%%%%%%%%%%%%%%%%%%%%%%%%%%%%%%%%%%%%%%%%%%%%%%%%%%%%%%%%
%%%%%%%%
%%%%%%%%             Section: Kac--Rice with degeneracies
%%%%%%%%
%%%%%%%%%%%%%%%%%%%%%%%%%%%%%%%%%%%%%%%%%%%%%%%%%%%%%%%%%%%%
%%%%%%%%%%%%%%%%%%%%%%%%%%%%%%%%%%%%%%%%%%%%%%%%%%%%%%%%%%%%

\section{Kac--Rice with degeneracies}\label{app:KR-with-degeneracies}

Because of the zero blocks, our Hessians are degenerate in the sense that the joint law of their upper-triangular entries is supported on a lower-dimensional space. Thus the textbook versions of the Kac--Rice formula do not apply, and we need to adapt them. Adaptations for degeneracies have previously appeared in the literature, for example in \cite{FanMeiMon2021}, although the degeneracies are slightly different: in \cite{FanMeiMon2021} the gradient and Hessian are each non-degenerate but are too correlated to have a joint density, whereas in our case the Hessian itself is degenerate but is still independent of the gradient at a point.

%%%%%%%%%%%%%%%%%%%%%%%%%%%%%%%%%%%%%%%%%%%%%%%%%%%%%%%%%%%%
%%%%%%%%             Subsection: Real case
%%%%%%%%%%%%%%%%%%%%%%%%%%%%%%%%%%%%%%%%%%%%%%%%%%%%%%%%%%%%

\subsection{Real case}

We will frequently use that the process is isotropic, and take our favorite representative point to be the $p$-tuple north pole
\[
    \mathbf{n} \defeq \overbrace{(1,\underbrace{0,\ldots,0}_{d-1},1,\underbrace{0,\ldots,0}_{d-1},\ldots,1,\underbrace{0,\ldots,0}_{d-1})}^{p \text{ times}}.
\]
We start by computing the joint distribution of the components of the covariant gradient associated to the Levi-Civita connection, which we call Riemannian gradient and Hessian at a point.

\begin{lemma}
\label{lem:real-covariance-computation}
For $x \in (\mathbb{S}^{d-1})^{p}$, let $(f_{(a,i)}(x))_{a \leq p, 2\le i \leq d}$ be the Riemannian gradient of $f_T$ at $x$ (i.e., the derivative in the $i$th component of the $p$th factor), and let $(f_{(a,i),(b,j)}(x))_{a,b \leq p, 2\le i, j \leq d}$ be the Riemannian Hessian of $f_T$ at $x$. Then $f_T(x)$, $(f_{(a,i)}(x))_{a \leq p, 2\le i \leq d}$, and $(f_{(a,i),(b,j)}(x))_{a, b \leq p,2\le  i, j \leq d}$ are centered, jointly Gaussian variables with covariance structure
\begin{align*}
	\E[f_T(x)^2] =& \, \frac{1}{N}, \\
	\E[f_T(x) f_{(a,i)} (x)] = \E[f_{(a,i)} (x) f_{(b,j),(c,k)}(x)] =& \, 0, \\
	\E[f_{(a,i)}(x) f_{(b,j)}(x)] =& \, \frac{1}{N}\delta_{ab}\delta_{ij}, \\
	\E[f_T(x)f_{(a,i),(b,j)}(x)] =& \, - \frac{1}{N}\delta_{ab}\delta_{ij}, \\
	\E[f_{(a,i),(b,j)}(x)f_{(c,k),(d,\ell)}(x)] =& \, \frac{1}{N}\bigg( \delta_{ab}\delta_{cd}\delta_{ij}\delta_{k\ell} + \delta_{ac}\delta_{bd}\delta_{ik}\delta_{j\ell} + \delta_{ad}\delta_{bc}\delta_{i\ell}\delta_{jk} \\
	&- \delta_{a=b=c=d}(\delta_{ik}\delta_{j\ell} + \delta_{i\ell}\delta_{jk}) \bigg),
\end{align*}
as well as the conditional distributions
\begin{align*}
	\E[f_{(a,i),(b,j)}(x) \vert f_T(x) = u] =& \, -\delta_{ab}\delta_{ij}u, \\
	\Cov[f_{(a,i),(b,j)}(x), f_{(c,k),(d,\ell)}(x) \vert f_T(x) = u] =& \, \frac{1}{N} \left( \delta_{ac}\delta_{bd}\delta_{ik}\delta_{j\ell} + \delta_{ad}\delta_{bc}\delta_{i\ell}\delta_{jk} \right. \\
 & \left.- \delta_{a=b=c=d}(\delta_{ik}\delta_{j\ell}+\delta_{i\ell}\delta_{jk}) \right)
\end{align*}
In other words, conditionally on $\{f_T(x) = u\}$, the random matrix $(f_{(a,i),(b,j)}(x))_{a, b \leq p, 2\le i,j \leq d}$ has the same distribution as
\[
    \textup{BHGOE}(d-1,p) - u\Id
\]
(recall Definition \ref{def:BHGOE}).
\end{lemma}
\begin{proof}[Proof of Lemma \ref{lem:real-covariance-computation}]
Since the process is isotropic, it suffices to take $x = \mathbf{n}$. We define $\Psi : (\mathbb{S}^{d-1})^p \to (\R^{d-1})^p$ by
\[
	\Psi(x^{(1)}_1,\ldots,x^{(1)}_d,x^{(2)}_1, \ldots, x^{(2)}_d, \ldots, x^{(p)}_1, \ldots, x^{(p)}_d) = (x^{(1)}_2,\ldots,x^{(1)}_{d},x^{(2)}_2, \ldots, x^{(2)}_{d}, \ldots, x^{(p)}_2, \ldots, x^{(p)}_{d}),
\]
which is a chart in some neighborhood $U$ of $\mathbf{n}$. Then
\[
	\overline{f} = f_T \circ \Psi^{-1} 
\]
is a centered Gaussian process on $\Psi(U)$ with covariance
\[
	C_T(x,y) = \frac{1}{N} \prod_{j=1}^p \left\{ \sum_{i=2}^{d} x^{(j)}_i y^{(j)}_i + \sqrt{\left(1-\sum_{i=2}^{d} (x^{(j)}_i)^2\right) \left(1-\sum_{i=2}^{d} (y^{(j)}_i)^2\right)} \right\}.
\]
We can compute formulas for the covariance structure of $\overline{f}(0)$, $\overline{f}_{(a,i)}(0)$, and $\overline{f}_{(a,i),(b,j)}(0)$ using the identity
\[
	\Cov\left(\frac{\partial^k \overline{f}(x)}{\partial x_{i_1} \ldots \partial x_{i_k}}, \frac{\partial^\ell \overline{f}(y)}{\partial y_{j_1} \ldots \partial y_{j_\ell}} \right) = \frac{\partial^{k+\ell} C_T(x,y)}{\partial x_{i_1} \ldots \partial x_{i_k} \partial y_{j_1} \ldots \partial y_{j_\ell}},
\]
which immediately transfer to the stated formulas for the original field because the covariant derivatives for the Levi-Civita connection of $f_T$ agrees with the Euclidean derivatives of $\overline{f}$ at $0$. The Gaussian conditioning is standard.
\end{proof}
Let $U\subset (\mathbb{S}^{d-1})^p$ be an open set, $f:(\mathbb{S}^{d-1})^p\rightarrow \R$ and $D\subset \R$ open. We denote
$$\Crt_f(D,U)\defeq\{x\in U:\nabla f(x)=0, f(x)\in D\}.$$
This naturally extends the notation used in Section \ref{sec:KR-formula-real} for the number of critical points on the full $(\mathbb{S}^{d-1})^p$ to a restriction over an open set of $(\mathbb{S}^{d-1})^p$.
\begin{lemma}
\label{lem:real-Kac--Rice-mollified}
Let $(U,\phi)$ be a chart for $(\mathbb{S}^{d-1})^p$, and write $S = \phi(U) \subset \R^N$. Let $f : (\mathbb{S}^{d-1})^p \to \R$. Let $D \subset \R$ be open. Assume the following:
\begin{enumerate}
\item $f \circ \phi^{-1}$, $\nabla f \circ \phi^{-1}$, and $\nabla^2 f \circ \phi^{-1}$ are continuous on $\overline{S}$, the closure of $S$.
\item There are no points $s \in \overline{S}$ such that both $(\nabla f \circ \phi^{-1})(s) = 0$ and either $(f \circ \phi^{-1})(s) \in \partial D$ or $\det (\nabla^2 f \circ \phi^{-1})(s) = 0$.
\item There are no points $s \in \partial S$ such that $(\nabla f \circ \phi^{-1})(s) = 0$.
\end{enumerate}
Then we have 
\[
    \Crt_f(D,U) = \lim_{\epsilon \to 0} \int_U \diff\mathrm{Vol}(x) \delta_\epsilon(\nabla f(x)) \mathbbm{1}_{f(x) \in D} \abs{\det \nabla^2f(x)}.
\]
\end{lemma}
\begin{proof}
This follows from applying \cite[Theorem 11.2.3]{adler2007random} in charts.
\end{proof}

\begin{lemma}
\label{lem:BHGOE-singular}
Let $W \sim \textup{BHGOE}(d,p)$. For each $u \in \R$, we have $\P(\det(W - u\Id) = 0) = 0$.
\end{lemma}
\begin{proof}
Consider the space of all upper-triangular entries of $\textup{BHGOE}(d,p)$ which are not structurally zero, which can be thought of as $\R^k$ with $k = \frac{(p-1)(p-2)d^2}{2}$. On the one hand, the deterministic map from this space to the determinant of the corresponding $W-u$ is some polynomial, so its zero set has zero Lebesgue measure unless this polynomial is the zero polynomial. On the other hand, the law of BHGOE on this space is absolutely continuous with respect to Lebesgue measure, so it assigns probability zero to any set of Lebesgue measure zero. Thus it suffices to show that the determinant of $W-u$ is not the zero polynomial on this space, and to do this it suffices to exhibit any point in $\R^k$ at which this polynomial does not vanish. One such example consists of the matrix all of whose off-diagonal blocks are $\Id_{d \times d}$, since this matrix, minus $u\Id_{pd \times pd}$, can be seen to have linearly independent columns as long as $u \neq -1$ (for $u = -1$, we can just, e.g., make the $(1,2)$ and the $(2,1)$ blocks both $2\Id_{d \times d}$ instead of $\Id_{d \times d}$). 
\end{proof}

\begin{lemma}
\label{lem:real-Kac--Rice-almost-surely}
For any chart $(U,\phi)$ on $(\mathbb{S}^{d-1})^p$, the assumptions of Lemma \ref{lem:real-Kac--Rice-mollified} hold with probability one for $f_T$.
\end{lemma}
\begin{proof}[Proof of Lemma \ref{lem:real-Kac--Rice-almost-surely}]
Continuity is clear. The remainder of the proof goes as in \cite[Lemma A.3]{FanMeiMon2021}; since the full details are given there, we just sketch the argument here.

First we show that there are no critical points on the boundary. Since $f_T$ is a polynomial, on the event $\mc{E}_M$ that the underlying tensor $T$ has all entries at most $M$, the gradient is Lipschitz with some Lipschitz constant $L = L(M,d,p)$. Since the boundary $\partial S$ has Lebesgue measure zero, in particular Lebesgue outer measure zero in $N$ dimensions, for any $\epsilon > 0$ it can be covered by a countable collection of balls $B_i$, with respective centers $s_i$ and radii $r_i$, such that $\sum_{i=1}^\infty r_i^N \leq \epsilon$. On the event $\mc{E}_M$, if there is some critical point on the boundary, it is contained in some $B_i$, and therefore the gradient at the pullback of $s_i$ is at most $Lr_i$ in magnitude. Since the gradient at each point is a centered Gaussian with covariance matrix $1/N\Id$, this probability is at most $r_i^N$ up to constants (which depend on $L$ but not $r$); since $\sum_i r_i^N \leq \epsilon$, we can then apply the union bound over the balls to show that the probability that there is any critical point on the boundary is at most $\epsilon$ times some constant not depending on $\epsilon$. Taking $\epsilon \to 0$, we find that the probability that there is any critical point on the boundary is bounded by $\P(\mc{E}_M^c)$. Taking $M \to \infty$ finishes the proof.

Next, we check that there are no critical points at which the Hessian has determinant zero. We cover the surface of $S$ with balls $\{B(s_i,r)\}_{i \in \mathcal{I}}$ with common radius $r$ and varying centers $s_i$. By a volume argument, we can take $\abs{\mathcal{I}}$ to be order $1/r^N$ (up to constants not depending on $r$). On the same event $\mc{E}_M$ as before, that all entries of the tensor are at most $M$ in absolute value, the gradient and the determinant of the Hessian are both Lipschitz, with some Lipschitz constant $L = L(M,d,p)$. Thus, by union-bounding over the balls, 
\begin{align*}
    &\P(\mc{E}_M \cap \{\exists \, s \in S : (\nabla f_T \circ \phi^{-1})(s) = 0, \det (\nabla^2 f_T \circ \phi^{-1})(s) = 0\}) \\
    &\leq \sum_{i \in \mathcal{I}} \P(\mc{E}_M, \|\nabla f_T \circ \phi^{-1}(s_i)\| \leq Lr, \abs{\det(\nabla^2 f \circ \phi^{-1})(s_i)} \leq Lr) \\
    &\lesssim \frac{1}{r^N} \P(\|\nabla f_T \circ \phi^{-1}(s_i)\| \leq Lr, \abs{\det(\nabla^2 f \circ \phi^{-1})(s_i)} \leq Lr) \\
    &= \frac{1}{r^N} \P(\|\nabla f_T \circ \phi^{-1}(s_i)\| \leq Lr) \P( \abs{\det(\nabla^2 f \circ \phi^{-1})(s_i)} \leq Lr)
\end{align*}
where we discarded the event $\mc{E}_M$ in the upper bound and then used that the gradient and Hessian are independent at each point. As before, $\P(\|\nabla f_T \circ \phi^{-1}(s_i)\| \leq Lr)$ scales like $r^N$ up to constants not depending on $r$, so the right-hand side is order $\P(\abs{\det(\nabla^2 f \circ \phi^{-1})(s_i)} \leq Lr)$. When $r \to 0$, the right-hand side thus tends to some constant times $\P(\abs{\det (\nabla^2 f \circ \phi^{-1})(s_i)} = 0)$. By conditioning on the field value and applying Lemma \ref{lem:BHGOE-singular}, we see that this vanishes. Taking $M \to \infty$ finishes the proof.

The argument for no critical points at which the field takes values in $\partial D$ is analogous, since the gradient is also independent of the field value at a point, and the probability of the field taking some exact value is still zero (and $\partial D$ is a finite set since $D$ is a finite union of intervals). This finishes the proof.
\end{proof}

\begin{lemma}
\label{lem:real-Kac--Rice-before-rmt}
We have 
\[
    \E[\Crt_{f_T}(D)] \leq \text{Vol}((\mathbb{S}^{d-1})^p) \E[ |\det \nabla^2 f_T(\mathbf{n}) | \mathbbm{1}_{f_T(\mathbf{n}) \in D}] \phi_{\nabla f_T(\mathbf{n})}(0)
\]
\end{lemma}

\begin{proof}[Proof of Lemma \ref{lem:real-Kac--Rice-before-rmt}]
For any finite atlas $(U_i,\phi_i)_{i \in I}$, by combining Lemmas \ref{lem:real-Kac--Rice-mollified} and \ref{lem:real-Kac--Rice-almost-surely} with Fatou's lemma and Fubini's theorem (since the integrand is nonnegative), we find
\begin{align*}
    \E[\Crt_{f_T}(D,U_i)] &= \E\left[ \lim_{\epsilon\rightarrow 0}\int_{U_i}\diff\mathrm{Vol}(x)\delta_\epsilon(\nabla f_T(x))\mathbbm{1}_{f_T(x)\in D}\left\lvert \det \nabla^2f_T(x) \right\rvert \right] \\
    &\leq \liminf_{\epsilon\rightarrow 0} \E\left[ \int_{U_i}\diff\mathrm{Vol}(x)\delta_\epsilon(\nabla f_T(x))\mathbbm{1}_{f_T(x)\in D}\left\lvert \det \nabla^2f_T(x) \right\rvert \right] \\
    &= \liminf_{\epsilon\rightarrow 0} \int_{U_i}\diff\mathrm{Vol}(x)  \E\left[\delta_\epsilon(\nabla f_T(x))\mathbbm{1}_{f_T(x)\in D}\left\lvert \det \nabla^2f_T(x) \right\rvert \right]
\end{align*}
for each $i$. Since the process is isotropic, the inner expectation does not depend on $x \in U_i$, so we can pull it out (and replace $x$ with any point on the sphere, not necessarily in $U_i$, say the north pole $\mathbf{n}$ for concreteness) and just pick up a volume factor $\text{Vol}(U_i)$. Isotropy additionally implies that, at each point, the gradient is independent of the Hessian and the field, so the expectation factors:
\begin{align*}
    \E[\Crt_{f_T}(D,U_i)] &\leq \text{Vol}(U_i) \E[ |\det \nabla^2 f_T(\mathbf{n}) | \mathbbm{1}_{f_T(x) \in D}] \left( \liminf_{\epsilon \rightarrow 0} \E[\delta_\epsilon(\nabla f_T(\mathbf{n})] \right) \\
    &= \text{Vol}(U_i) \E[ |\det \nabla^2 f_T(\mathbf{n}) | \mathbbm{1}_{f_T(x) \in D}] \phi_{\nabla f_T(\mathbf{n})}(0),
\end{align*}
where the last equality is a short exercise using that $\nabla f_T(\mathbf{n})$ is a vector of $N$ i.i.d. Gaussians. Since the number of critical points is nonnegative and monotonic on nested sets, for any finite atlas $(U_i,\phi_i)_{i \in I}$ we find
\[
    \E[\Crt_{f_T}(D)] \leq \sum_{i \in I} \E[\Crt_{f_T}(D,U_i)] \leq \left( \sum_{i \in I} \text{Vol}(U_i) \right) \E[ |\det \nabla^2 f_T(\mathbf{n}) | \mathbbm{1}_{f_T(\mathbb{n}) \in D}] \phi_{\nabla f_T(\mathbf{n})}(0).
\]
By choosing a sequence of atlases such that $\sum_{i \in I} \text{Vol}(U_i)$ tends to $\text{Vol}((\mathbb{S}^{d-1})^p)$ (which can even be done explicitly since the limit is a product of spheres), we obtain the result.
\end{proof}

\begin{proof}[Proof of Lemma \ref{lem:exact_kac_rice}]
This essentially just combines Lemmas \ref{lem:real-Kac--Rice-before-rmt} and \ref{lem:real-covariance-computation}: The latter gives 
\[
    \phi_{\nabla f_T(\mathbf{n})}(0) = \frac{1}{(2\pi)^{N/2}\left(\frac{1}{N}\right)^{N/2}},
\]
and of course $\text{Vol}((\mathbb{S}^{d-1})^p) = \left(\frac{2\pi^{d/2}}{\Gamma(d/2)}\right)^p$. Then by conditioning on $f_T(\mathbf{n})$ in the remaining expectation we find
\[
    \E[\Crt_{f_T}(D)] \leq \left(\frac{2\pi^{d/2}}{\Gamma(d/2)}\right)^p \frac{1}{(2\pi)^{N/2}\left(\frac{1}{N}\right)^{N/2}} \sqrt{\frac{N}{2\pi}} \int_{D} e^{-N\frac{u^2}{2}} \E_{\textup{BHGOE}}[\abs{\det(W_N - u)}] \diff u
\]
which is the result.
\end{proof}

%%%%%%%%%%%%%%%%%%%%%%%%%%%%%%%%%%%%%%%%%%%%%%%%%%%%%%%%%%%%
%%%%%%%%             Subsection: Complex case
%%%%%%%%%%%%%%%%%%%%%%%%%%%%%%%%%%%%%%%%%%%%%%%%%%%%%%%%%%%%

\subsection{Complex case}

This is the result of simple but heavy computations. We apply Kac--Rice formula to the random function $g_T:\mathbb{S}_{\mathbb{C}}^{d-1}\times (\mathbb{U}^{d-1})^{\times (p-1)}\rightarrow \mathbb{R}$. To this aim one has to pick a real chart, which we do by expressing everything in terms of the real and imaginary parts of the vectors $x^{(i)}$. More explicitly, we consider the $p$-tuple north pole
\[
    \mathbf{n} \defeq (1,\underbrace{0,\ldots,0}_{2d-1},\overbrace{1,\underbrace{0,\ldots,0}_{2d-2},1,\underbrace{0,\ldots,0}_{2d-2},\ldots,1,\underbrace{0,\ldots,0}_{2d-2}}^{p-1 \text{ times}}).
\]

We consider the map $\Psi:\mathbb{S}_{\mathbb{C}}^{d-1}\times (\mathbb{U}^{d-1})^{\times (p-1)}\rightarrow \mathbb{R}^{2p(d-1)+1}$ defined by
\begin{equation}
\label{eqn:def_psi_complex_KR}
    \Psi=(\psi_1, \psi_2,\ldots,\psi_p), \quad \psi_1:\mathbb{S}^{2d-1}\rightarrow \mathbb{R}^{2d-1} \textrm{ and } \psi_{i\ge 2}:\mathbb{S}^{2d-2}\rightarrow \mathbb{R}^{2d-2},
\end{equation}
with 
\begin{align}
\psi_1\left( \begin{pmatrix}x^{(1)}_{1,R} \\ x^{(1)}_{1,I}\end{pmatrix},\begin{pmatrix}x^{(1)}_{2,R}\\x^{(1)}_{2,I}\end{pmatrix},\ldots, \begin{pmatrix}x^{(1)}_{d,R}\\x^{(1)}_{d,I}\end{pmatrix}\right)&=(x^{(1)}_{1,I},(x^{(1)}_{2,R},x^{(1)}_{2,I}),\ldots, (x^{(1)}_{d,R},x^{(1)}_{d,I}))\\
\psi_{i\ge 2}\left(\left(x^{(i)}_{1,R},\begin{pmatrix}x^{(i)}_{2,R}\\x^{(i)}_{2,I}\end{pmatrix}\ldots, \begin{pmatrix}x^{(i)}_{d,R}\\x^{(i)}_{d,I}\end{pmatrix}\right) \right)&=((x^{(i)}_{2,R},x^{(i)}_{2,I}),\ldots, (x^{(i)}_{d,R},x^{(i)}_{d,I}))
\end{align}
where $x^{(k)}_{j,R}$ denotes the real part of the $j^{th}$ component of $x^{(k)}$ and $x^{(k)}_{j,I}$ its imaginary part. This is a chart in some neighborhood $U$ of $\mathbf{n}$, with image $\Psi(U) \subset \R^{2p(d-1)+1} = \R^N$. As suggested by the above notation, we index points in $\Psi(U)$ as 
\begin{align*}
    x &= (x^{(1)},\ldots,x^{(p)}) \\
    &= (x^{(1)}_{1,I},(x^{(1)}_{2,R},x^{(1)}_{2,I}),\ldots, (x^{(1)}_{d,R},x^{(1)}_{d,I}),(x^{(2)}_{2,R},x^{(2)}_{2,I}),\ldots, (x^{(2)}_{d,R},x^{(2)}_{d,I}),\ldots,(x^{(p)}_{2,R},x^{(p)}_{2,I}),\ldots, (x^{(p)}_{d,R},x^{(p)}_{d,I})).
\end{align*}

In the following we will need the set
\[
    \mc{G} = \left\{(1,i,Q) : 1 \leq i \leq d, Q \in \begin{cases} \{I\} & \text{if } i = 1, \\ \{I,R\} & \text{if } i \geq 2 \end{cases}\right\} \cup \{(k,i,Q) : 2 \leq k \leq p, 2 \leq i \leq d, Q \in \{I,R\}\},
\]
of effective coordinates on $U$ (i.e., of $(k,i,Q)$ such that $x^{(k)}_{i,Q}$ appears in the above representation), which has cardinality $1 + 2(d-1) + (p-1)(d-1)2 = 2p(d-1)+1 = N$. 
\begin{lemma}
\label{lem:complex-covariance-computation}
For $x \in \mathbb{S}_\C^{d-1} \times (\mathbb{U}^{d-1})^{\times (p-1)}$, let $(g_{(k,i,Q)}(x))_{(k,i,Q) \in \mc{G}}$ be the Riemannian gradient of $g_T$ at $x$, and let $(g_{(k,i,Q),(k',i',Q')}(x))_{(k,i,Q),(k',i',Q') \in \mc{G}}$ be the Riemannian Hessian of $g_T$ at $x$. Then $g_T(x)$, $(g_{(k,i,Q)}(x))_{(k,i,Q) \in \mc{G}}$, and $(g_{(k,i,Q),(k',i',Q')}(x))_{(k,i,Q),(k',i',Q') \in \mc{G}}$ are centered, jointly Gaussian variables with covariance structure
\begin{align*}
    &\E[g_T(x)^2] = \frac{1}{2p(d-1)} = \frac{1}{N-1}, \\
    &\E[g_T(x)g_{(k,i,Q)}(x)] = \E[g_{(k,i,Q)}(x) g_{(k',i',Q'),(k'',i'',Q'')}(x)] = 0, \\
    &\E[g_{(k,i,Q)}(x)g_{(k',i',Q')}(x)] = \frac{\delta_{(k,i,Q),(k',i',Q')}}{2p(d-1)} = \frac{\delta_{(k,i,Q),(k',i',Q')}}{N-1} \\
    &\E[g_T(x)g_{(k,i,Q),(k',i',Q')}(x)] = -\frac{\delta_{(k,i,Q),(k',i',Q')}}{2p(d-1)} = -\frac{\delta_{(k,i,Q),(k',i',Q')}}{N-1} \\
    &\E[g_{(k,i,Q),(k',i',Q')}(x)g_{(k'',i'',Q''),(k''',i''',Q''')}(x)] \\
    &= \begin{cases} \frac{1}{N-1}\delta_{(i,Q),(i',Q')} \delta_{(i'',Q''),(i''',Q''')} & \text{if } k=k'=k''=k''' \\
    &\text{ or } k=k' \neq k'' = k''', \\
    \frac{1}{N-1} \delta_{i=i''}\delta_{i'=i'''}[\delta_{Q=Q''}\delta_{Q'=Q'''} - \delta_{Q=Q' \neq Q''=Q'''} + \delta_{Q=Q''' \neq Q'=Q''}] & \text{if } k = k'' \neq k' = k''', \\ \frac{1}{N-1}\delta_{i=i'''}\delta_{i'=i''}[\delta_{Q=Q'''}\delta_{Q'=Q''} - \delta_{Q=Q'\neq Q''=Q'''} + \delta_{Q=Q'' \neq Q'=Q'''}] & \text{if } k = k''' \neq k' = k'', \\ 0 & \text{otherwise.} \end{cases}
\end{align*}
In particular, we find the conditional distribution
\begin{align*}
    &\E[g_{(k,i,Q),(k',i',Q')}(x) | g_T(x) = u] = -\delta_{(k,i,Q),(k',i',Q')}u \\
    &\Cov[g_{(k,i,Q),(k',i',Q')}(x), g_{(k'',i'',Q''),(k''',i''',Q''')}(x) | g_T(x) = u] \\
    %&= \E[g_{(k,i,Q),(k',i',Q')}(x), g_{(k'',i'',Q''),(k''',i''',Q''')}(x)] - \frac{1}{N-1}\delta_{(k,i,Q),(k',i',Q')}\delta_{(k'',i'',Q''),(k''',i''',Q''')} \\
    &= \begin{cases} \frac{1}{N-1} \delta_{i=i''}\delta_{i'=i'''}[\delta_{Q=Q''}\delta_{Q'=Q'''} - \delta_{Q=Q' \neq Q''=Q'''} + \delta_{Q=Q''' \neq Q'=Q''}] & \text{if } k = k'' \neq k' = k''', \\ \frac{1}{N-1}\delta_{i=i'''}\delta_{i'=i''}[\delta_{Q=Q'''}\delta_{Q'=Q''} - \delta_{Q=Q'\neq Q''=Q'''} + \delta_{Q=Q'' \neq Q'=Q'''}] & \text{if } k = k''' \neq k' = k'', \\ 0 & \text{otherwise.} \end{cases}
\end{align*}
In other words, conditionally on $\{g_T(x) = u\}$, the random matrix $(g_{(k,i,Q),(k',i',Q')}(x))_{(k,i,Q),(k',i',Q') \in \mc{G}}$ has the same distribution as
\[
    \textup{tBHGOE}(d-1,p) - u\Id.
\]
(recall Definition \ref{def:tBHGOE}).

\end{lemma}
\begin{proof}
Since the process is isotropic, it suffices to take $x = \mathbf{n}$.
Then we consider
\[
    \overline{g} = g_T \circ \Psi^{-1},
\]
which is a centered, real-valued Gaussian process on $\Psi(U)$ whose covariance is somewhat involved: We write it as
\[
    C_T(x,y) = \frac{1}{2p(d-1)} \mathfrak{R} \left( \prod_{i=1}^N \ip{\Psi^{-1} x^{(i)},\Psi^{-1} y^{(i)}} \right),
\]
where we compute
\begin{multline}
\label{eqn:Psi^{-1}x^1}
    \langle \Psi^{-1} x^{(1)}, \Psi^{-1} y^{(1)}\rangle =\sqrt{1-(x^{(1)}_{1,I})^2-\sum_{j=2}^d (x^{(1)}_{j,R})^2+(x^{(1)}_{j,I})^2}\sqrt{1-(y^{(1)}_{1,I})^2-\sum_{j=2}^d (y^{(1)}_{j,R})^2+(y^{(1)}_{j,I})^2}\\
    +x_{1,I}^{(1)}y_{1,I}^{(1)}+\ii\left(x_{1,I}^{(1)}\sqrt{1-(y^{(1)}_{1,I})^2-\sum_{j=2}^d (y^{(1)}_{j,R})^2+(y^{(1)}_{j,I})^2}-y_{1,I}^{(1)}\sqrt{1-(x^{(1)}_{1,I})^2-\sum_{j=2}^d (x^{(1)}_{j,R})^2+(x^{(1)}_{j,I})^2}\right) \\
    +\sum_{j=2}^d(x_{j,R}^{(1)}y_{j,R}^{(1)}+x_{j,I}^{(1)}y_{j,I}^{(1)})+\ii\sum_{j=2}^d(x_{j,I}^{(1)}y_{j,R}^{(1)}-x_{j,R}^{(1)}y_{j,I}^{(1)})
\end{multline}
and for $i \geq 2$ 
\begin{multline}
\label{eqn:Psi^{-1}x^i}
    \langle \Psi^{-1} x^{(i)}, \Psi^{-1} y^{(i)} \rangle = \sqrt{1-\sum_{j=2}^d (x^{(i)}_{j,R})^2+(x^{(i)}_{j,I})^2}\sqrt{1-\sum_{j=2}^d (y^{(i)}_{j,R})^2+(y^{(i)}_{j,I})^2} \\ +\sum_{j=2}^d(x_{j,R}^{(i)}y_{j,R}^{(i)}+x_{j,I}^{(i)}y_{j,I}^{(i)})+\ii\sum_{j=2}^d(x_{j,I}^{(i)}y_{j,R}^{(i)}-x_{j,R}^{(i)}y_{j,I}^{(i)}).
\end{multline}
As in the real case, we now compute formulas for the covariance structure of $\overline{g}(0)$, $\overline{g}_{(k,i,Q)}(0)$, and $\overline{g}_{(k,i,Q),(k',i',Q')}(0)$ by taking derivatives of this formula, and the Riemannian derivatives of $g_T$ agree with the Euclidean derivatives of $\overline{g}$ at $0$. 

To see that the distribution is the same as the $\textup{tBHGOE}$, one should order the coordinates in such a way that the conditional Hessian minus its mean has the form
\begin{equation}
   H_N = \begin{pNiceMatrix}
      H_N^{(R,R)} &  H_N^{(R,I)} &   \Block{2-1}{v^T}\\
      H_N^{(I,R)} & H_N^{(I,I)} \\
      \hline
      \Block{1-2}{v} & &   \Block{1-1}{0} \\
      %& & & &
      
      \CodeAfter
      \tikz \draw (1-|3) -- (4-|3) ;
    \end{pNiceMatrix},
\end{equation}
where $H_N^{(R,R)}$ stores the correlation between real parts and real parts, and so on. The additional row/column store the additional degree of freedom $x^{(1)}_{1,I}$ which has no corresponding $x^{(1)}_{1,R}$ (so that the block corresponding to $x^{(1)}$ is split, with the additional degree of freedom appearing in a different part of the matrix than the usual degrees of freedom).
\end{proof}

The analogue of Lemma \ref{lem:real-Kac--Rice-mollified}, whose proof also follows directly from \cite[Theorem 11.2.3]{adler2007random}, and with transparent notations, is the following.

\begin{lemma}
\label{lem:complex-Kac--Rice-mollified}
Let $(U,\phi)$ be a chart for $\mathbb{S}_\C^{d-1} \times (\mathbb{U}^{d-1})^{\times (p-1)}$, and write $S = \phi(U) \subset \R^N$. Let $g : \mathbb{S}_\C^{d-1} \times (\mathbb{U}^{d-1})^{\times (p-1)} \to \R$. Let $D \subset \R$ be open. Assume the following:
\begin{enumerate}
\item $f \circ \phi^{-1}$, $\nabla f \circ \phi^{-1}$, and $\nabla^2 f \circ \phi^{-1}$ are continuous on $\overline{S}$, the closure of $S$.
\item There are no points $s \in \overline{S}$ such that both $(\nabla f \circ \phi^{-1})(s) = 0$ and either $(f \circ \phi^{-1})(s) \in \partial D$ or $\det (\nabla^2 f \circ \phi^{-1})(s) = 0$.
\item There are no points $s \in \partial S$ such that $(\nabla f \circ \phi^{-1})(s) = 0$.
\end{enumerate}
Then we have 
\[
    \Crt_g(D,U) = \lim_{\epsilon \to 0} \int_U \diff\mathrm{Vol}(x) \delta_\epsilon(\nabla g(x)) \mathbbm{1}_{g(x) \in D} \abs{\det \nabla^2g(x)}.
\]
\end{lemma}

\begin{lemma}
\label{lem:tBHGOE-singular}
Let $W \sim \textup{tBHGOE}(d,p)$. For each $u \in \R$, we have $\P(\det(W - u\Id) = 0) = 0$.
\end{lemma}
\begin{proof}
As in the real case, we consider the space of all underlying independent entries in the matrix -- that is, in terms of the notation of \eqref{eqn:tBHGOE}, the space of upper-triangular entries of $B$ which are not structurally zero, upper-triangular entries of $C$ which are not structurally zero, and entries of $\theta$ which are not structurally zero, which can be thought of as some $\R^k$. The same argument as in the real case shows that it suffices to exhibit a single point in this space at which the determinant of the corresponding $W-u$ is nonzero. One such example is to take all the off-diagonal blocks of $B$ and $C$ to be the identity $\Id_{d \times d}$, and to take the two halves of $\theta$ to be $(0_d, J_d, 2J_d, \ldots, (p-1)J_d)$, where $J_d$ is the length-$d$ vector of all ones, because the corresponding matrix -- given pictorially in the case $p = 3$ below -- has linearly independent columns for any $u$ value.
\begin{equation}
   \begin{pNiceArray}{ccc|ccc|c}
        -u\Id & \Id & \Id & 0 & \Id & \Id & 0 \\
        \Id & -u\Id & \Id & \Id & 0 & \Id & J \\
        \Id & \Id & -u\Id & \Id & \Id & 0 & 2J \\ \hline
        0 & \Id & \Id & -u\Id & -\Id & -\Id & 0 \\
        \Id & 0 & \Id & -\Id & -u\Id & -\Id & J \\
        \Id & \Id & 0 & -\Id & -\Id & -u\Id & 2J \\ \hline
        0 & J & 2J & 0 & J & 2J & 0
    \end{pNiceArray},
\end{equation}
\end{proof}

The analogue of Lemma \ref{lem:real-Kac--Rice-almost-surely} is the following. 
\begin{lemma}
\label{lem:complex-Kac--Rice-almost-surely}
For any chart $(U,\phi)$ on $(\mathbb{S}_\C^{d-1}) \times (\mathbb{U}^{d-1})^{\times (p-1)}$, the assumptions of Lemma \ref{lem:complex-Kac--Rice-mollified} hold with probability one for $g_T$.
\end{lemma}
\begin{proof}
This is essentially a repetition of the proof of Lemma \ref{lem:real-Kac--Rice-almost-surely}; the only difference is that we need a new proof that the conditional Hessian is almost surely nonsingular, but for this we replace Lemma \ref{lem:BHGOE-singular} with Lemma \ref{lem:tBHGOE-singular}.
\end{proof}

\begin{remark}
\label{rem:correct-quotient}
We observe that Lemma \ref{lem:complex-Kac--Rice-almost-surely} is only correct because we restricted to the correct submanifold of $(\mathbb{S}^{d-1}_\C)^p$ (as previously mentioned, on the full manifold, the set of maximizers is a submanifold rather than a discrete point set, so the number of critical points would be infinite and the result would be wrong). As a consistency check, one can also see this informally in our proof: Suppose, for example, that we were working on $(\mathbb{S}^{d-1}_\C)^{\times 2} \times (\mathbb{U}^{d-1})^{\times (p-2)}$ rather than $(\mathbb{S}^{d-1}_\C) \times (\mathbb{U}^{d-1})^{\times (p-1)}$. Define the obvious analogue of the map $\Psi$ from \eqref{eqn:def_psi_complex_KR}, and consider the analogue of Lemma \ref{lem:complex-covariance-computation}; now $\ip{\Psi^{-1} x^{(i)}, \Psi^{-1} x^{(i)}}$ ``looks like'' \eqref{eqn:Psi^{-1}x^1} for $i = 1, 2$ and ``looks like'' \eqref{eqn:Psi^{-1}x^i} for $i \geq 3$, instead of for $i = 1$ and $i \geq 2$, respectively, previously. By taking derivatives, one can check, for example, that $\Corr(\bar{g}_{(1,1,I),(3,19,R)}(0),\bar{g}_{(2,1,I),(3,19,R)}(0)) = \frac1{2p(d-1)}$ (coming essentially from two imaginary parts that multiply to be real, namely the terms looking like $\ii x^{(1)}_{1,I} \sqrt{\cdots}$ and $-\ii y^{(2)}_{1,I} \sqrt{\cdots}$; notice this is only possible if $\ip{\Psi^{-1} x^{(i)}, \Psi^{-1} y^{(i)}}$ ``looks like'' \eqref{eqn:Psi^{-1}x^1} for at least two $i$ values). Notice that the corresponding conditional Hessian looks like $(\begin{smallmatrix} B & C \\ C & -B \end{smallmatrix})$, except that there are now two additional rows and columns instead of one. The argument we just gave shows that, with probability one, the extraneous two rows are identical, and thus the determinant vanishes with probability one.
\end{remark}

We omit the proof of the following lemma, since it is almost exactly the same proof as that of the real analogue, Lemma \ref{lem:real-Kac--Rice-before-rmt}.

\begin{lemma}
\label{lem:complex-Kac--Rice-before-rmt}
We have 
\[
    \E[\Crt_{g_T}(D)] \leq \text{Vol}(\mathbb{S}_\C^{d-1} \times (\mathbb{U}^{d-1})^{\times (p-1)}) \E[ |\det \nabla^2 g_T(\mathbf{n}) | \mathbbm{1}_{g_T(\mathbf{n}) \in D}] \phi_{\nabla g_T(\mathbf{n})}(0)
\]
\end{lemma}

\begin{proof}[Proof of Lemma \ref{lem:K-R_Complex}]
This is basically analogous to the proof of Lemma \ref{lem:exact_kac_rice}; we use Lemma \ref{lem:complex-covariance-computation} to show 
\[
    \phi_{\nabla g_T(\mathbf{n})}(0) = \frac{1}{(2\pi)^{N/2} \left(\frac{1}{N-1}\right)^{N/2}},
\]
compute $\text{Vol}(\mathbb{S}_\C^{d-1} \times (\mathbb{U}^{d-1})^{\times (p-1)}) = \frac{2\pi^d}{\Gamma(d)} \cdot \left( \frac{2\pi^{d-\frac{1}{2}}}{\Gamma(\frac{2d-1}{2})} \right)^{p-1}$, and condition on $g_T(\mathbf{n})$.
\end{proof}

\bibliographystyle{alpha}
\bibliography{inj_norm_biblio_v1}

\end{document}